\documentclass{IEEEtran}
\usepackage{amssymb,amsmath,amsthm,graphicx,subfigure}
\usepackage{algpseudocode,cases,booktabs}
\usepackage{epstopdf,siunitx,comment,float,xypic,url}

\xyoption{graph}
\xyoption{frame}

\newcommand{\norm}[1]{\ensuremath{\left\| #1 \right\|}}
\newcommand{\bracket}[1]{\ensuremath{\left[ #1 \right]}}

\newcommand{\refeqn}[1]{(\ref{eqn:#1})}
\newcommand{\reffig}[1]{Figure \ref{fig:#1}}
\newcommand{\tr}[1]{\mathrm{tr}\ensuremath{\negthickspace\bracket{#1}}}

\newcommand{\SO}{\ensuremath{\mathsf{SO(3)}}}
\newcommand{\T}{\ensuremath{\mathsf{T}}}

\newcommand{\so}{\ensuremath{\mathfrak{so}(3)}}
\newcommand{\SE}{\ensuremath{\mathsf{SE(3)}}}

\renewcommand{\Re}{\ensuremath{\mathbb{R}}}
\newcommand{\Ree}{\ensuremath{\mathbb{R}}}

\renewcommand{\d}{\ensuremath{\mathfrak{d}}}
\newcommand{\Sph}{\ensuremath{\mathsf{S}}}

\newcommand{\g}{\ensuremath{\mathfrak{g}}}

\newcommand{\Lya}{\text{Lyapunov}}
\newcommand{\diag}[1]{\mbox{diag}\ensuremath{\negthickspace\bracket{#1}}}

\date{}

\newtheorem{prop}{Proposition}[section]

\title{Geometric Adaptive Control with Neural Networks\\ for a Quadrotor UAV in Wind fields}
\author{Mahdis Bisheban, and Taeyoung Lee%
	\thanks{M. Bisheban, Mechanical and Aerospace Engineering, George Washington University, Washington DC 20052 {\tt mbshbn@gwu.edu}}
	\thanks{T. Lee, Mechanical and Aerospace Engineering, George Washington University, Washington DC 20052 {\tt tylee@gwu.edu}}
	\thanks{This research has been supported in part by NSF under the grant, CNS-1837382. }
}

\begin{document}
	
\maketitle
	
\begin{abstract}
    This paper proposes a geometric adaptive controller for a quadrotor unmanned aerial vehicle with artificial neural networks. 
    It is assumed that the dynamics of a quadrotor is disturbed by arbitrary, unstructured forces and moments caused by wind. 
    To address this, the proposed control system is augmented with multilayer neural networks, and the weights of neural networks are adjusted online according to an adaptive law. 
    By utilizing the universal approximation theorem, it is shown that the effects of unknown disturbances can be mitigated. 
    More specifically, under the proposed control system, the tracking errors in the position and the heading direction are uniformly ultimately bounded where the ultimate bound can be reduced arbitrarily. 
    These are developed directly on the special Euclidean group to avoid complexities or singularities inherent to local parameterizations. 
    The efficacy of the proposed control system is first illustrated by numerical examples. 
    Then, several indoor flight experiments are presented to demonstrate that the proposed controller successfully rejects the effects of wind disturbances even for aggressive, agile maneuvers. 

\end{abstract}

\section{Introduction}







Multirotor unmanned aerial vehicles are subject to various disturbance forces and moments. 
In particular, wind disturbances may severely degrade the performance and stability of small  aerial vehicles. 
Thus it is critical to carefully characterize these effects and to alleviate them for reliable autonomous flights in various outdoor environments. 
To address this issue, several approaches have been considered for comprehensive aerodynamic modeling of wind effects, system identification of wind effect  modeling parameters, and feedback control systems to mitigate the wind effects. 

With regard to the wind effects modeling, the thrust and the drag forces for forward flights are studied in~\cite{Gill2017}, and it is shown that the assumptions for hovering flight models become deteriorated when the relative wind speed is greater than $4$ to $7\si{\meter \per \second}$.
In~\cite{Craig_Paley_16}, the blade-flapping response of a small-stiff propeller in wind is studied with a rotor--pendulum system. 
Once a mathematical model for wind effects is determined, the modeling parameters should be identified via experiments with a particular unmanned aerial vehicle under consideration. 
To determine the unknown aerodynamic modeling parameters, \cite{Bisheban_Lee_IJCAS_17,Bisheban_Lee_CCTA_17} present computational geometric approaches for system identification of the quadrotor dynamics, 
where the system identification problem is converted into an optimization problem to minimize the discrepancy between the identified model and the actual response. 


To reject the undesired effects of wind disturbances, control systems are proposed to cancel out the wind effects from the above mathematical models.
In~\cite{Tran15}, a look-up table is used to estimate wind forces and moments in real-time based on relative wind speed and the rotational speed of propellers.
The table is generated by solving computational intensive aerodynamic expressions.
Reference~\cite{Bangura_Mahony_17} presents the dynamics of a brushless DC motor that is constructed to determine the power level to follow a given desired trajectory while rejecting axial wind effects.  
In~\cite{craig2019geometric}, wind velocity data from flow probes is utilize in a control system to guarantee stability in the presence of winds. 
While these cancellation techniques have been successful, the robustness and performance are limited by the accuracy of the wind effect model used in the controller, and the estimated wind velocity. 
The control force and moment resisting wind would be reliable within the flight envelop considered for the aerodynamic modeling, which is additionally limited by computing resources available in real-time. 
Further, they may deteriorate for unexpected wind gusts as there is no mechanism to adjust the modeling errors online. 

On the other hand, several alternative control techniques have been presented to reduce the undesired dependency on wind effect modeling accuracy or wind measurement errors. 
For example, \cite{Goodarzi_Lee_13} presents a geometric proportional-integral-derivative  controller on the special Euclidean group to reject unknown, fixed uncertainties. 
Also, parametric uncertainties are addressed with a geometric adaptive control scheme in~\cite{GooLeeAJDSMC15}. 
In~\cite{Mellinger_Kumar_2014}, to overcome the effects of modeling errors, data of successive indoor experimental trials are used to tune control parameters for aggressive maneuvers.
In~\cite{Nicol_Ramirez-Serrano_08}, an adaptive neural network is used for the reduced dynamics of a quadrotor in the altitudes and the attitudes. 



This paper proposes a geometric adaptive control scheme for a quadrotor unmanned aerial vehicle, where the effects of wind are considered as unstructured, unknown disturbances. 
Instead of counterbalancing those with an aerodynamic model and a measured wind velocity, wind disturbances are compensated by artificial neural network whose weighing parameters are adjusted online.
More specifically, we adopt geometric controller proposed in~\cite{LeeLeoPICDC10}, and augment it with multi-layer neural networks and an adaptive law to mitigate unknown disturbance forces and moments that are considered as an arbitrary function of quadrotor states.
The dynamics of a quadrotor is globally formulated on the special Euclidean group to avoid singularities and complexities inherent to Euler angles or quaternions. 
It is shown that the tracking errors are uniformly ultimately bounded with an ultimate bound that can be reduced arbitrarily up to any desired precision. 
These are illustrated by numerical examples with simulated aerodynamic effects of wind. 
Next, we show that the proposed geometric adaptive controller is able to mitigate wind effects even for aggressive maneuvers through indoor flight experiments with artificial wind gusts generated by an industrial fan.  

The preliminary results are presented in~\cite{Bisheban_Lee_CDC18}. 
However, this paper presents the complete Lyapunov stability proof, extensive numerical examples, and results of flight experiments that are not available in~\cite{Bisheban_Lee_CDC18}.

In short, the main contribution of this paper is presenting a geometric neural network based adaptive controller for a quadrotor that is capable of compensating unknown aerodynamic forces and moments caused by wind. 
This requires neither a precise mathematical model of wind effects nor the actual wind velocity, and it can be implemented without additional onboard anemometer. 
Furthermore, autonomous agile maneuvers under strong wind have not been presented in literature.




\section{Problem Formulation}\label{sec:QDM}

\subsection{Quadrotor Dynamics with Disturbances}

This section formulates the quadrotor dynamics including unknown disturbances in the translational dynamics and the rotational dynamics. 
As they are considered as arbitrary disturbing forces and moments, they may represent the wind disturbance effects as discussed later in Section IV. 
The quadrotor UAV is regarded as a rigid body whose configuration is represented by the position of the center of mass $x \in \Ree^3$ in the inertial frame, and the orientation of the body-fixed frame with respect to the inertial frame $R\in\SO=\{R\in\Ree^{3\times 3}\,|\, R^TR=I_{3\times 3},\mathrm{det}[R]=+1\}$.
Thus the configuration space of a quadrotor is the special Euclidean group $\SE$, which is the semi-direct product of $\SO$ and $\Re^3$. 


The equations of motion are given by
\begin{gather}
\dot{x}=v,\label{eqn:EC1}\\
m\dot{v}=U_{e},\label{eqn:EC2}\\
\dot R = R\hat\Omega,\label{eqn:EC3}\\
J\dot \Omega + \Omega\times J\Omega = M_e,\label{eqn:EC4}
\end{gather}
where $U_e,M_e\in\Ree^3$ are the resultant force resolved in the inertial frame and the resultant moment resolved in the body-fixed frame. 
The mass and the inertia matrix are denoted by $m\in\Ree$, and $J\in\Ree^3$, respectively. 
The vector $v\in\Ree^3$ is the linear velocity in the inertial frame, and $\Omega\in\Ree^3$ is the angular velocity resolved in the body-fixed frame. 
The hat map $\wedge :\Ree^3\rightarrow\so$ is defined such that $\hat x y = x\times y$ and $(\hat x)^T =-\hat x$ for any $x,y\in\Ree^3$. 
The inverse of the hat map is denoted by the vee map $\vee :\so\rightarrow\Ree^3$.


Suppose that $d_h,d_v\in\Ree$ specify the horizontal and vertical distances from the origin of the body-fixed frame to the center of a rotor. 
The location of four rotors in the body-fixed frame are given by
\begin{gather}
r_1=\begin{bmatrix} d_h,0,d_v \end{bmatrix}^T, \;  r_2=\begin{bmatrix} 0,-d_h,d_v \end{bmatrix}^T,\\ 
r_3=\begin{bmatrix} -d_h,0,d_v \end{bmatrix}^T, \; r_4=\begin{bmatrix} 0,d_h,d_v \end{bmatrix}^T.
\end{gather} 
Let the thrust $T_j^\prime\in \Ree$ and torque $Q_j^\prime\in \Ree$ of the $j$-th motor be given by
\begin{gather}
T_j^\prime=C_T^\prime\omega_j^2,\quad
Q_j^\prime=C_Q^\prime\omega_j^2\equiv C_{TQ}T_j^\prime,\label{eqn:T_Q_w}
\end{gather} 
where $C_T^\prime,C_Q^\prime\in \Ree$ are constant thrust and torque coefficients, and $C_{TQ}=\frac{C'_Q}{C'_T}\in \Ree$ determines the relation between reactive torque and thrust. 
The resultant force and moment acting on a quadrotor can be written as
\begin{gather}
U_e^\prime=mge_3-fRe_3-\Delta_1,\label{eqn:E3}\\
M_e^\prime=-\Sigma_{j=1}^4 r_j \times T_j^\prime e_3-(-1)^{j+1}Q_j^\prime e_3-\Delta_2,\label{eqn:E4}
\end{gather}
where $f=\Sigma_{j=1}^4T_j\in\Ree$ is the sum of the four rotor thrusts, and $mge_3$ is the gravitational force with $e_3=[0,0,1]\in\Re^3$. 
Unknown disturbance force and moment are denoted by $\Delta_1$ and $\Delta_2\in\Ree^3$ respectively.

\begin{table}\centering
	\caption{Summary of notations}\label{tab:notations}
	\begin{tabular}{cc}\hline
		Notation& Refers to  
		\\\hline
		$\hat{}$  & hat map\\
		$\vee$ & vee map \\
		$\bar{}$ & estimated value \\
		$\tilde{}$ & estimation error value  \\
		$\dot{}$ & time derivative  \\
		$\prime{}$ & alternative value  \\
		$\times$ & cross product  \\
		$\norm{}$ & Frobenius norm of a matrix, and 2-norm of a vector \\ 
		$\lambda_m{()}$ & minimum eigen value of a matrix \\	
		$\lambda_M{()}$ & maximum eigen value of a matrix \\		\hline
	\end{tabular}
\end{table}

\subsection{Position Tracking Control Problem}

Suppose that the desired position trajectory is given as a smooth function of time, i.e., $x_d(t)\in\Re^3$. 
It is considered that $x_d(t)$ and all of its time-derivatives are bounded. 
We wish to design a control system for the rotor thrusts such that the actual position trajectory asymptotically follows the desired value in the presence of the unknown disturbance. 
Instead of designing the rotor thrusts, the control input is considered as the total thrust $f$, and the control moment $M=[M_1,M_2,M_3]^T\in\Re^3$ in the body-fixed. 
For a given $(f,M)$, the equivalent thrust at each rotor can be computed by
\begin{align}
\begin{bmatrix} T_1^\prime\\T_2^\prime\\T_3^\prime\\T_4^\prime\end{bmatrix}=\begin{bmatrix}1&1&1&1\\ 0&-d_h&0&-d_h\\d_h&0&-d_h&0\\-C_{TQ}&C_{TQ}&-C_{TQ}&C_{TQ}  \end{bmatrix}^{-1} \begin{bmatrix} f\\M_1\\M_2\\M_3\end{bmatrix}.\label{eqn:compute_all_thrusts}
\end{align}

\section{Geometric Adaptive Controller With Neural Networks}

In this section, we present a geometric adaptive control system for a quadrotor to reject the effects of unknown disturbances without any prior knowledge. 

\subsection{Controller Structure}


The presented quadrotor dynamics is underactuated as there are four control inputs. 
In~\cite{LeeLeoPICDC10}, a geometric control system for a quadrotor is presented with a backtepping approach, which is adopted in this paper.
The overall controller structure is summarized as follows. 
Let the tracking errors in the position and the velocity be 
\begin{gather}
e_x=x-x_d,\quad 
e_v=v-\dot{x}_d.\label{eqn:ev}
\end{gather}
For positive controller gain $k_x,k_v$, consider an ideal control force  $A\in\Re^3$ defined as
\begin{align}
A=&\bar \Delta_1-k_x e_x-k_v e_v -mge_3+m \ddot x_d,\label{eqn:A}
\end{align}
where $\bar \Delta_1 \in\Re^3$ is an adaptive control term to mitigated the effects of the disturbance $\Delta_1$. 
It is straightforward to show that the control objective will be achieved if the control force term $-fRe_3$ in \eqref{eqn:E3} is replaced by the above ideal value. 
However, that is not achievable as the total control thrust is always opposite to the third body fixed axis, i.e., the direction of the total thrust is always $-R e_3$, and only its magnitude $f$ can be adjusted arbitrarily.

To address this, an attitude controller is introduced such that the actual attitude is guided toward to the ideal thrust direction defined by \eqref{eqn:A}.
More specifically, the desired direction for the third body-fixed axis is given by
\begin{align}
    b_{3c}=-\frac{A}{||A||}.
\end{align}
As it is a two-dimensional unit vector, the desired heading direction, namely $b_{1_d}(t)\in\Sph^2=\{q\in\Re^3\,|\,\|q\|=1\}$ is further introduced as a function of time.
These yield the complete desired attitude as 
\begin{align}
    R_c = [b_{1_c}, b_{2_c}, b_{3_c}],\label{eqn:Rc}
\end{align}
where
\begin{align*}
    b_{1_c} &= b_{2_c}\times b_{3_c},\\
    b_{2_c} &= -\frac{b_{1_d}\times b_{3_c}}{\| b_{1_d}\times b_{3_c}\|}.
\end{align*}
One can show the above construction guarantees $R_c\in\SO$, and by taking its time-derivative, the desired angular velocity also can be constructed as
\begin{align}
	\Omega_c=(R_c^T\dot{R}_c)^\vee.\label{eqn:omegac}
\end{align}

Any attitude tracking control system can be implemented to asymptotically follow $R_c$, and the total thrust is chosen as the ideal control force projected to the current thrust direction as follows.
\begin{align}
f=&-A^T Re_3,\label{eqn:f}\\
M_c=&\bar \Delta_2-k_R e_R-k_\Omega e_\Omega+\Omega \times J \Omega\nonumber\\
&-J(\hat \Omega R^T R_c \Omega_c-R^T R_c \dot{\Omega}_c),\label{eqn:Mc}
\end{align}
where $k_R,k_\Omega$ are positive attitude control gains, and the tracking errors for the attitude and the angular velocity are given by
\begin{gather}
e_R=\frac{1}{2}(R_c^TR-R^TR_c)^\vee,\quad 
e_\Omega=\Omega-R^TR_c\Omega_c,\label{eqn:eomega}
\end{gather}
Also, $\bar\Delta_2\in\Re^3$ denotes an adaptive term to eliminate the effects of the unknown disturbance $\Delta_2$. 

In the absence of the disturbances and the adaptive control terms, local exponential stability has been established in~\cite{LeeLeoPICDC10}.
Next, we will formulate the expression for the adaptive terms and the adaptive control laws to address the unknown disturbances. 
Here we assume
\begin{gather}
\norm{-mge_3+m \ddot x_d+\bar \Delta_1}\leq B_1\label{eqn:acceleration_bound},
\end{gather}
for a given positive constant $B_1$. 


\subsection{Adaptive Neural Network Structure}\label{Sec:ANN}
\begin{figure}
\def\OutNeuron#1{\Neuron @+
  \Output{#1}}
\def\InNeuron#1{\Connect @+
  \Input {#1}}
\def\Hidden{\Neuron
 ;@@{\connect\dir{-}}\POS p}
\def\HiddenL{\Connect
 ;@@{\connect\dir{-}}\POS p}
\def\Neuron{\POS*+[o]{},*\frm{o}}
\def\Connect{\POS*[o][F]{}}
\def\Output#1{\ar +(1.0,0)*[r]+(0.2,0){#1}}
\def\Input#1{\ar @{<-}-(1.0,0)*[l]+(0.2,0){#1}}
\def\Extend{\save\drop!C\txt{\bf\vdots}\restore}
$$\xy\xygraph{!{<3.5pc,0pc>:<0pc,1.5pc>::@(}
     [] !{+(0,0.5)}!{\InNeuron{1}}
    [d] !{\InNeuron{x^\circ_{nn_1}}}
    [d] !{\InNeuron{x^\circ_{nn_2}}}
    [d] !{\Extend}
    [d] !{\Extend}
    [d] !{\InNeuron{x^\circ_{nn_{N_1}}}}
[r(1.2)uuuu] !{+(0,0.5)}!{\Hidden}!{\ar +(0.8,0)^{\varsigma_1}}
    [d] !{\Hidden}!{\ar +(0.8,0)^{\varsigma_2}}
    [d] !{\Extend}
    [d] !{\Hidden}!{\ar +(0.8,0)^{\varsigma_{N_2-1}}}
    [d] !{\Hidden}!{\ar +(0.8,0)^{\varsigma_{N_2}}}
    !{@i@)}
    !{@(}
[rruuu] !{+(0,0.5)}!{\OutNeuron{y_{nn_1}}}
    [d] !{\OutNeuron{y_{nn_2}}}
    [d] !{\Extend}
    [d] !{\OutNeuron{y_{nn_{N_3}}}}
[l(1.2)uuuu] !{+(0,0.5)}!{\HiddenL \Input{1}}
    [d] !{\HiddenL}
    [d] !{\HiddenL}
    [d] !{\Extend}
    [d] !{\HiddenL}
    [d] !{\HiddenL}
!{@i@)}
}
\endxy $$
\caption{Structures of  3-layer neural networks}\label{fig:nnstruct}
\end{figure}

\begin{figure*}
	\centering
	\includegraphics[width=0.7\textwidth]{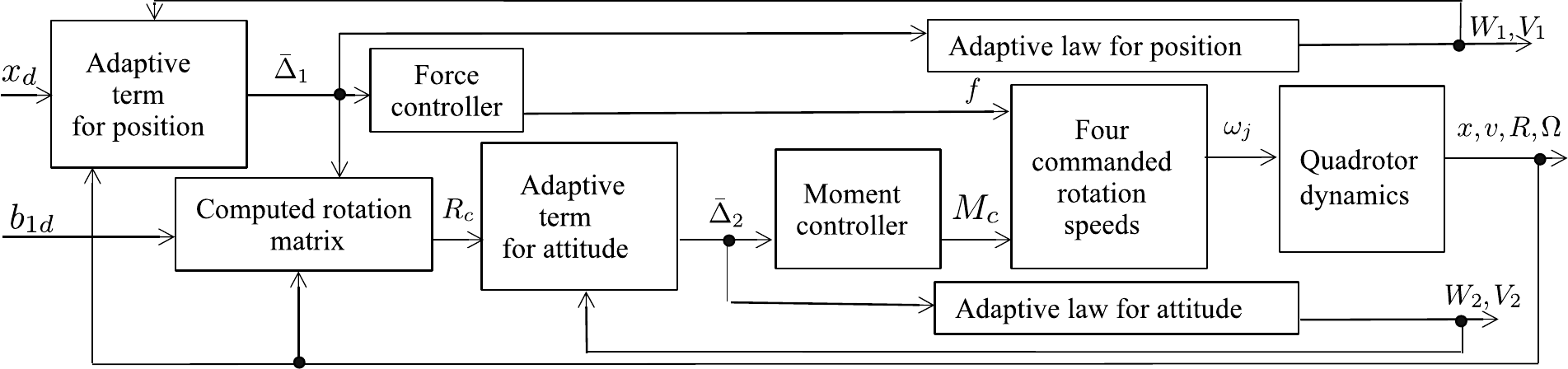}
	\caption{Adaptive controller structure (The adaptive term for the position and the attitude dynamics are given by \refeqn{delta_est}, the force controller is given by \refeqn{f}, the computed rotation matrix is given by \refeqn{Rc}, the moment controller is given by \refeqn{Mc}, the adaptive law is given by \refeqn{Wdot}--\refeqn{Vdot}, and \refeqn{W_bounded}--\refeqn{V_bounded}, four commanded rotation speeds are given by \refeqn{compute_all_thrusts}, \refeqn{T_Q_w}, and quadrotor dynamics are given by \refeqn{EC1}--\refeqn{EC3}, \refeqn{EC2}--\refeqn{EC4})}\label{cont_fig}
\end{figure*}


Consider a three-layer artificial neural network as illustrated in Figure~\ref{fig:nnstruct}. 
The number of neurons at the input layer, the hidden layer, and the output layer are denoted by $N_1+1$, $N_2+1$, and $N_3$, respectively. 
The input to the neural network is arranged in a vector form $x_{nn}\in\Re^{N_1+1}$ as
\begin{align*}
    x_{nn}=[1, x^\circ_{nn_1},\ldots x^\circ_{nn_{N_1}} ].
\end{align*}
The input to the hidden layer, namely $z\in\Re^{N_2}$, is a weighted sum of the above, given by
\begin{align*}
    z = V^T x_{nn},
\end{align*}
for a weighting matrix $V\in\Re^{N_1+1\times N_2}$. 
The output $y\in\Re^{N_3}$ of the neural network is 
\begin{align*}
    y = W^T \sigma(z),
\end{align*}
where the weighing of the output layer is denoted by $W\in\Re^{N_2+1\times N_3}$, and the activation function $\sigma:\Re^{N_2}\rightarrow \Re^{N_2+1}$ is defined as
\begin{align*}
    \sigma(z) = [1, \varsigma_1, \ldots , \varsigma_{N_2}],
\end{align*}
for the sigmoid function 
\[
    \varsigma_k = \frac{1}{1+e^{-z_k}},
\]
for $k\in\{1,\ldots N_2\}$. 

We assume that the unknown disturbance force and moment, namely $(\Delta_1,\Delta_2)$ in \eqref{eqn:E3}, \eqref{eqn:E4}, are dependent of the quadrotor state. 
According to the universal approximation theorem~\cite{hornik1989multilayer}, there exist artificial neural networks that approximate these disturbances up to an arbitrary level of accuracy. 

More explicitly, the particular structures of the  artificial neural networks utilized in this paper are defined as follows. 
Throughout the remainder of this paper, the subscript $i=1$ denotes the position dynamics, and $i=2$ denotes the attitude dynamics. 
Let the input to the neural network $x_{nn_i} \in \Ree^{N_{1_i}+1}$ be
\begin{gather}
x_{nn_i}=[1,x_{1_i},x_{2_i}],
\end{gather}
where $x_{1_1}=x$, $x_{2_1}=v$ are for position dynamics, and $x_{1_2}=E(R)^T$, $x_{2_2}=\Omega$, $E(R)=[\theta, \phi, \psi]$ contains the Euler angles from the rotation matrix $R$, are for attitude dynamics. 
Consequently, $N_{1_1}=N_{1_2}=6$.
Since the neural network is formulated to approximate the disturbance force and moment, the number of output is $N_{3_1}=N_{3_2}=3$. 
The universal approximation theorem implies that there exists an ideal value of the weighting parameters $(W_i,V_i)$ and the number of the hidden layer $N_2$ such that
\begin{align*}
    \Delta_i = W_i^T \sigma (V_i^T x_{nn_i}) + \epsilon(x_{nn_i}),
\end{align*}
for the approximation error satisfying $\|\epsilon(x_{nn_i})\| \leq \epsilon_N$ for some $\epsilon_N > 0$.

While the ideal values $(W_i,V_i)$ are not available, it is assumed that upper bounds $W_{M_i},V_{M_i}>0$ are given such that
\begin{gather}
    ||W_i|| \leq W_{M_i}, \quad  ||V_i|| \leq V_{M_i}.
\label{eqn:NNbound}
\end{gather}
Let $( \bar W_i, \bar V_i)$ be the current estimate to the ideal weighting matrices. 
The adaptive control term in \eqref{eqn:A} and \eqref{eqn:Mc} are computed by
\begin{gather}
\bar \Delta_i=\bar W_i^T \sigma(\bar{z}_i),\label{eqn:delta_est}
\end{gather}
with $\bar z_i = \bar V_i^T x_{nn_i}$. 
And they are updated according to the following adaptive law:
\begin{gather}
\dot{\bar{W_i}}= 
\begin{cases}
\dot{\bar{W_i}}^\prime & \parbox[t]{0.25\textwidth}{if $\norm{\bar{W_i}}< W_{M_i} $ or ($\norm{\bar{W_i}}= W_{M_i} , \, \dot{\bar{W_i}}^T \bar{W_i}\leq 0$)}\\
[I_W-\frac{\bar{W_i} \bar{W_i}^T}{\bar{W_i}^T \bar{W_i}}]\dot{\bar{W_i}}^\prime & \text{otherwise},
\end{cases}\label{eqn:W_bounded}\\
\dot{\bar{V_i}} = \begin{cases}
\dot{\bar{V_i}}^\prime & \parbox[t]{0.25\textwidth}{if $\norm{\bar{V_i}}< V_{M_i} $ or ($\norm{\bar{V_i}}= V_{M_i} , \, \dot{\bar{V_i}}^T\bar{V_i}\leq 0$)}\\
[I_V-\frac{\bar V_i \bar V_i^T}{\bar V_i^T \bar V_i}]\dot{\bar{V_i}}^\prime & \text{otherwise},
\end{cases}\label{eqn:V_bounded}
\end{gather}
where $I_W \in \Re^{N_{2_i}+1\times N_{2_i}+1}$, $I_V \in \Re^{N_{1_i}+1 \times N_{1_i}+1 }$ are identity matrices, and $\norm{}$ indicates Frobenius norm of a matrix.
These correspond to the projection of the following adaptive law to a bounded region satisfying \eqref{eqn:NNbound}~\cite{IoannouRobustadaptiveComtrol}: 
\begin{gather}
\dot{\bar{W_i}}^\prime = -\gamma_{w_i}[\sigma(z_i)a_i^T- \sigma^\prime(z_i)z_ia_i^T]-\kappa_i \gamma_{w_i} \bar W_i,\label{eqn:Wdot}\\
\dot{\bar{V_i}}^\prime = -\gamma_{v_i}x_{nn_i}[ \sigma^\prime(z_i)^T\bar W_ia_i]^T-\kappa_i \gamma_{v_i} \bar V_i,\label{eqn:Vdot}\\
a_1=e_v+c_1 e_x,\quad a_2=e_\Omega+c_2  e_R,\label{eqn:a1a2}
\end{gather}
for positive adaptive gains and parameters $\gamma_{w_i}, \gamma_{v_i}, \kappa_i, c_1, c_2\in \Ree^{+}$.



The proposed design of the adaptive law is based on the following expression of the estimation error. 
Let the errors in the weighting parameters be denoted by 
\begin{gather}
\tilde{W_i}=W_i-\bar W_i,\quad \tilde{V_i}=V_i-\bar V_i.\label{eqn:errorNN}
\end{gather}
The output error of the neural network be $\tilde{\Delta}_i=\Delta_i-\bar{\Delta}_i$ can be written as
\begin{gather}
\tilde{\Delta}_i=\tilde W_i^T [\sigma(\bar{z}_i)-\sigma^\prime(\bar{z}_i)\bar{z}_i]+\bar W_i^T \sigma^\prime(\bar{z}_i)\tilde{z}_i-w_i,\label{eqn:OENN}\\
w_i=-\tilde W_i \sigma^\prime(\bar{z}_i) z_i-W_i^T\mathcal{O}(\tilde z_i)-\varepsilon(x_{nn_i}),\label{eqn:wi}\\
\mathcal{O}(\tilde z_i) =\sigma(z_i)-\sigma(\bar{z}_i)-\sigma^\prime(\bar{z}_i) \tilde{z}_i,\label{eqn:O}
\end{gather}
where $\tilde{z}_i= \tilde{V}_i^T x_{nn_i}$.
Further, it can be shown that $w_i$ is bounded by
\begin{gather}
\norm{w_i}\leq C_{1_i}+\norm{\tilde{Z}_i}(C_{2_i}+C_{3_i} \norm{x_{1_i}}+C_{4_i} \norm{x_{2_i}}),
\end{gather}
where $C_{k_i}, k\in{1,\ldots,4}$ are positive constants, and $\tilde Z_i=\mathrm{diag}[\tilde W_i,\tilde V_i]\in \Ree^{N_{2_i}+N_{1_i}+2, N_{2_i}+N_{3_i}}$~\cite{LeeKimJGCD01}.

The resulting stability properties of the proposed control system are summarized as follows.

\begin{prop}{}\label{prop:NN_prop}
Consider the control force $f$ and moment $M_c$ defined at \refeqn{f}, \refeqn{Mc}. Suppose that the initial condition satisfies
	\begin{align}
	\Psi(R(0),R_d(0))\leq \psi_1 < 1,\quad 
	\norm{e_x(0)}<e_{x_{max}},
	\end{align}
	for fixed constants $\psi_1$ and $e_{x_{max}}$. 
	There exist the values of the controller parameters such that all of the tracking errors of the quadrotor UAV, as well as the neural network weight errors are uniformly ultimately bounded. 
\end{prop}
%

\begin{proof}
	See Appendix.
\end{proof}

This theorem implies that arbitrary disturbance forces and moments can be mitigated by adaptive neural networks that are adjusted online to cancel out the disturbances. 
This does not achieve stability in the sense of Lyapunov or attractivity, as the universal approximation theorem implies approximation up to a small bounded error. 
However, the ultimate bound of the tracking errors can be adjusted by increasing the controller gains according to~\refeqn{D}.
As such, there should be a proper trade-off between the size of the ultimate bound and the magnitude of the rotor thrust. 
Compared with the conventional adaptive control, it is not required that the uncertain term follows the form of linear regression. 
As such, the proposed adaptive control scheme can deal with a large class of unstructured uncertainties. 
In contrast to nonlinear robust controls, such as presented in~\cite{LeeLeoAJC13}, there is no issue chattering in control inputs.  

\section{Numerical Example}

The efficacy of the proposed control system is illustrated by a numerical example. 
In particular, we consider a scenario where the quadrotor is flying under wind gusts. 
To simulate the effects of wind disturbances, we first present an aerodynamic model of a quadrotor, inspired by the literature in the helicopter rotor dynamics.

\subsection{Quadrotor Dynamics under Wind Disturbance}\label{Sec:Wind_effects}

Suppose that the wind vector presented in the inertial frame is denoted by $v_w\in\Ree^3$.
The relative wind on the $j$-th rotor in the body-fixed frame is denoted by $v_{w_j}=[u_{1_j},u_{2_j},u_{3_j}]^T$. 
It is caused by the wind vector and the quadrotor translational and rotational velocities, as follows
\begin{equation}
v_{w_j}=R^T (v_w-v)+ \hat\Omega r_j.\label{eqn:vrel}
\end{equation}

The external resultant force acting on the quadrotor is given by
\begin{gather}
U_e=mge_3-C_d ||v-v_{w}|| (v-v_{w})+R \Sigma_{j=1}^4 T_j d_j,\label{eqn:Uewind}
\end{gather}
where the second term on the right hand side represents the drag force acting on the center of mass, and $C_d\in\Ree$ is the drag coefficient. 

The variable $T_j$ represents the thrust for the $j$-th rotor, given by
\begin{gather}
T_j=C_{T_j}\rho A_p (r_p\omega_j)^2,
\end{gather} 
where $\rho\in\Ree$ is the air density and the rotor sweeping area is given by $A_p=(\pi r_p)^2$ for the radius $r_p$.
The rotating speed is shown by $\omega_j$. 
The parameter $C_{T_j},\rho\in\Ree$ represents the thrust coefficient, and it follows the following expression that models the effects of induced velocity~\cite{Padfield2007}: 
\begin{align}
C_{T_j}&=\frac{s C_{l\alpha}}{2}[\theta_0(\frac{1}{3}+\frac{\mu_{x_j}^2}{2})-\frac{1}{2}(\lambda_j+\mu_{z_j})],\label{eqn:CT}\\
\lambda_j&=\frac{C_{T_j}}{2\sqrt{\mu_{x_j}^2+(\lambda_j+\mu_{z_j})^2}},\\
\mu_{x_j}&=\frac{\sqrt{u_{1_j}^2+u_{2_j}^2}}{\omega_j r_p},\\
\mu_{z_j}&=\frac{u_{3_j}}{\omega_j r_p}.
\end{align}
where $\lambda_j\in \Ree$ is the inflow ratio, which is the induced air velocity over by the tip speed, and
$s=\frac{N_b c}{\pi r_p}\in\Ree$ is the solidity ratio which is the approximated blade area over the blade sweeping area. 
Next, $c,N_b$ represents the blade chord, and the number of blades for one rotor respectively. 
The blade lift curve slope and blade pitch angle are shown by $C_{l\alpha},\theta_0\in\Ree$. 
Also, $\mu_{z_j},\mu_{x_j}$ are the perpendicular and parallel
advance ratios to the rotor plane.
As described above, $C_{T_j}$ is defined implicitly. 
Therefore, Newton's iterative is used in the numerical simulation to obtain the thrust coefficient and the inflow ratio.

Next, in~\eqref{eqn:Uewind}, the direction of rotor thrust in the body-fixed frame is denoted by the unit-vector $d_j \in \Sph^2$, and it is computed by
\begin{gather}
d_{j}=\begin{bmatrix}\frac{-\sin{\alpha_j}}{\sqrt{u^2_{1_j}+u^2_{2_j}}} u_{1_j}, \frac{-\sin{\alpha_j}}{\sqrt{u^2_{1_j}+u^2_{2_j}}} u_{2_j},-\cos{\alpha_j} \end{bmatrix}^T,\label{eqn:d_T}
\end{gather}
where the blade flapping angle of the $j$-th rotor is shown by $\alpha_j\in \Ree$. 
If the first and second elements of relative wind become zero, i.e., $u_{1_j},u_{2_j}=0$, then the $d_j=-e_3$, and so there is no thrust component in the $b_1-b_2$ plane. 
Let, $C_{\alpha}\in\Ree$, be the fixed flapping angle coefficient~\cite{Hoffmann2011b,Sydney2013}. Then, the flapping angle can be approximated with
\begin{gather}
\alpha_j=C_{\alpha} \sqrt{u^2_{1_j}+u^2_{2_j}}. \label{eqn:alpha}
\end{gather}

Finally, let the stiffness of the rotor blade be shown by $K_\beta\in\Ree$, and the blade drag coefficient be shown by $C_{D_0}\in\Ree$.
From~\cite{Padfield2007,Hoffmann2011b}, the external resultant moment can be approximated by 
\begin{align}
M_e=&\Sigma_{j=1}^4 r_j \times T_j d_j+(-1)^{j+1}Q_jd_j\nonumber\\
&+\frac{N_b}{2}K_\beta \alpha_j (d_j\cdot e_1+d_j\cdot e_2),\label{eqn:Mewind}\\
Q_j=&C_{Q_j}\rho A_p r_p(r_p\omega_j)^2,\label{eqn:Qwind}
\end{align}
where $C_{Q_j}\in \Ree$ is the torque coefficient~\cite{Padfield2007} given by
\begin{gather}
C_{Q_j}=C_{T_j}(\lambda_j+\mu_{z_j})+\frac{C_{D_0}s}{8}(1+3\mu_{x_j}^2). \label{eqn:CQ}
\end{gather}
In short, $U_e^\prime, M_e^\prime$ in~\eqref{eqn:EC2} and \eqref{eqn:EC4} are replaced by \eqref{eqn:Uewind} and \eqref{eqn:Mewind}, respectively, to simulate the quadrotor dynamics under the effects of winds.  

\begin{figure}
	\centerline{
		\subfigure[Position ($\si{\meter}$)]{\includegraphics[width=0.65\columnwidth]{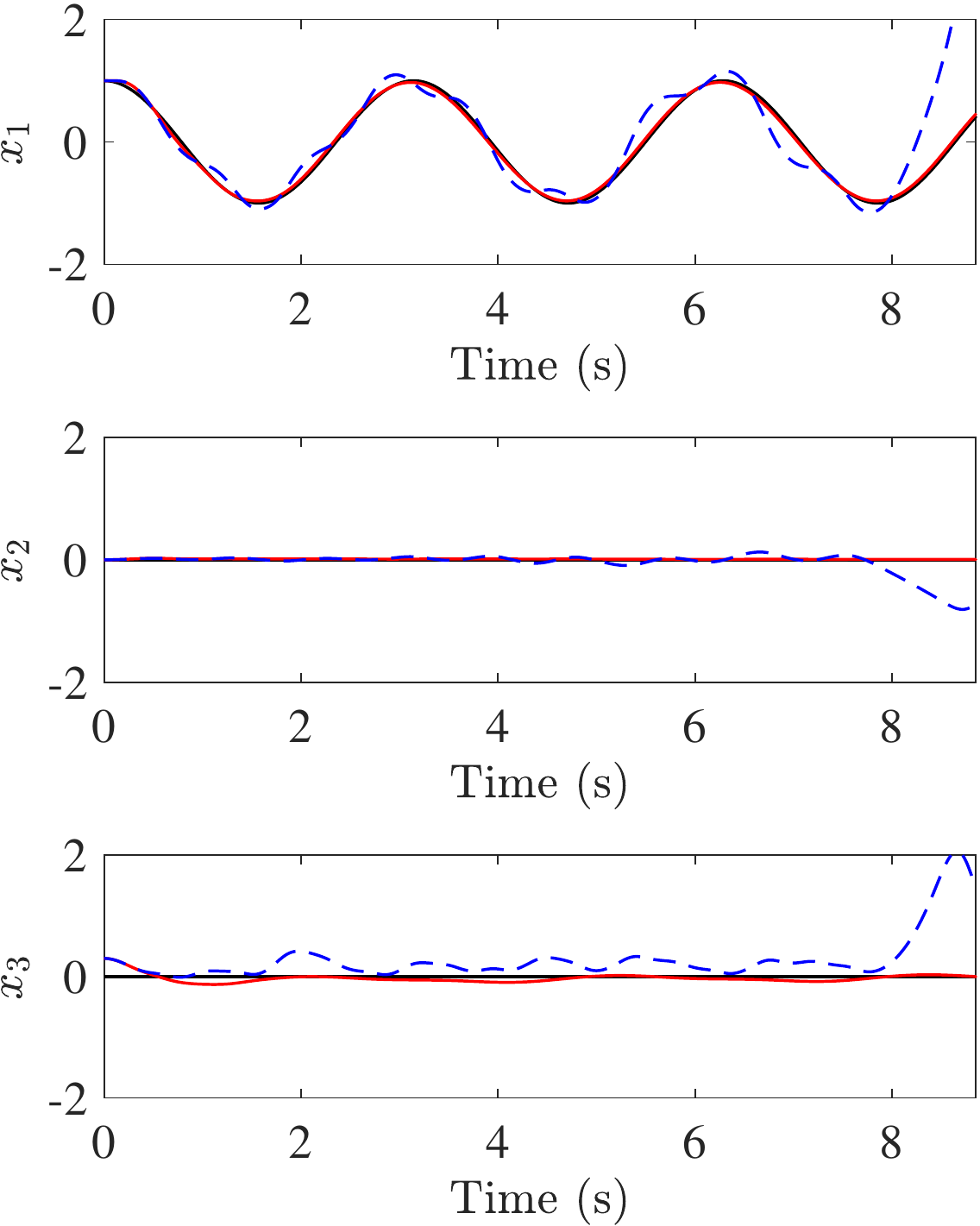}}}
	\centerline{
		\subfigure[Translational velocity ($\si{\meter \per \second}$)]{\includegraphics[width=0.65\columnwidth]{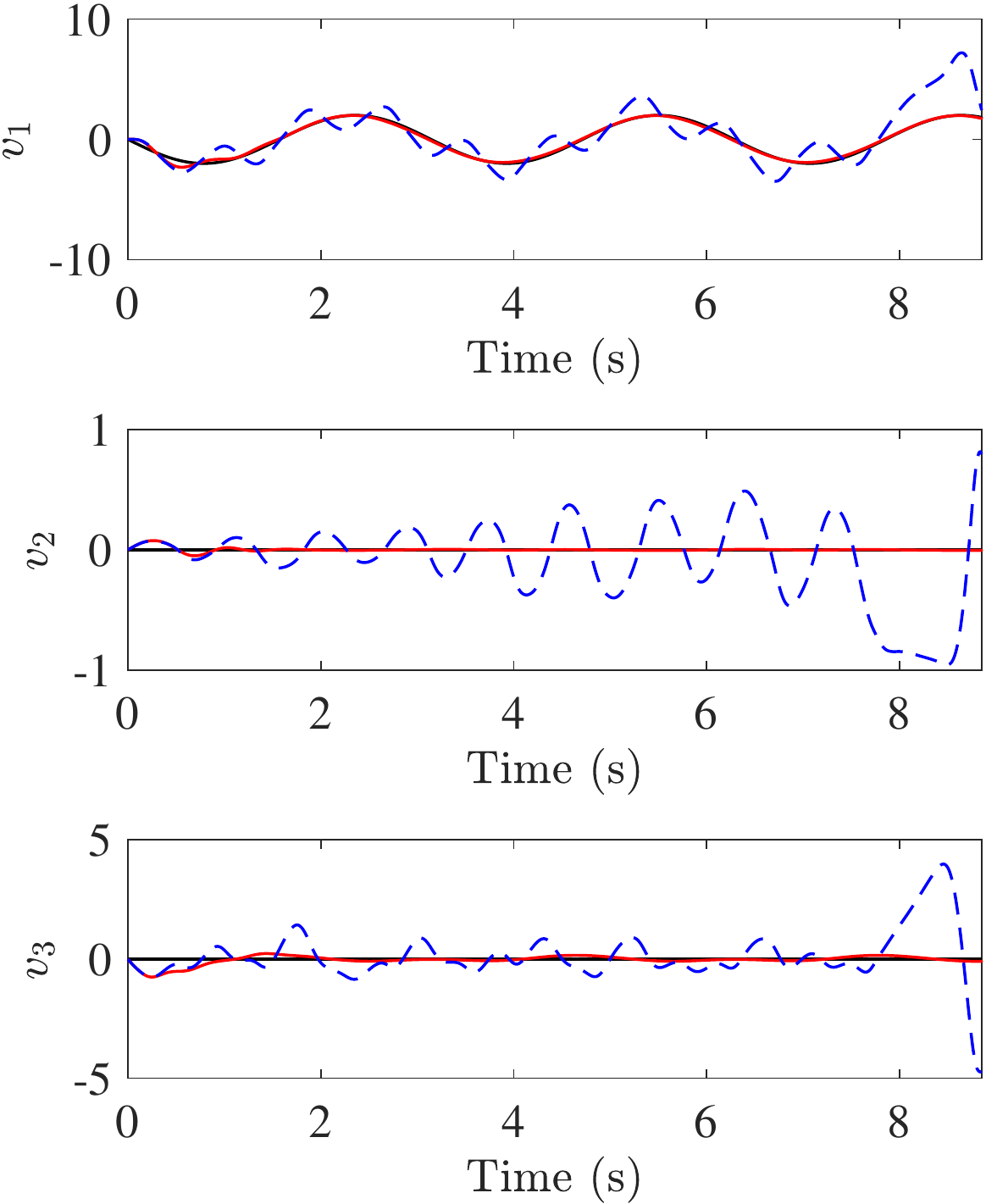}}
	}
	\centerline{
	\subfigure[Thrust ($\si{\newton}$)]{\includegraphics[width=0.65\columnwidth]{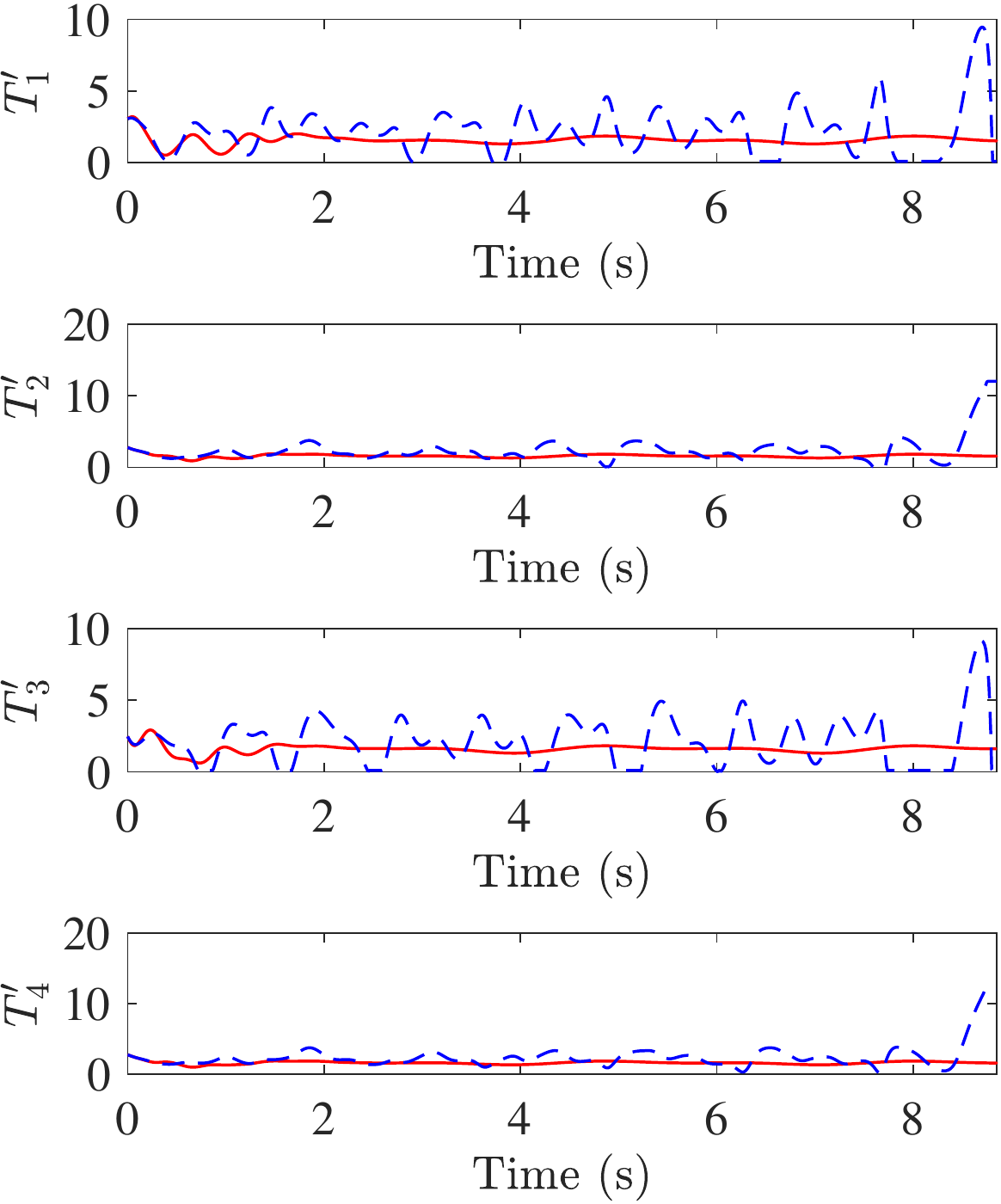}}}
	\caption{Position, velocity and thrust trajectories (desired:black, adaptive controller:red, without disturbance rejection~\cite{LeeLeoPICDC10}:blue)}
	\label{fig:results1}
\end{figure}
\begin{figure}
	\centerline{
	\subfigure[Attitude]{\includegraphics[width=1\columnwidth]{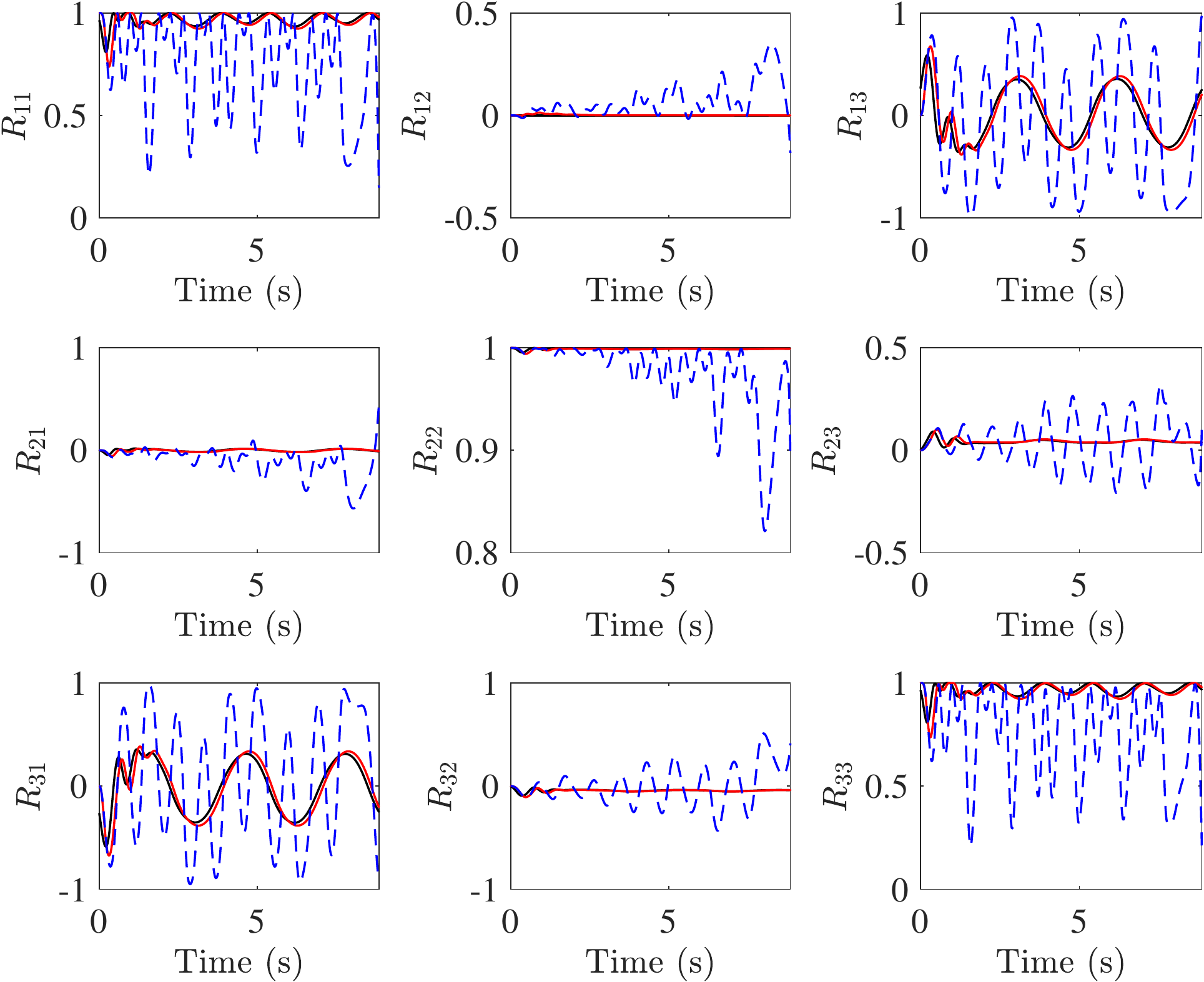}}}
\centerline{
	\subfigure[Angular velocity ($\si{\radian \per \second}$)]{\includegraphics[width=0.65\columnwidth]{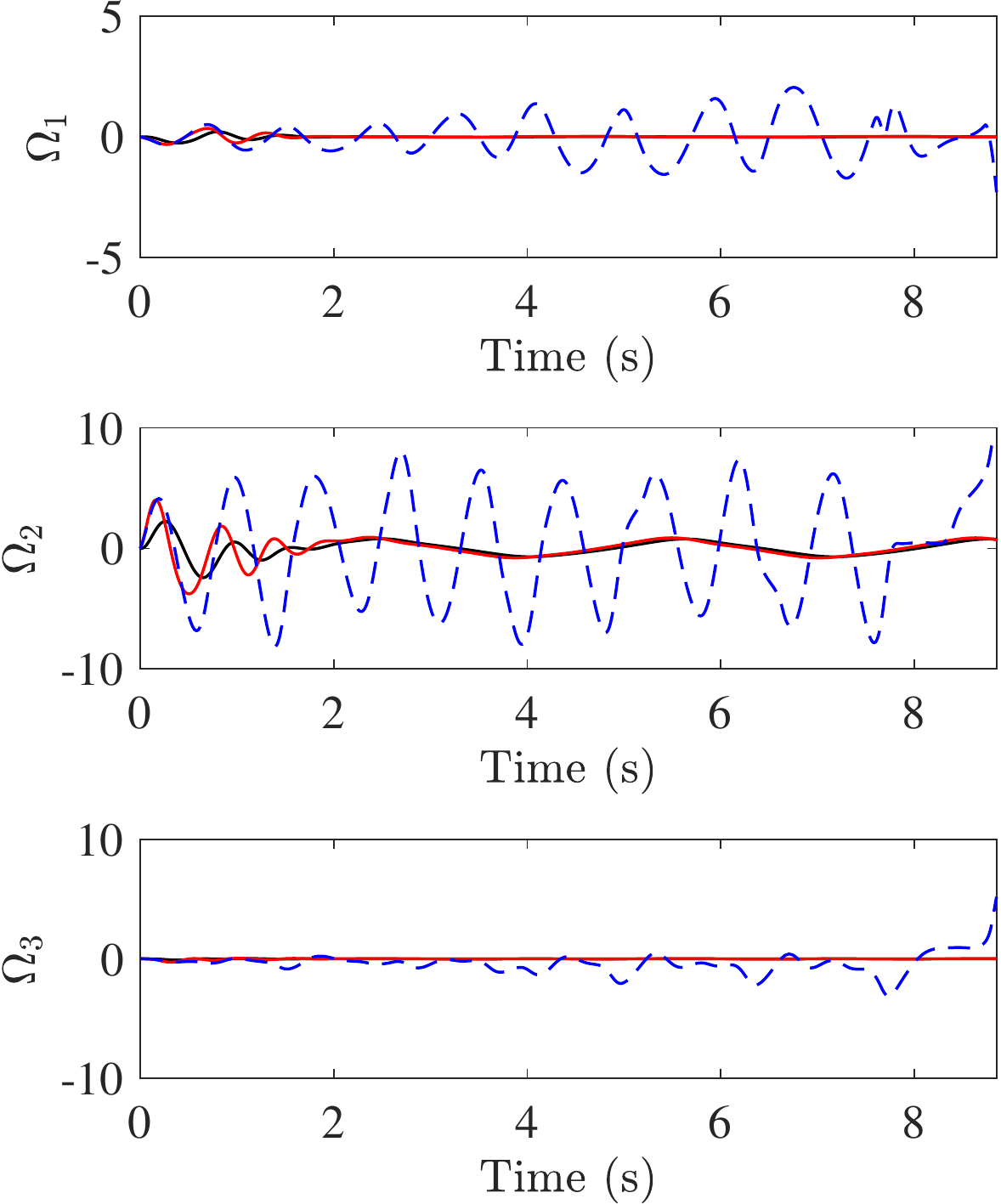}}
}
	\caption{Attitude and angular velocity trajectories (desired:black, adaptive controller:red, without disturbance rejection~\cite{LeeLeoPICDC10}:blue)}
	\label{fig:results2}
\end{figure}

\subsection{Position Tracking Control}

The parameters of the quadrotor considered in the numerical simulation are as follows. 
\begin{gather*}
m=\SI{0.755}{\kilogram}, \, d_h=0.169\si{\meter}, \, d_v=0.1 \si{\meter}, \\
J=10^{-2} \diag{0.557,0.557,1.05}\si{\kilogram\meter\squared},\,T_{max}=7\si{\newton},\\
\,C_{TQ}=1.67\times 10^{-2} \si{\meter},\, C_\alpha=1\times 10^{-3} \si{\radian \second \per \meter},\\ C_d=0.01\si{\kilo \g \per \meter},\,
c=0.01\si{\meter},\, N_b=2,\,r_p=0.1016\si{\meter}.
\end{gather*}
Initially the quadrotor is at rest as specified by
\begin{gather*}
x_0=[0,0,0.3]^T\,\si{m}, \, 
v_0=[0,0,0]^T\,\si{\meter \per \second},\\
R_0=I_{3\times 3}, \, 
\Omega_0=[0,0,0]^T\si{\radian \per \second}.
\end{gather*}

The controller gains are chosen as
\begin{gather*}
k_x=12.08,\, k_v=4.2280,\,k_R=0.3780,\,k_\Omega=0.0882,\label{eqn:geo_gains}\\
\gamma_{w_1}=1  ,\, \gamma_{v_1}=0.05  ,\,\kappa_1=0.00001  ,\\
\gamma_{w_2}=1  ,\, \gamma_{v_2}=0.01  ,\,\kappa_2=0.001,\\
N_{2_1}=N_{2_2} =3.
\end{gather*}

The desired trajectory is a sinusoidal oscillation along the first inertial axis. 
More specifically,
\begin{gather}
x_{d}(t)=[\cos{2t},\,0,\,0]^T\si{\meter},
\end{gather}
and the desired direction of the first body-fixed axis is
\begin{gather}
b_{1d}=[1,0,0]^T.\label{eqn:desired}
\end{gather}

It is assumed that the wind is blowing in the inertial frame as follows
\begin{gather}
v_w=[3,5,0.5]^T\si{\meter \per \second}.
\end{gather}

The corresponding simulation results are presented in Figure~\ref{fig:results1}--\ref{fig:results2}.
To illustrate the advantage of the adaptive controller, we also present the simulation results without using neural network~\cite{LeeLeoPICDC10}.
Specifically, the total thrust and torque are given by \refeqn{A}, \refeqn{f}--\refeqn{Mc} with $\bar \Delta_1,\bar \Delta_2=0_{3\times 1}$.
In Figure~\ref{fig:results1}--\ref{fig:results2}, the desired trajectory, the results of the proposed adaptive controller, and the simulation results of the controller in~\cite{LeeLeoPICDC10} are denoted by the black solid line, the red solid line, and the blue dashed line, respectively. 
It is shown that in the absence of adaptive neural network terms, the controlled trajectories diverges as time increases. 
However, the proposed controller successfully mitigates the wind effects for both  the translational dynamics and the rotational dynamics.
Furthermore, as shown in Figure~\ref{fig:results1}, the thrust at each rotor remains in the acceptable range, well under the maximum thrust  $T_{max}=7\si{\newton}$.

\section{Quadrotor UAV Flight Experiments}\label{chap:Experiment} 

In this section, the proposed geometric adaptive controller is validated via flight experiments with a quadrotor unmanned aerial vehicle that is designed and developed from the ground by the authors.
To demonstrate the capability to reject disturbances, flight experiments are performed under winds generated by an industrial fan. 
First, we describe the hardware and software configurations. 
Then, we present experimental results in two sections, including attitude and flight trajectory tracking. 
Additional experimental results are available in~\cite{BisPhd18}.

\subsection{Hardware Configuration}
The quadrotor UAV platform developed in Flight Dynamics and Control Laboratory (FDCL) at The George Washington University is shown in Figure~\ref{fig:Quad_photo}. 

\begin{figure}
	\centering 
	\includegraphics[width=0.5\columnwidth]{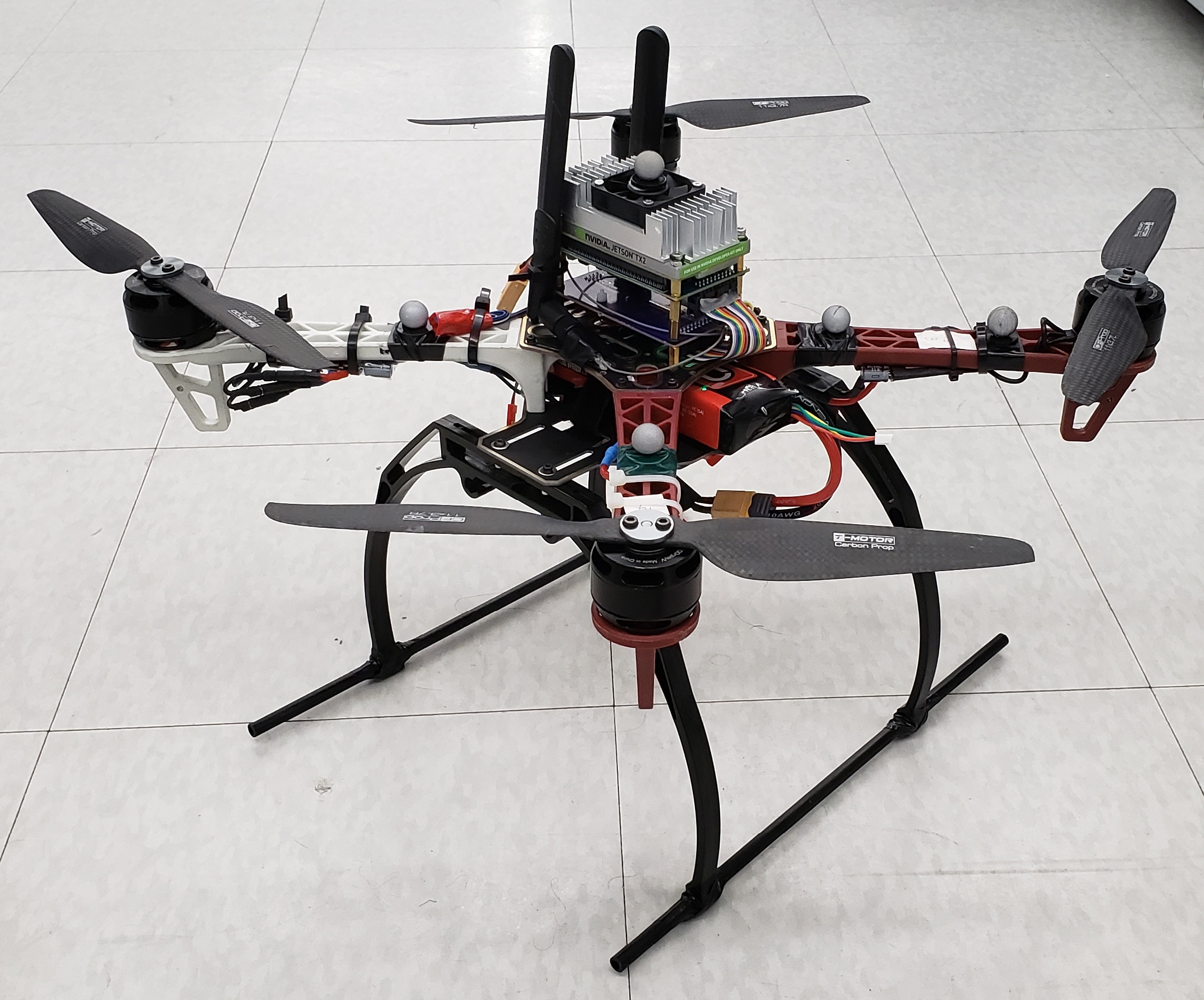}
	\caption{Quadrotor UAV developed in the Flight Dynamics and Control Laboratory} \label{fig:Quad_photo}
\end{figure}

It has four brush-less DC electrical motors (700 KV T-Motor) paired with  11 $\times$ 3.7 carbon fiber propellers.
To control the rotational speed of motors, each one is connected to an electronic speed control (MikroKopter BL-Ctrl v2) which receives the commands through Inter-integrated Circuit (I2C) protocols from an onboard computer.

All computations are done on an embedded system-on-module (NVIDIA Jetson TX2) running a Linux operating system (Ubuntu 16.04 with JetPack 3.3). 
The onboard computer is attached to an expansion board (Connect Tech’s Orbitty Carrier), which is connected to a custom-designed printed circuit board.
This board houses a 9-axis Inertial Measurement Unit (IMU) (VectorNav VN100 IMU) and I2C connection headers for the motor speed controller. 
The computing module communicates with a ground server (Macbook Pro) through Wi-Fi, to receive flight commands and data logging.
A single \SI{14.8}{\volt} Li-Po battery provides power for the motors and the onboard computer.
An optical motion capture system (VICON) measures the position and orientation of the quadrotor, and send their data through Wi-Fi to the onboard computer, which is fed to an estimator to integrate the measurements from IMU and VICON, and to determine the velocity.

The mass and the inertia matrix is measured by building a model in SOLIDWORKS are as follows
\begin{gather*}
    J=\mathrm{diag}[0.02, 0.02, 0.04] \si{kgm^2},\;
m= 2.1 \si{kg}, \; d_h=0.23 \si{m}.\label{eqn:quad_specs_exper}
\end{gather*}

\subsection{Flight Software}

Flight software is a multi-thread program written in C++ using POSIX thread library to execute multiple tasks simultaneously.
This includes threads for data log, communication, estimation, and control with the average frequencies of 100, 60, 100, 400 Hz respectively. 
Additional software is developed for the ground server that transmits commands to the quadrotor and receives the flight data from the onboard computer to monitor the quadrotor responses. 
A graphical user interface is designed using the Glade library to monitor the flight data and to enhance user interactions.  
The flight data is saved in the host computer for post-processing.

\section{Attitude Trajectory Tracking Control}\label{sec:att_experiment}

We first perform experiments for attitude controls, after attaching the quadrotor to a spherical joint to prevent any translation.
In particular, the spherical rolling joint model no. SRJ012C-P from Myostat Motion control Inc is affixed to an aluminum bar, as illustrated in~\reffig{stand}.
It allows up to 30 degrees in roll and pitch, and unlimited yaw.

As the spherical joint is below the mass center, this setup resembles the dynamics of an inverted rigid body pendulum, and there is an additional gravitational torque in~\refeqn{EC4}. 
As such, the control moment in~\refeqn{Mc} is augmented by a canceling term. 
Also, the moment of inertia is translated to the center of rotation~\cite{BisPhd18}.


\begin{figure}
	\centering 
	\includegraphics[width=0.3\columnwidth]{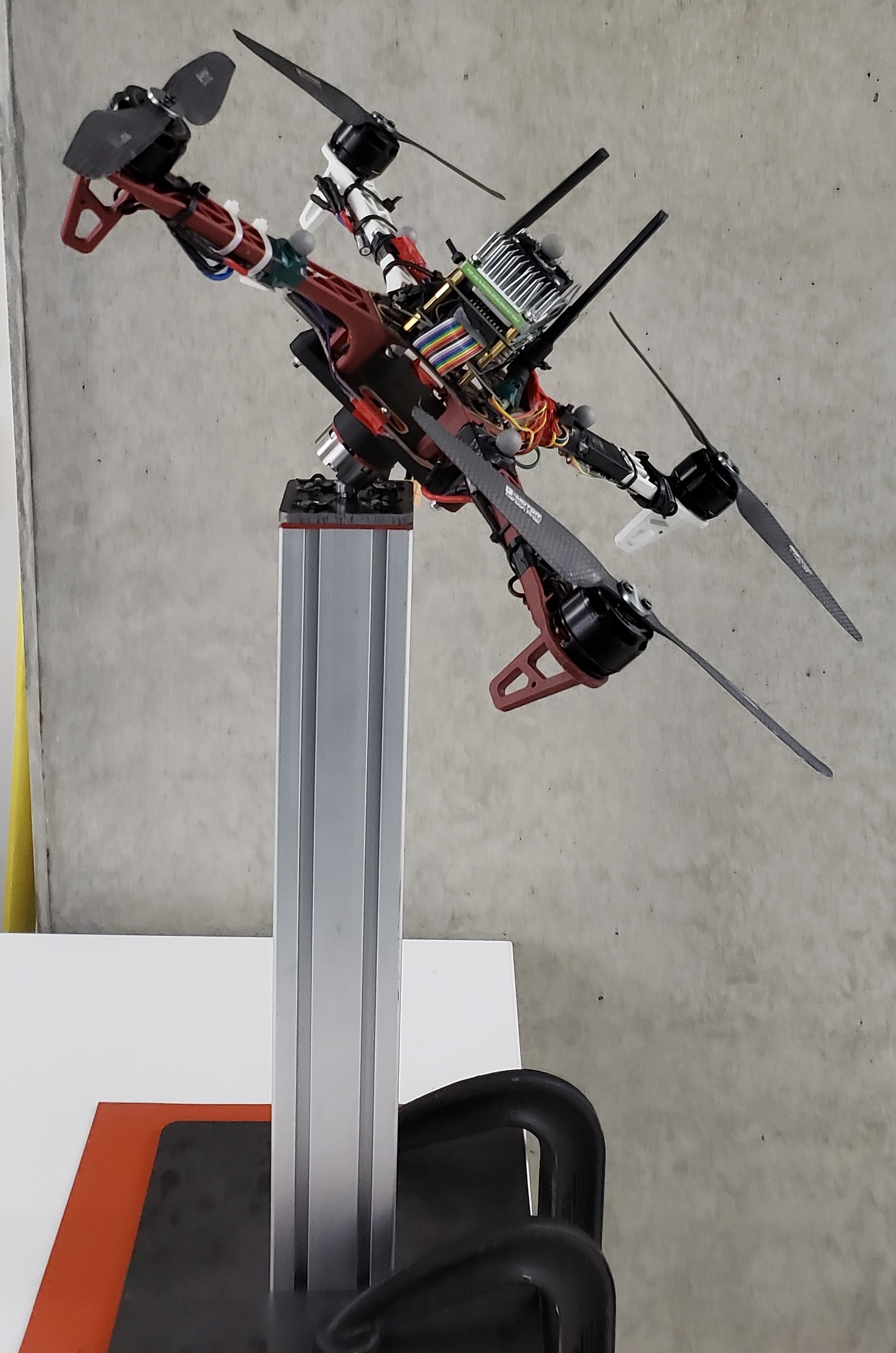}
	\caption{Quadrotor with the spherical joint setup for the attitude dynamics experiments} \label{fig:stand}
\end{figure}

To generate wind disturbance in the indoor flight test facility, an industrial pedestal fan, Air King fan model 9175 with the maximum air speed of 8780 Cubic Feet per Minute (CFM) is placed.
Wind blowing toward the quadrotor is measured with a TriSonica-Mini 3-dimensional sonic anemometer at several locations as shown in Figure \ref{fig:wind_map_attitude}. 
Most of the wind is generated along the $-e_2$ direction in the inertial frame, and there are nontrivial turbulence along every direction. 

We consider two cases: attitude hovering and attitude tracking, and each case is compared with the geometric control without any disturbance compensation presented in~\cite{LeeLeoPICDC10}.

\begin{figure}
	\centerline{
		\subfigure[Schematic of quadrotor UAV attitude test]{\includegraphics[width=0.7\columnwidth]{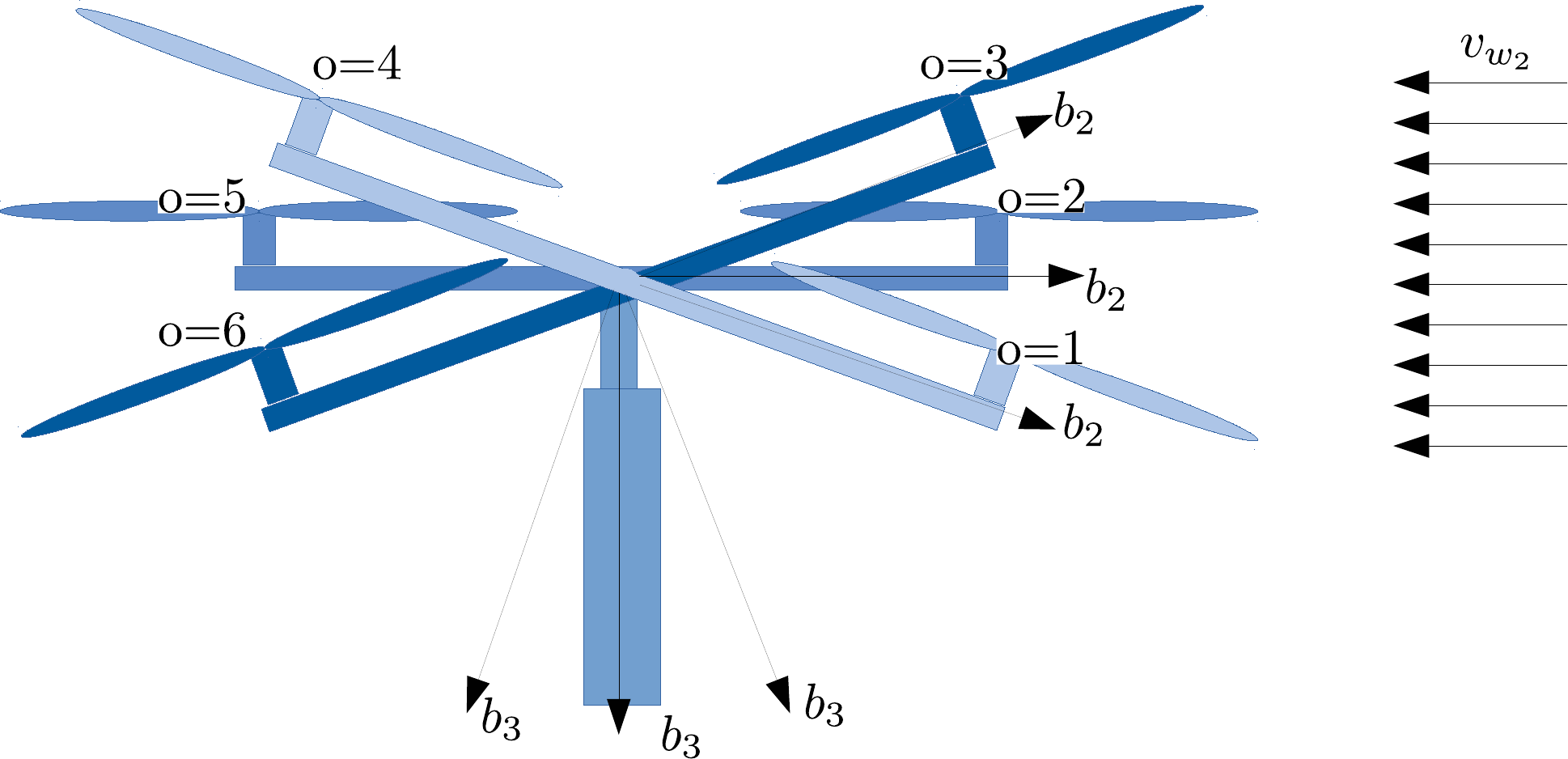}}
	}
	\centerline{
		\subfigure[Wind ($\si{\meter \per \second}$) versus time ($\si{\second}$) for different orientations in front of the fan]{\includegraphics[width=0.8\columnwidth]{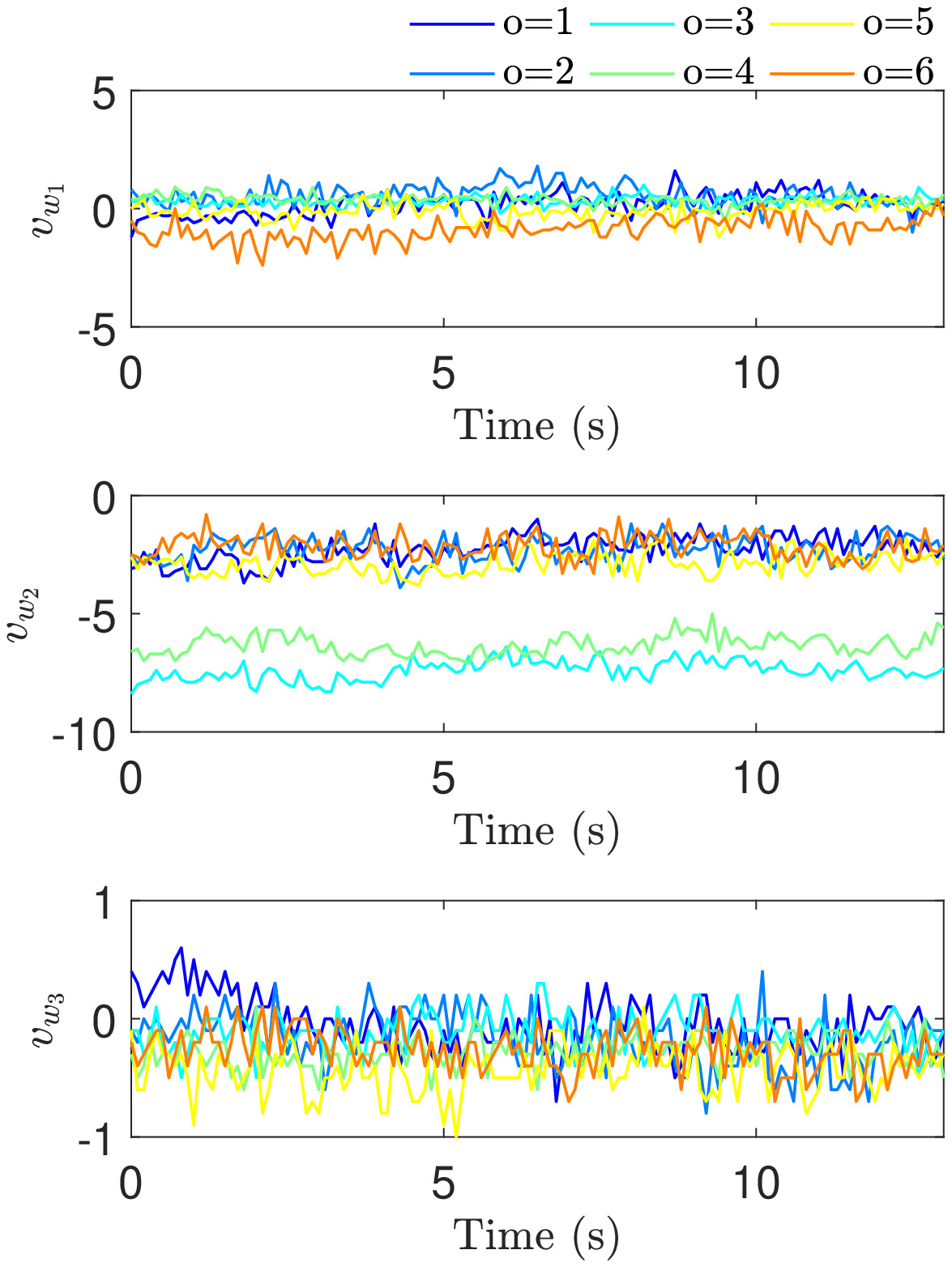}}
	}
\caption{Wind measurement for the attitude trajectory test. (At Figure (c), generated with Matlab \textit{boxplot} function, outliers are shown by '+' symbol, and are not included inside whiskers.) }
\label{fig:wind_map_attitude}
\end{figure}

\subsection{Geometric Adaptive Control for Hovering}\label{Sec:hover_attitude}

The desired attitude is $R_d(t)=I_{3\times 3}$.
The controller gains and parameters are chosen as
\begin{gather*} 
k_R = 1.2, \quad k_\Omega= 0.6,\\
\gamma_{w_2}=1  ,\quad \gamma_{v_2}=0.01  ,\quad \kappa_2=0.001, \quad c_2=1.
\end{gather*}
The number of neurons in the first, hidden and output layers are 
\begin{gather}
N_{1_2}=6,\quad N_{2_2}=3, \quad N_{3_2}=3.
\end{gather}

In \reffig{NNvsPDvsPID_states_hover}--\reffig{NNvsPDvsPID_errror_hover}, the black line shows the desired trajectories.
The trajectories with and without the disturbance rejection are plotted in red and blue respectively.
It can be seen that wind deteriorates tracking the desired trajectory, especially in the axes $b_1$ and $b_2$. 
However, the proposed geometric adaptive controller successfully reduces the error.

\reffig{experimentphoto2} shows the snapshot of the experiment in the $e_2-e_3$ plane, while wind is blowing toward $-e_2$, and $e_3$ points downward. 
In Figure (a) it is shown that in the absence of wind, both controllers reach the desired orientation. 
However, in Figure (b), wind changes the orientation of the UAV and results in an steady state attitude error in the absence of the adaptive controller.
\begin{figure}
	\centering 
	\includegraphics[width=1\columnwidth]{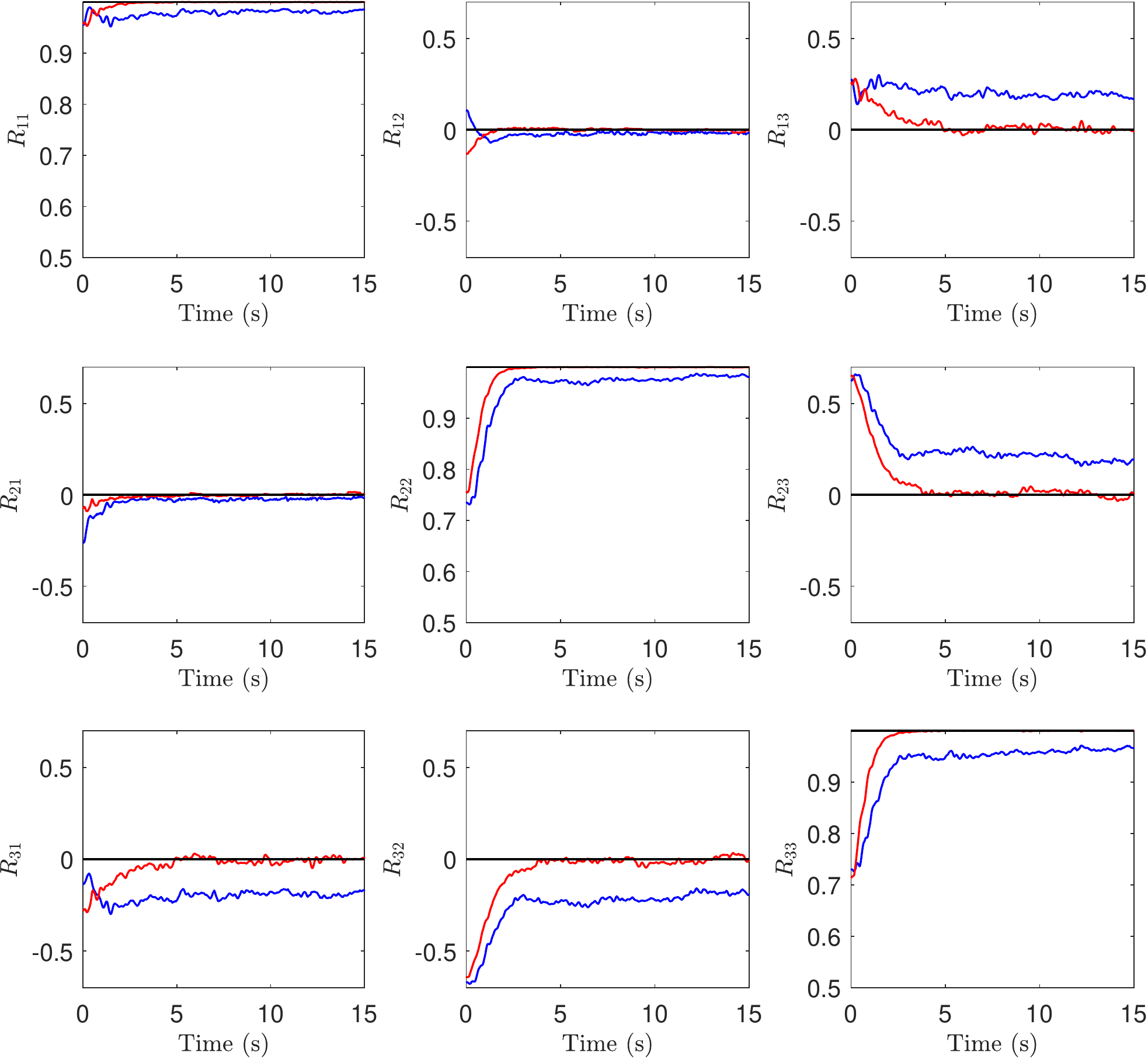} 
    \caption{Attitude hovering (rotation matrix), black: desired, blue: without disturbance rejection~\cite{LeeLeoPICDC10}, red: adaptive controller}
	\label{fig:NNvsPDvsPID_states_hover} 
\end{figure}

\begin{figure}
	\centerline{
		\subfigure[Attitude error]{\includegraphics[width=0.5\columnwidth]{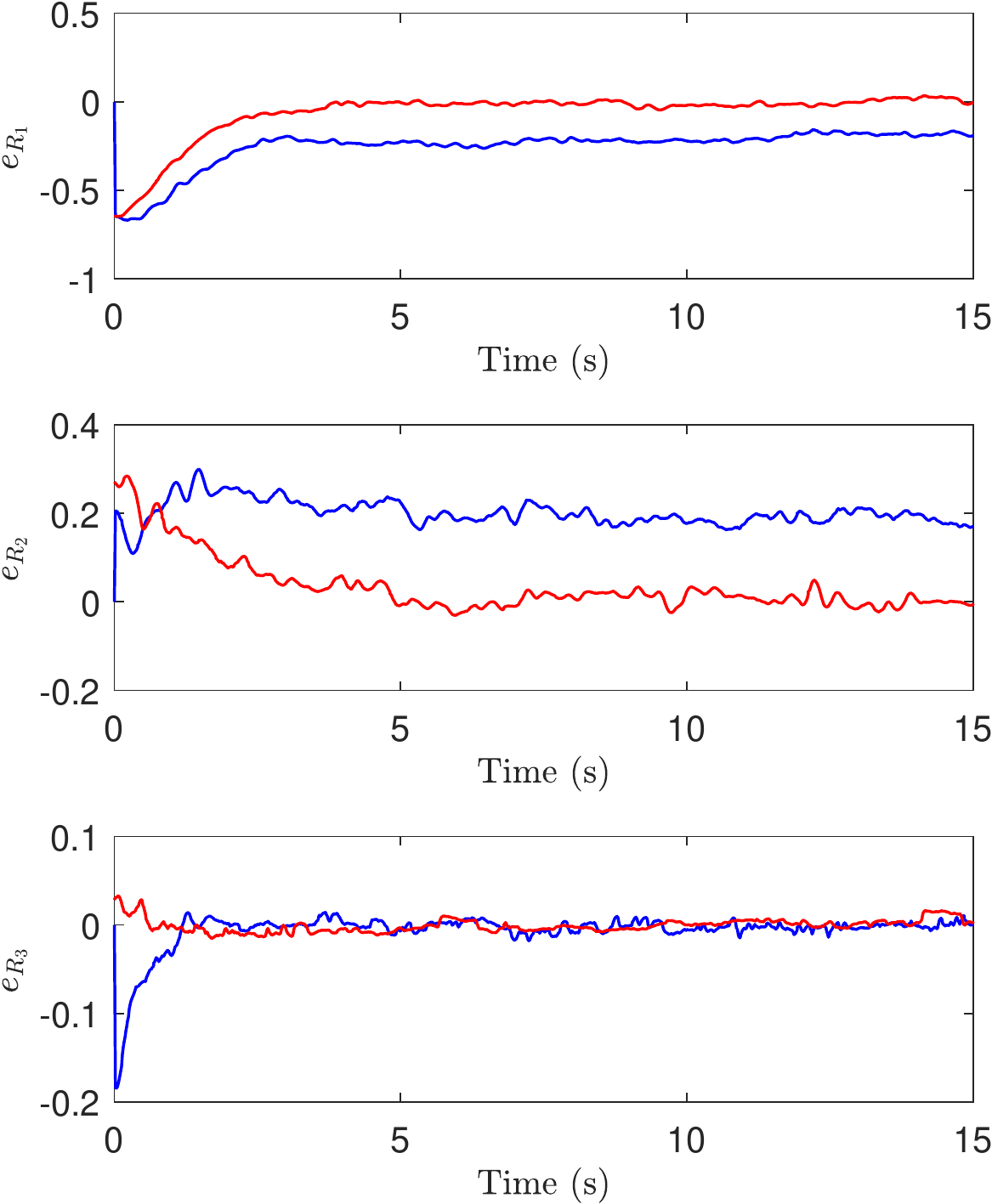}}
		\hfill
		\subfigure[Angular velocity error ( $\si{\radian \per \second}$)]{\includegraphics[width=0.5\columnwidth]{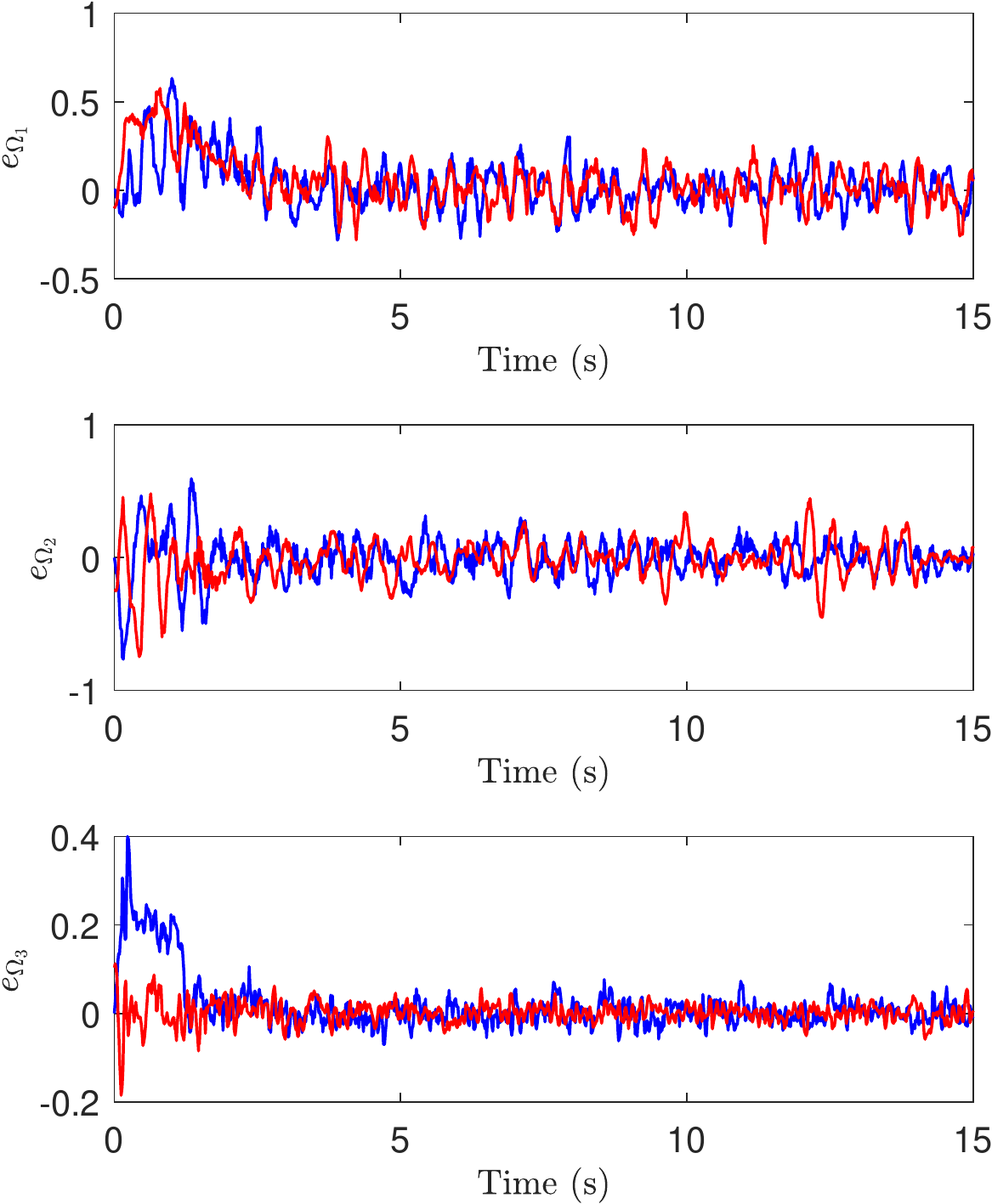}}
	}
	\centerline{
		\subfigure[Angular velocity ($\si{\radian \per \second}$)]{\includegraphics[width=0.5\columnwidth]{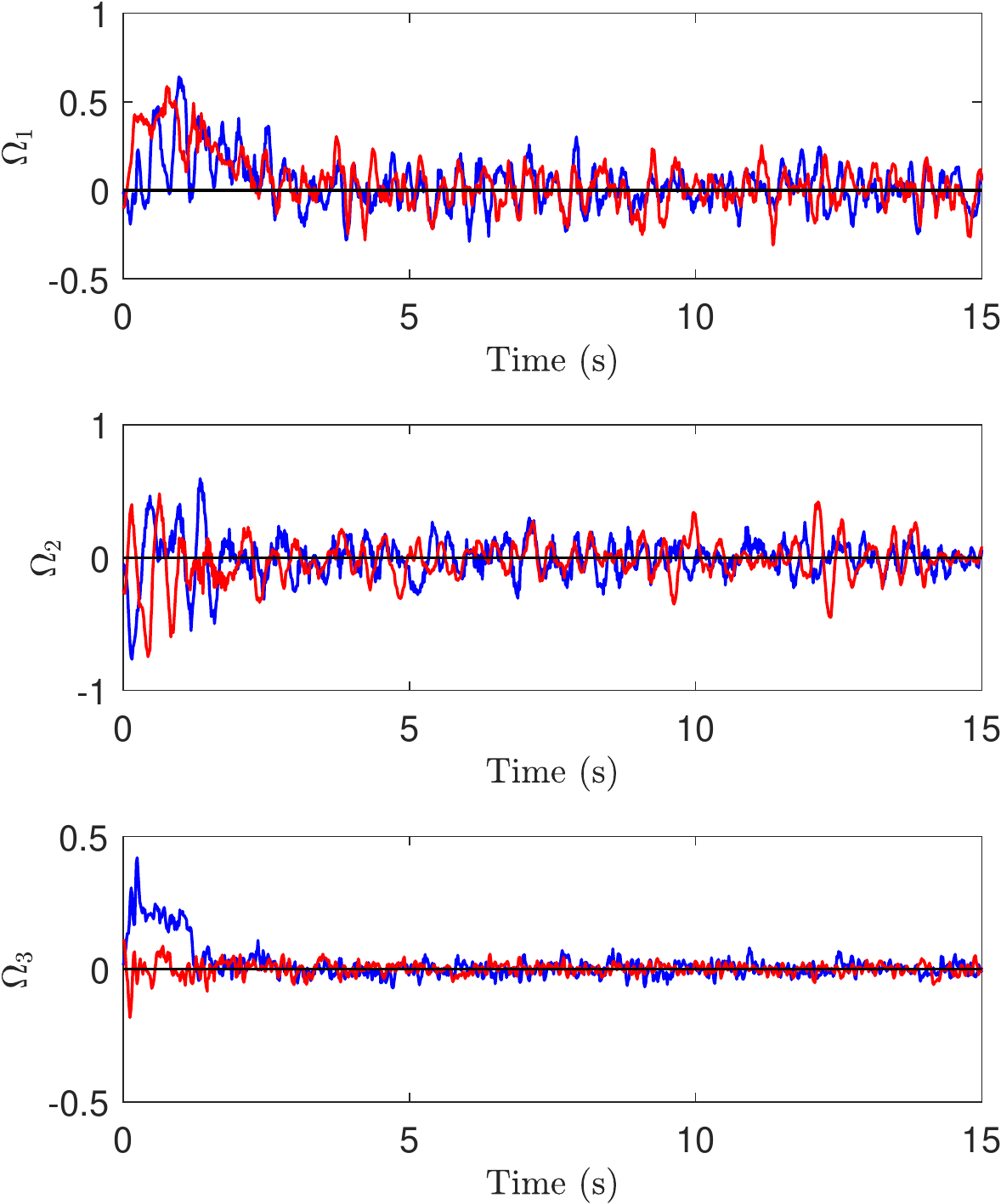}}
		\hfill
		\subfigure[Thrust ($\si{\newton}$)]{\includegraphics[width=0.5\columnwidth]{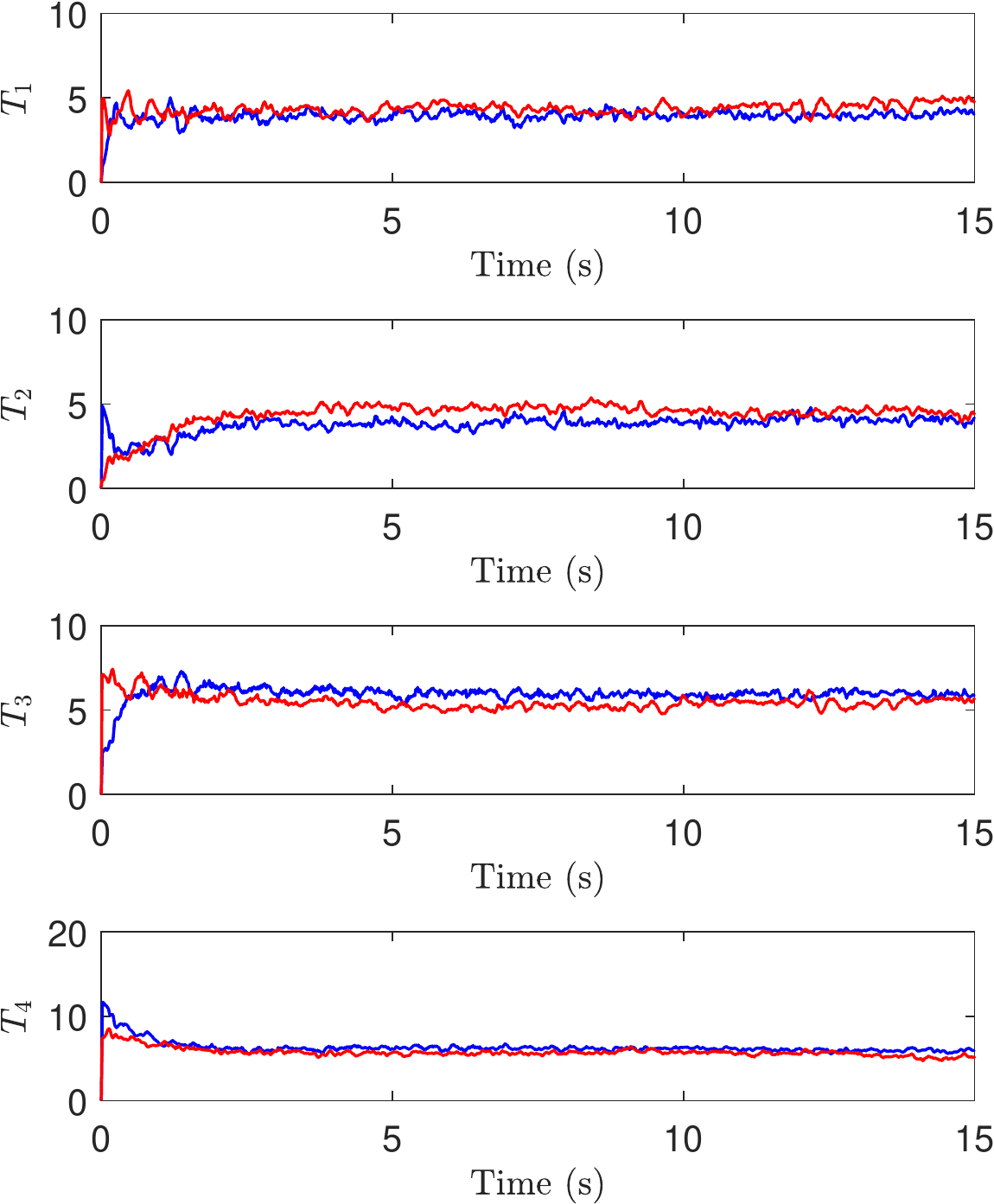}}
	}
    \caption{Attitude hovering (attitude and angular velocity errors, angular velocity and four rotor thrusts), black: desired, blue: without disturbance rejection~\cite{LeeLeoPICDC10}, red:  adaptive controller}
	\label{fig:NNvsPDvsPID_errror_hover} 
\end{figure}

\begin{figure}
	\centerline{
		\subfigure[Hovering without wind]{\includegraphics[width=1\columnwidth]{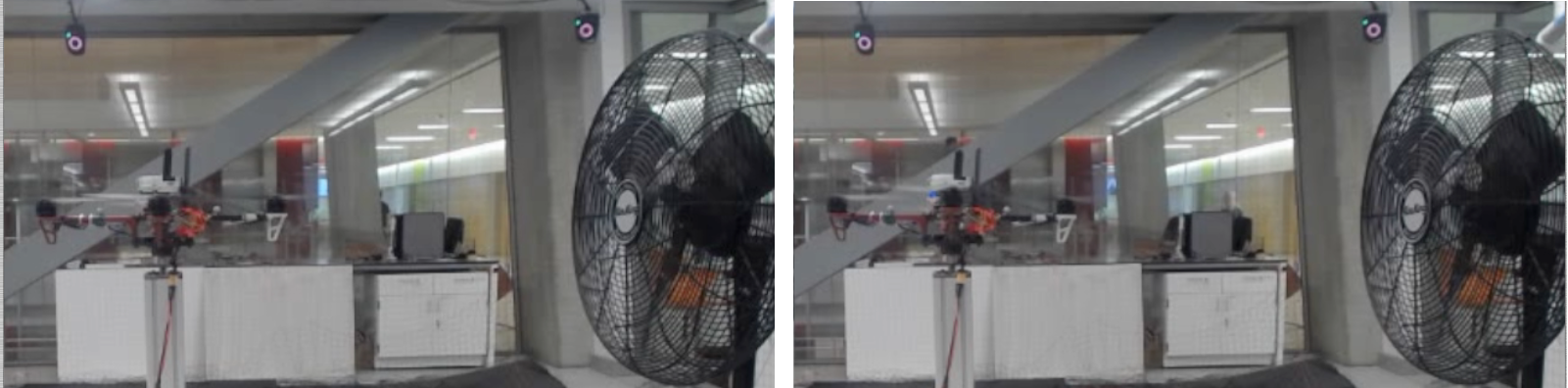}}
	}
	\centerline{
		\subfigure[Hovering without wind]{\includegraphics[width=1\columnwidth]{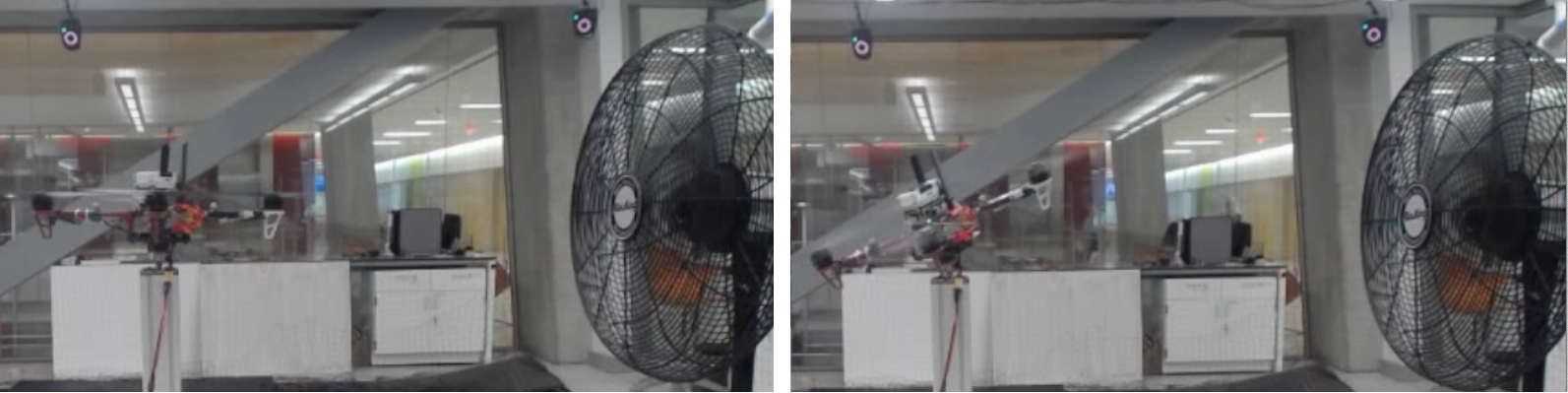}}
	}
    \caption{Attitude hovering (snapshots), left:adaptive controller, right:without disturbance rejection~\cite{LeeLeoPICDC10}}
	\label{fig:experimentphoto2} 
\end{figure}

\subsection{Geometric Adaptive Control for Attitude Tracking}\label{Sec:trac_attitude}

Next, we consider attitude tracking control. 
The desired attitude is parameterized as
\begin{gather}
R_d=\begin{bmatrix}   
c\theta c\phi& s \psi s \theta c \phi-\cos\psi s \phi& c \psi s \theta c \phi+s \psi s \phi\\
c \theta s \phi& s \psi s \theta s \phi+c \psi c \phi& c \psi s \theta s \phi-s \psi c \phi\\
-s \theta  & s \psi c \theta         & c \psi c \theta
\end{bmatrix},\label{eqn:Rd_at_exper}
\end{gather}
where 
$\cos$, and $\sin$ are shown by $c$ and $s$ respectively.
The Euler angles $\psi, \theta, \phi$ are chosen as
\begin{gather}
\psi(t)=\pi A_s \cos(2\pi B_s t),\label{eqn:psi_traj}\\
\theta(t)=\pi A_t \cos(2\pi B_t t),\label{eqn:theta_traj}\\
\phi(t) = \pi A_f \sin(2\pi B_f t),\label{eqn:phi_traj}
\end{gather}
and the trajectory parameters are set to
\begin{gather}
A_s= 0.15,\quad   A_t= 0.12, \quad   A_f= 0.11,\nonumber\\
B_s= 0.5, \quad   B_t= 0.5, \quad    B_f= 0.5.\label{eqn:att_traj_par}
\end{gather}
The desired trajectory is chosen such that the vehicle rotates along the three axes of $b_1, b_2$, and $b_3$ simultaneously, while wind is blowing toward the direction of $-e_2$ in the inertial frame.

The corresponding response of the three different controllers are presented in the following Figures~\ref{fig:NNvsPDvsPID_states}--\ref{fig:NNvsPDvsPID_errror}.
The blue line is for the geometric controller without disturbance rejection~\cite{LeeLeoPICDC10}, the green line is for the geometric controller with an integral term presented in~\cite{Goodarzi_Lee_13}, the red line is for the proposed method.

It can be seen that the geometric controller without disturbance rejection results in large trajectory errors.
The controller presents in~\cite{Goodarzi_Lee_13} improves the results.
However, the proposed geometric controller results in the best performance of trajectory tracking.

\begin{figure}
	\centering 
	\includegraphics[width=1\columnwidth]{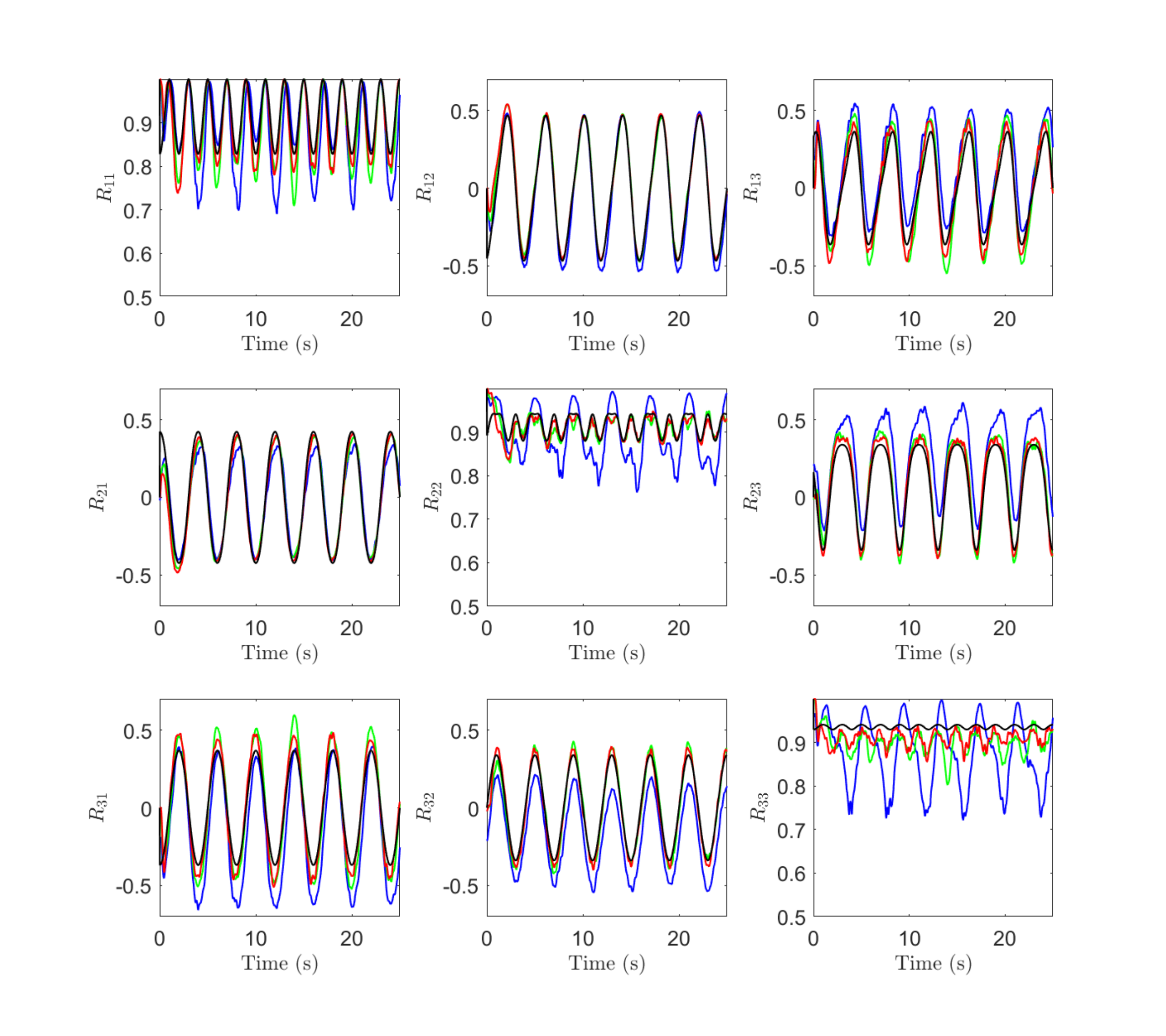} 
    \caption{Attitude tracking (rotation matrix), black: desired, blue: without disturbance rejection~\cite{LeeLeoPICDC10}, green: PID~\cite{Goodarzi_Lee_13}, red:  adaptive controller} \label{fig:NNvsPDvsPID_states} 
\end{figure}

\begin{figure}
	\centerline{
		\subfigure[Attitude error]{\includegraphics[width=0.5\columnwidth]{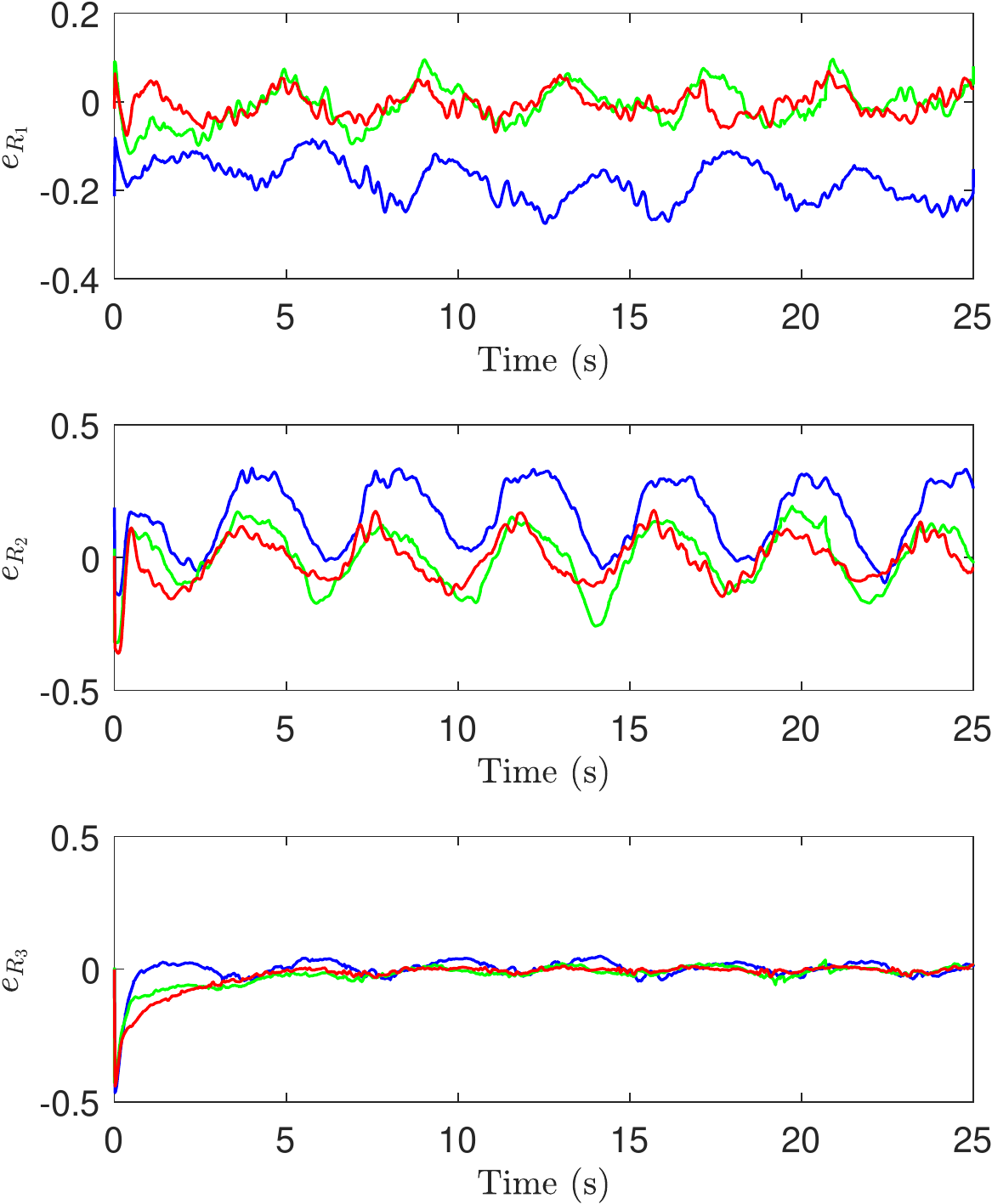}}
		\hfill
		\subfigure[Angular velocity error ($ \si{\radian \per \second}$)]{\includegraphics[width=0.5\columnwidth]{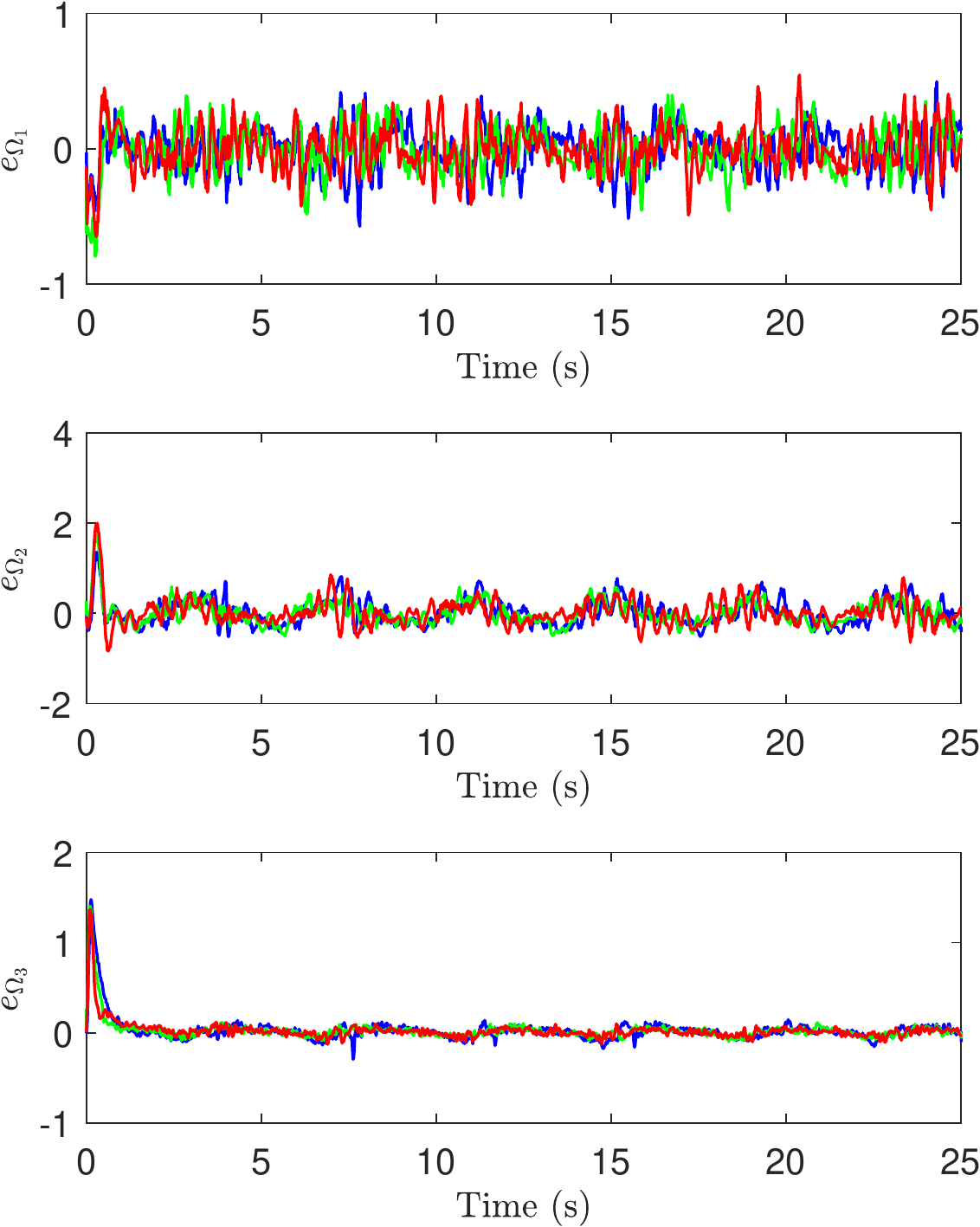}}
	}
	\centerline{
		\subfigure[Angular velocity ($ \si{\radian \per \second}$)]{\includegraphics[width=0.5\columnwidth]{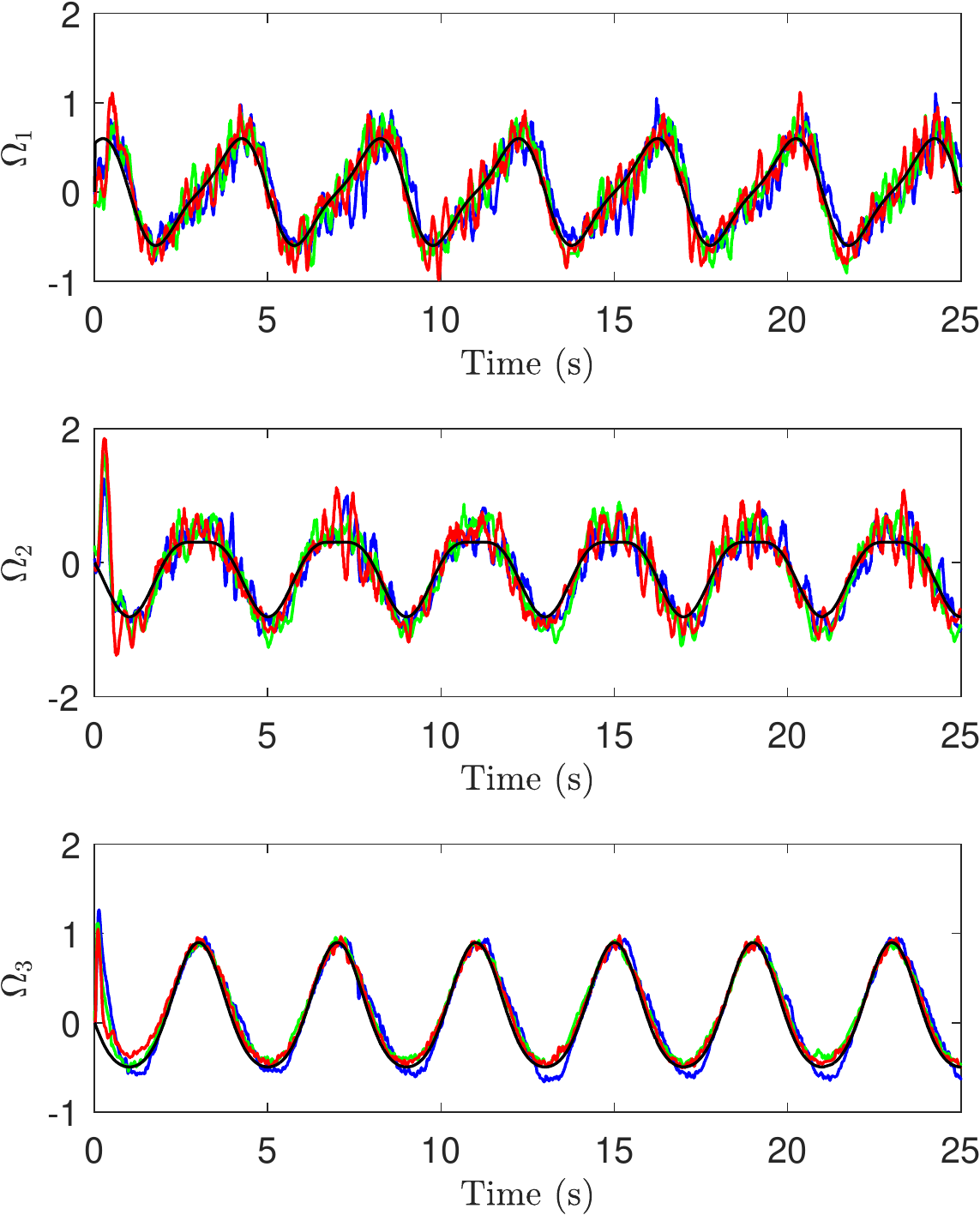}}
		\hfill
		\subfigure[Thrust ($\si{\newton}$)]{\includegraphics[width=0.5\columnwidth]{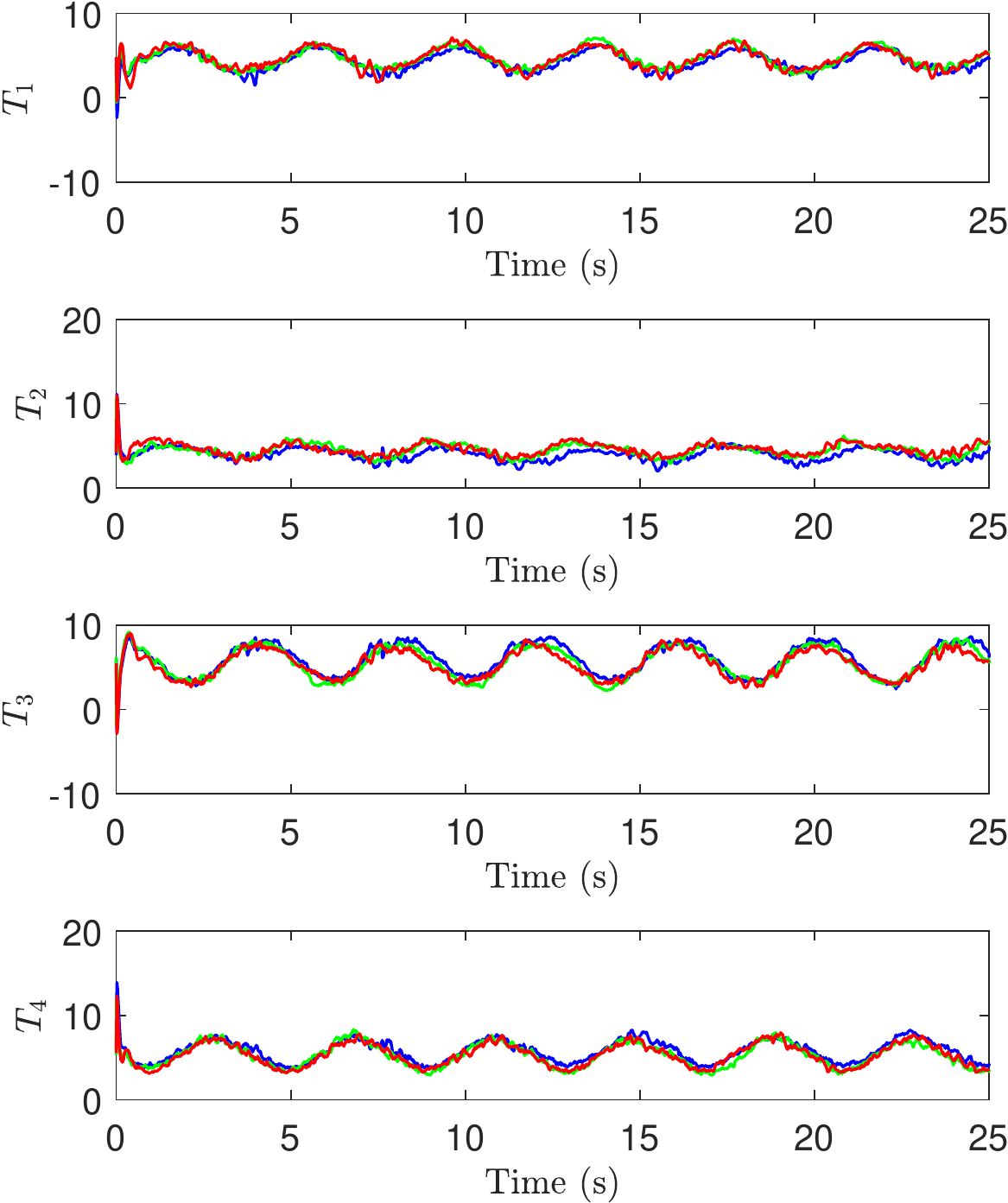}}
	}
    \caption{Attitude tracking (attitude and angular velocity errors, angular velocity and four rotor thrusts), black: desired, blue: without disturbance rejection~\cite{LeeLeoPICDC10}, green: PID~\cite{Goodarzi_Lee_13}, red:  geometric adaptive}
	\label{fig:NNvsPDvsPID_errror}  
\end{figure}

Figure~\ref{fig:experimentphoto} shows the experimental setup in the $e_2-e_3$ plane, while wind is blowing toward $-e_2$, and $e_3$ points downward. 
The photo is taken at the time of $0.5\si{second}$,
when the desired pitch angle is $\phi_d=\ang{19.8}$.
On the left, tracking with the proposed adaptive controller is shown, and on the right the geometric controller without wind disturbance rejection is presented.
It can be seen that there is an large deviation of the desired pitch angle (about $-\ang{19.8}$) in the presence of wind in the absence of disturbance rejection techniques\footnote{For the video file of this experiment, visit the FDCL YouTube channel at \url{https://youtu.be/zUsOif1SfEs} or the experiment section of the FDCL website at \url{http://fdcl.seas.gwu.edu/}.}.

\begin{figure}
	\centering 
	\includegraphics[width=1\columnwidth]{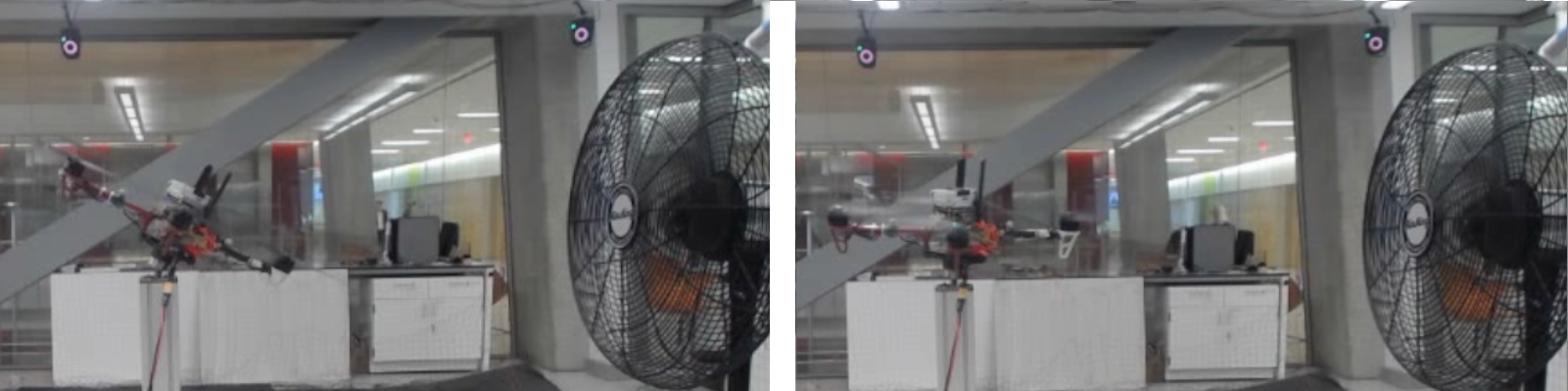} 
    \caption{Attitude tracking (snapshot at $t=0.5 \si{second}$, with the desired pitch angle of $\phi_d=\ang{19.8}$), left: adaptive controller, right: without disturbance rejection~\cite{LeeLeoPICDC10}}
	\label{fig:experimentphoto} 
\end{figure}

\section{Position Trajectory Tracking Control}

In this section, the quadrrotor UAV is detached from the spherical joint used Section \ref{sec:att_experiment}, and it is controlled with the position controller provided in Proposition~\ref{prop:NN_prop}. 
The quadrotor properties are given by
\begin{gather*}
J= \begin{bmatrix} 0.02&0&0\\
0&0.027&0\\0&0&0.04\end{bmatrix}\si{\kilogram\meter\squared},
m=\SI{2.1}{\kilogram}, \\ d_h=0.09\si{\meter},\quad T_{max}=12\si{\newton},\quad C_{TQ}=0.0135\si{\meter}
\end{gather*}
Wind data in front of the fan is measured with TriSonica-Mini 3-dimensional sonic anemometer, and provided in \reffig{wind_map_flight}.

We consider three cases: a hovering flight, a position tracking, and a backflip maneuver. 

\begin{figure}
	\centerline{
		\subfigure[Wind ($\si{\meter \per \second}$) versus time ($\si{\second}$) for different positions in front of the fan]{\includegraphics[width=0.5\columnwidth]{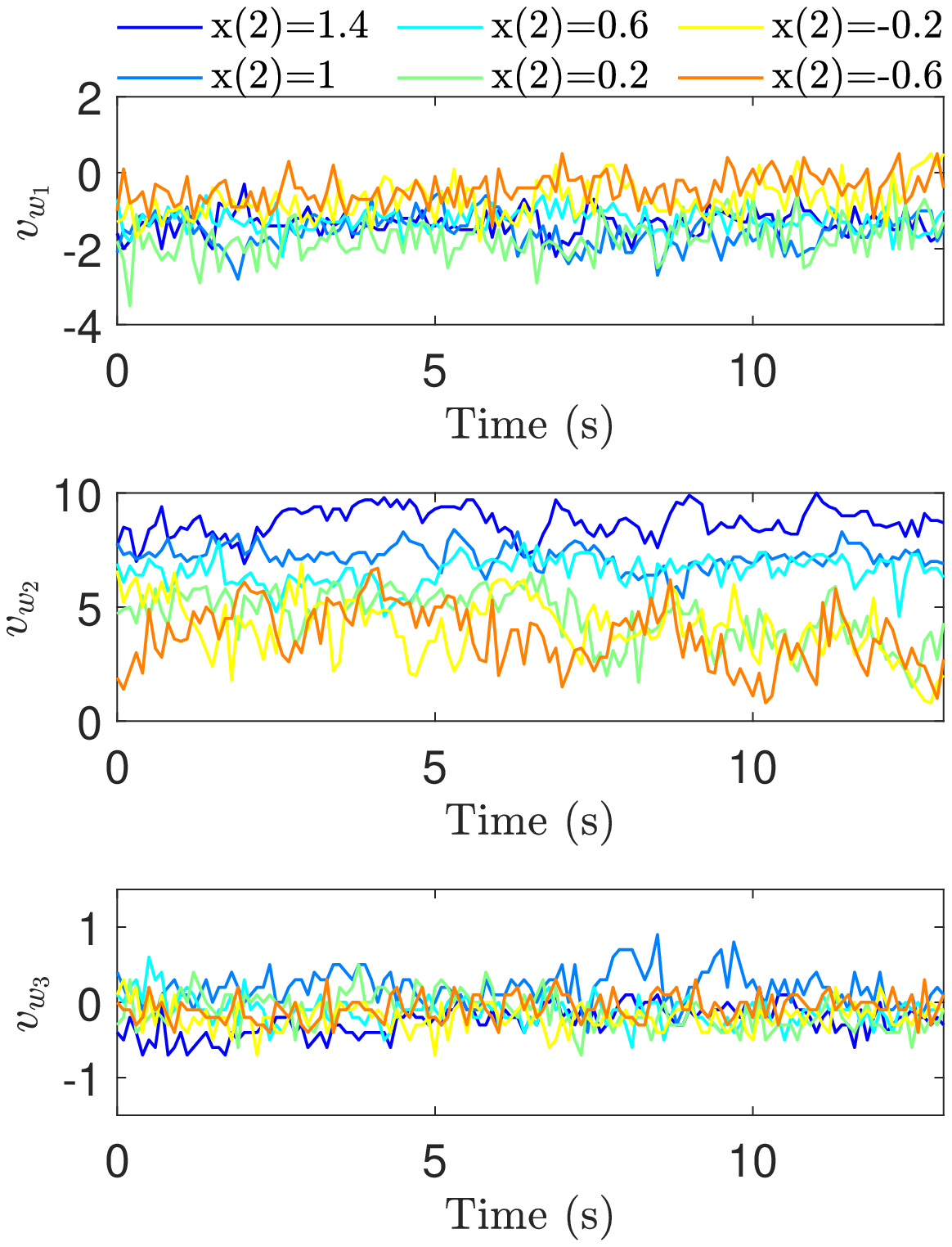}}
		\hfill
		\subfigure[Distribution of wind velocity element $v_{w_2}$ ($\si{\meter \per \second}$ versus distances in front of the fan)]{\includegraphics[width=0.5\columnwidth]{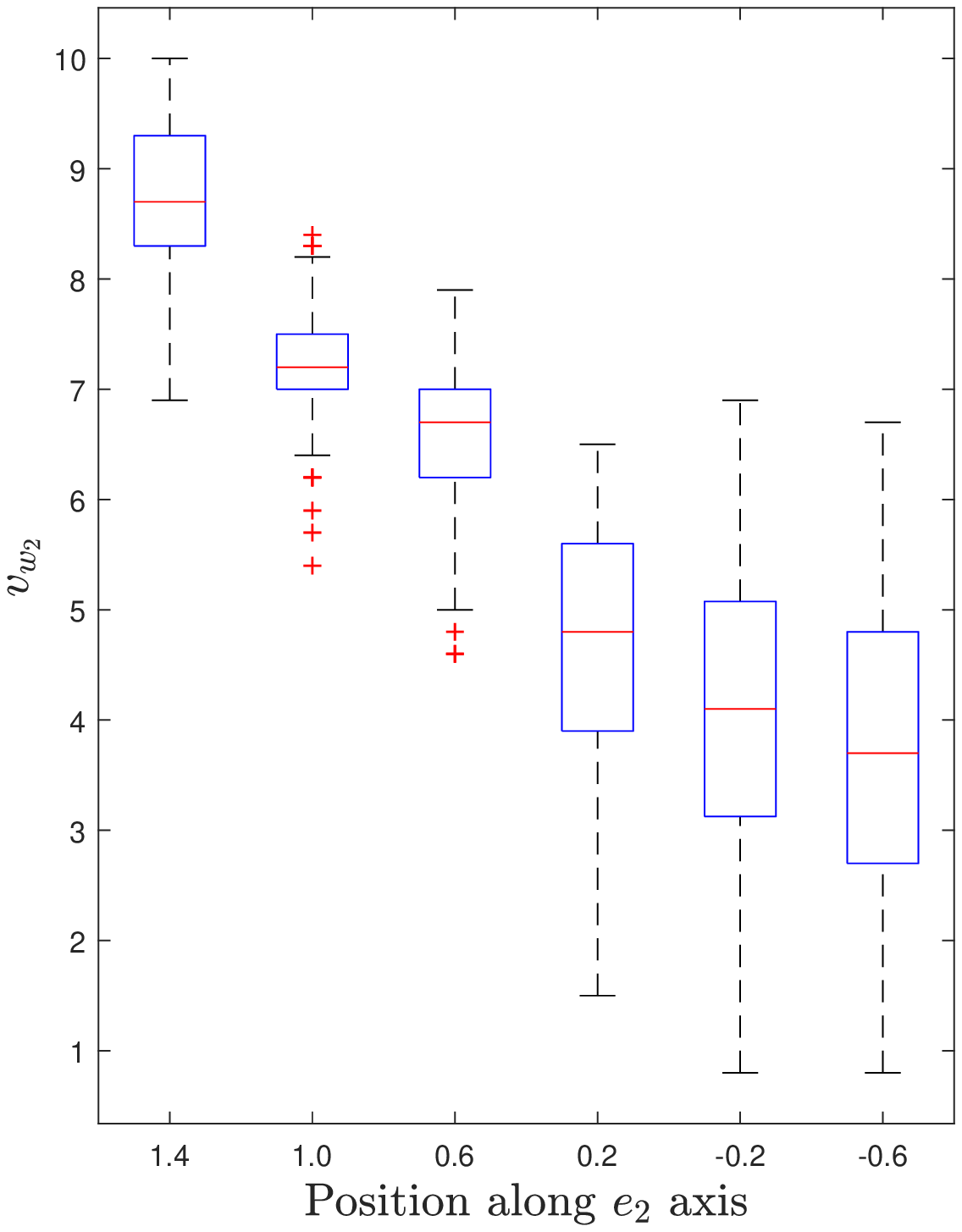}}
	}
	\caption{Distribution of wind velocity element $v_{w_2}$ ($\si{\meter \per \second}$ versus distances in front of the fan)}
	\label{fig:wind_map_flight}  
\end{figure}

\subsection{Geometric Adaptive Control for Hovering}\label{sec:hover_flight}

In this section, we observe the performance of the adaptive controller for hovering flight when the quadrotor is subject to the wind.

Initially the fan is turned off, and it is turned on at about $t=10$ seconds.
The location of the quadrotor along the second inertial frame is $x(2)=1.0 \si{m}$, and as such the average wind speed is about $7.3 \si{m/s}$ as shown at \reffig{wind_map_flight}.

The controller gains and parameters are chosen as
\begin{gather*}
k_x=16.0,\quad k_v= 5.0,\\
k_R=1.2,\quad k_\Omega= 0.3,\\
\gamma_{w_1}= 0.3  ,\quad \gamma_{v_1}=0.3  ,\quad \kappa_1=0.0001, \quad c_1=1,\\ 
\gamma_{w_2}=0.035  ,\quad \gamma_{v_2}=0.035  ,\quad \kappa_2=0.0001, \quad c_2=1.
\end{gather*}
The number of neurons in the first, hidden and output layers are 
\begin{gather}
N_{1_1}=6,\quad N_{2_1}=3,\quad N_{3_1}=3,\\
N_{1_2}=6,\quad N_{2_2}=3,\quad N_{3_2}=3.
\end{gather}

Experimental results are illustrated in~\reffig{hoverflightNNvsPDvsPID_eXeV} and \reffig{hoverflightDeltaNNvsPDvsPID}. 
The trajectories without disturbance rejection are plotted in blue, and with PID controller~\cite{Goodarzi_Lee_13} in green, and with adaptive controller in red.
It can be seen that both controllers with disturbance rejection techniques improve the tracking performance. 
However the adaptive controller outperforms the other, while it does not result in large thrusts.

\reffig{photo_hover_flight} shows the experiment photo. 
The top photo is for hovering flight with the adaptive controller, and in the bottom, the quadrotor supposed to fly closer to the fan, but due to the wind it is far form the desired position\footnote{For the video file this experiment, visit the FDCL YouTube channel at \url{https://youtu.be/ouSsrDfi8DM} or the experiment section of the FDCL website at \url{http://fdcl.seas.gwu.edu/}.}. 
\begin{figure}
	\centerline{\subfigure[Position error ($m$)]{\includegraphics[width=0.5\columnwidth]{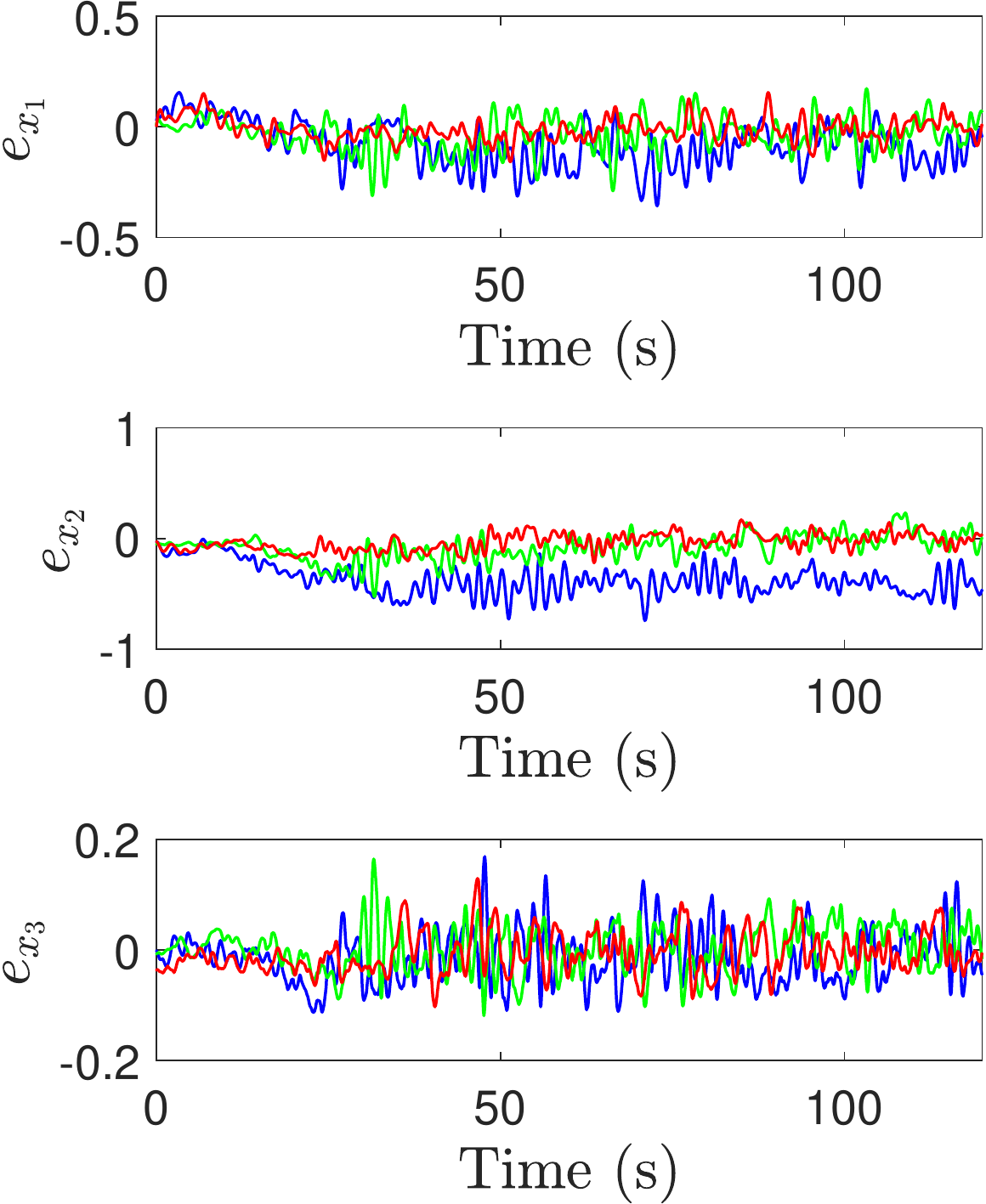}}
		\hfill
		\subfigure[Velocity error ($\si{\meter \per \second}$)]{\includegraphics[width=0.5\columnwidth]{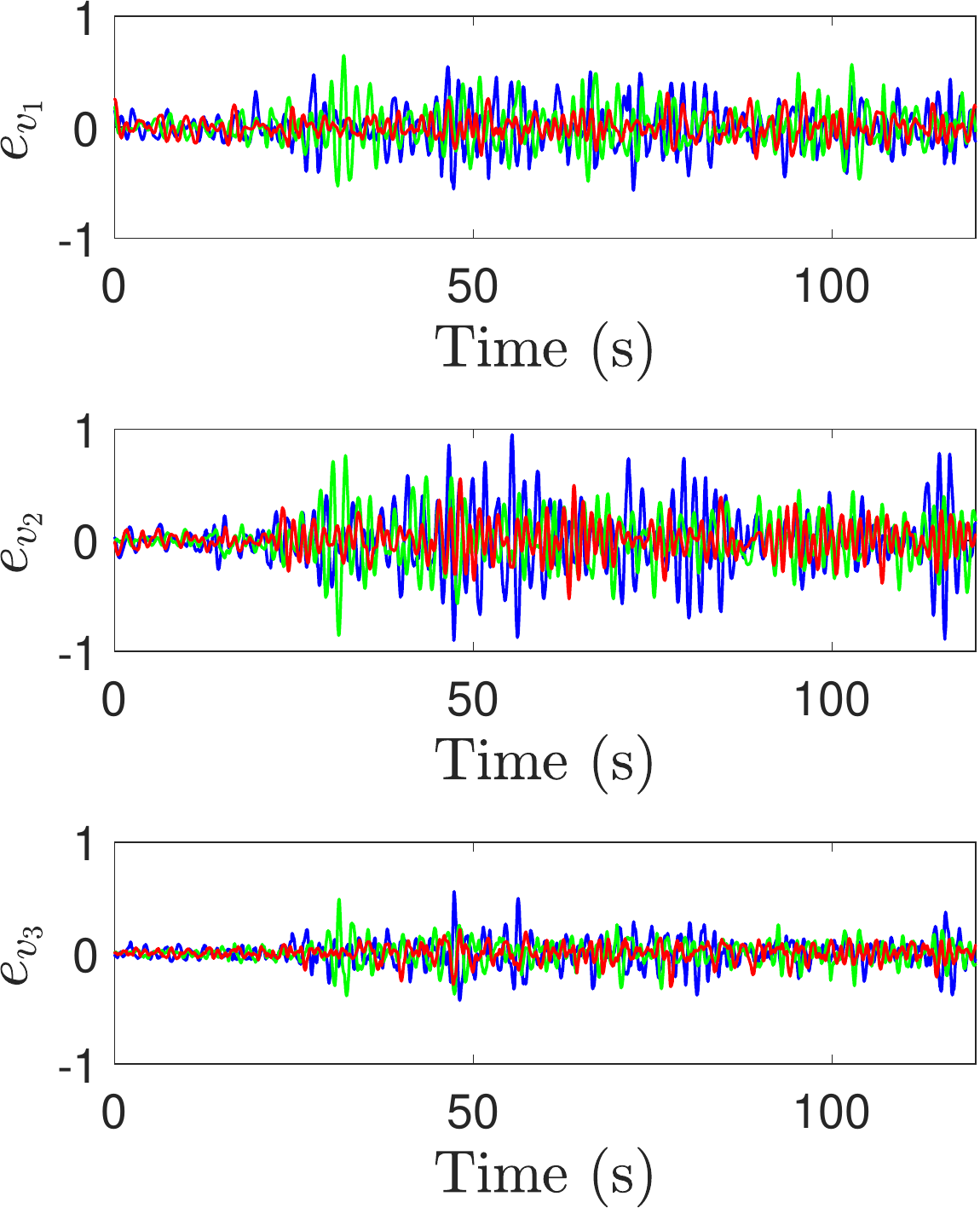}}}
	\centerline{\subfigure[Attitude error]{\includegraphics[width=0.5\columnwidth]{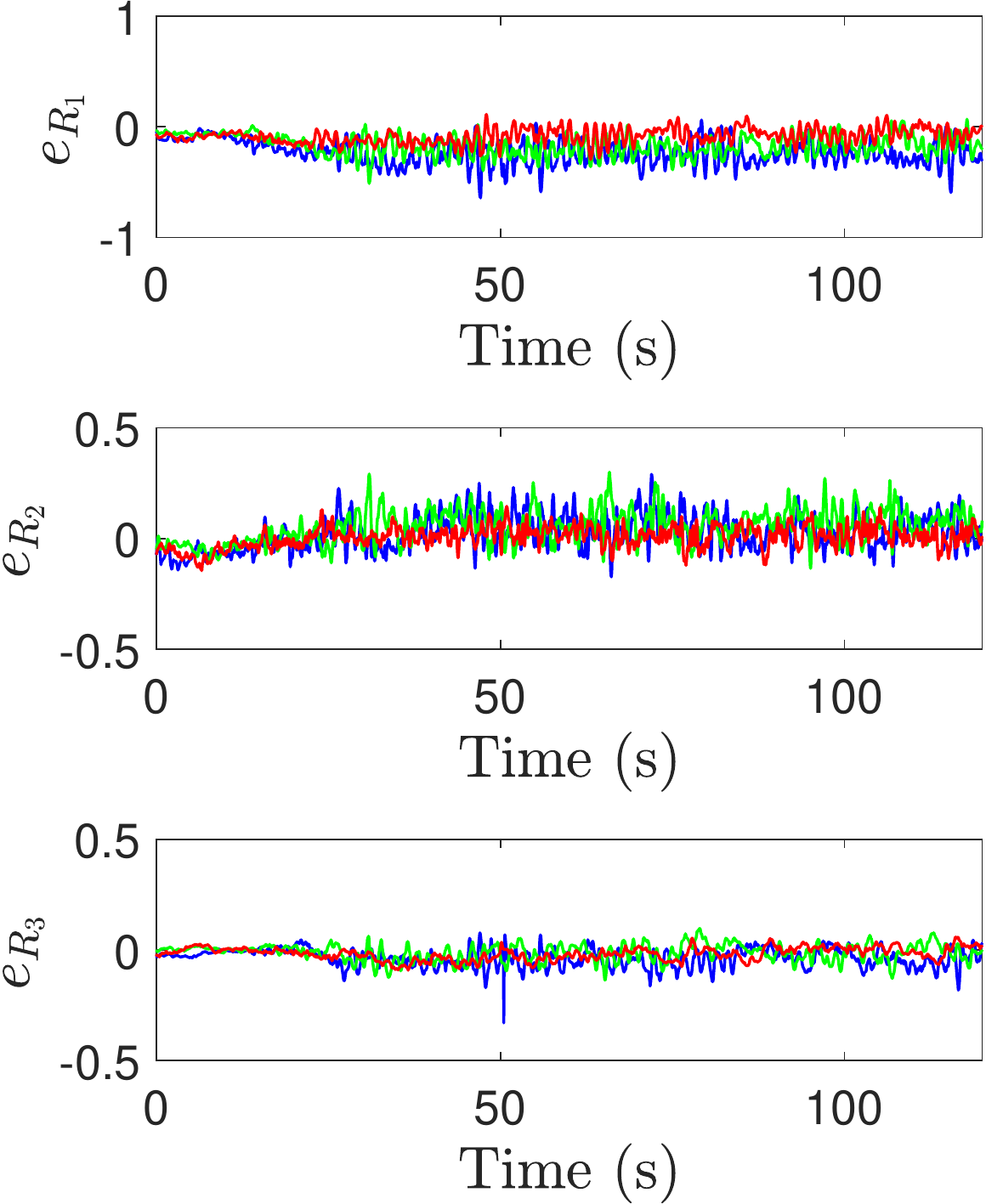}}
		\hfill
		\subfigure[Angular velocity error ($\frac{rad}{s})$)]{\includegraphics[width=0.5\columnwidth]{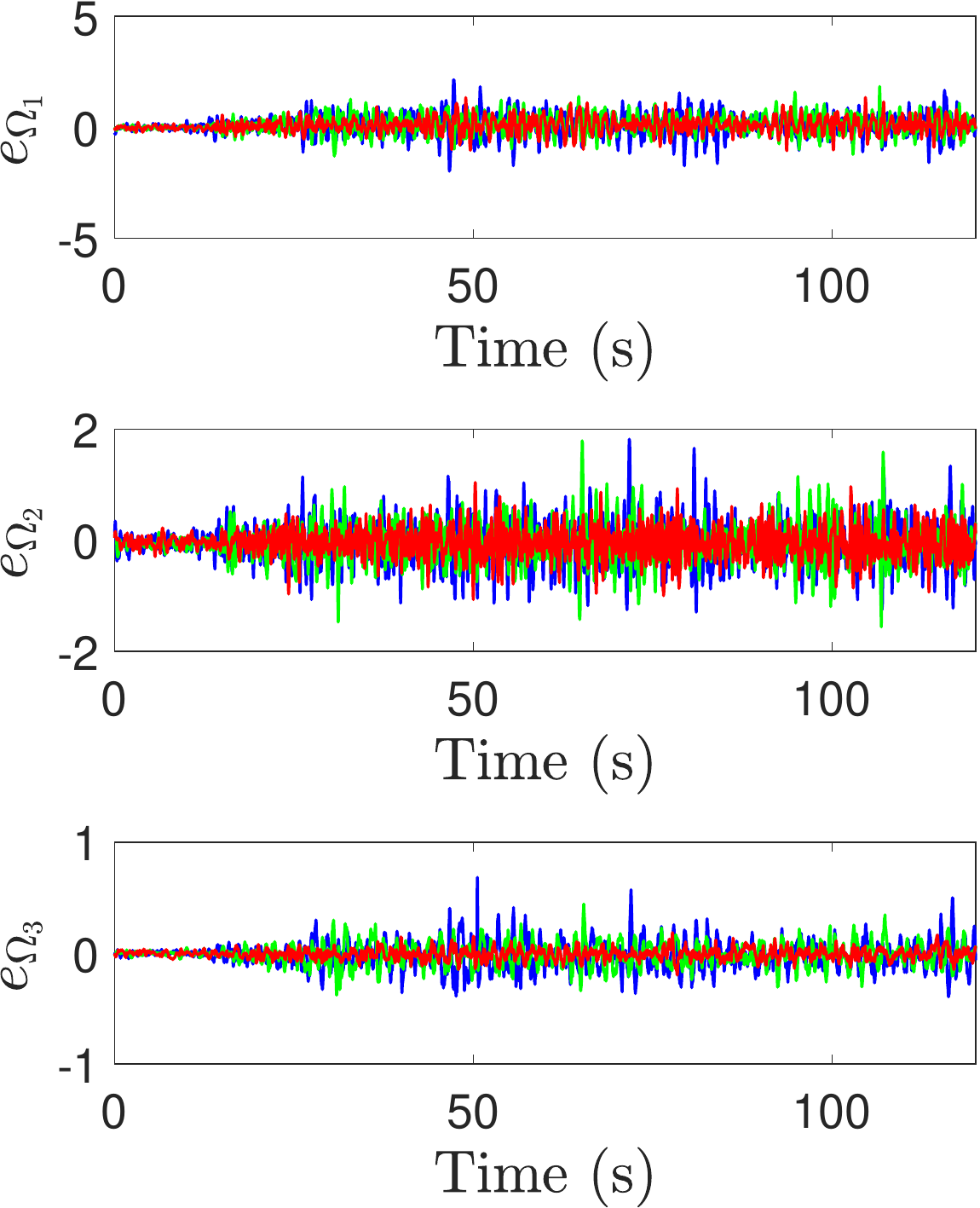}}}
        \caption{Hovering flight (tracking errors), blue:without disturbance rejection~\cite{LeeLeoPICDC10}, green:PID~\cite{Goodarzi_Lee_13}, red: adaptive controller}
	\label{fig:hoverflightNNvsPDvsPID_eXeV}  
\end{figure}

\begin{figure}
	\centerline{
		\subfigure[$\Delta_1$, green:PID, red: adaptive with wind]{\includegraphics[width=0.5\columnwidth]{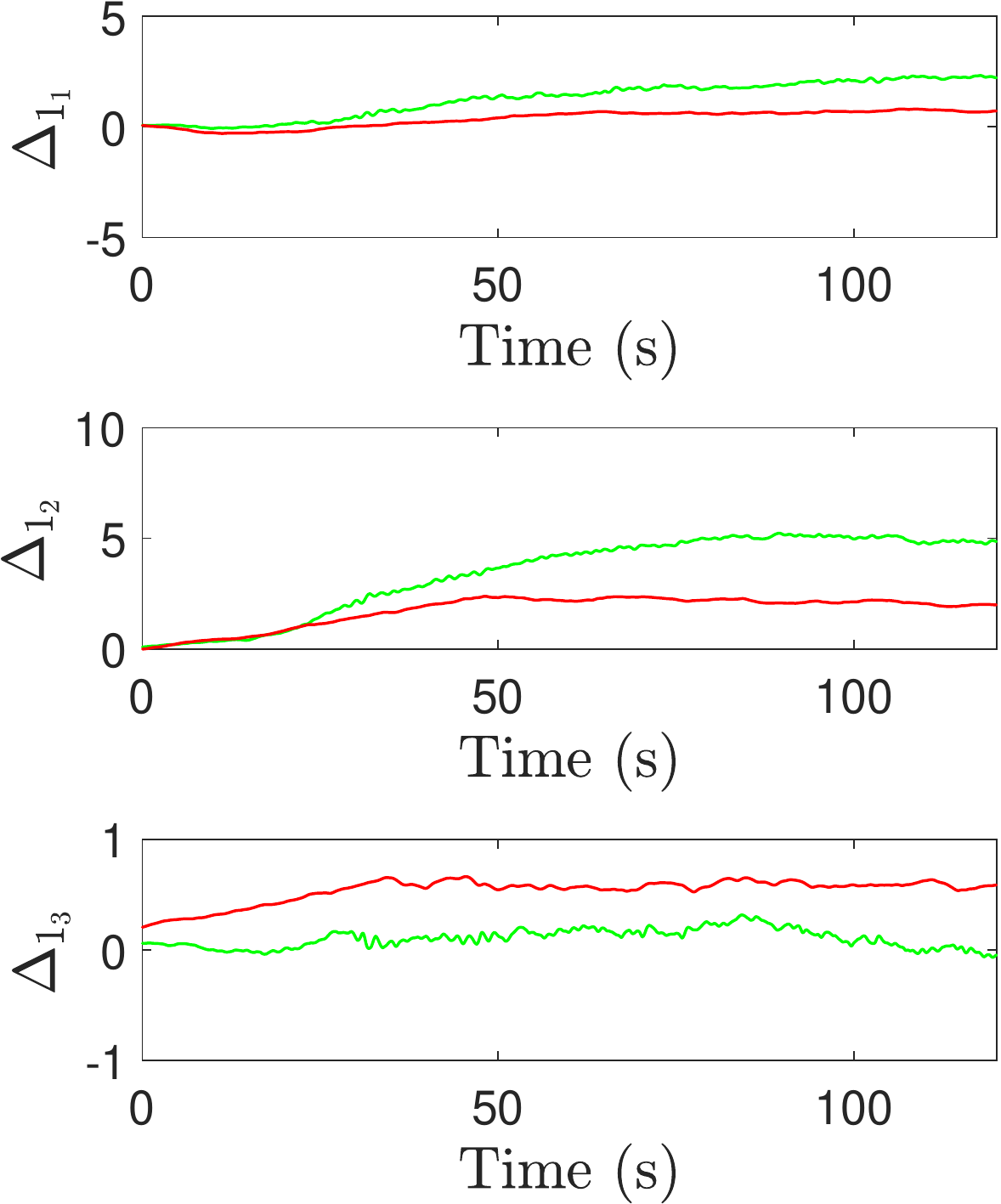}}
		\hfill
		\subfigure[$\Delta_2$, green:PID, red: adaptive with wind]{\includegraphics[width=0.5\columnwidth]{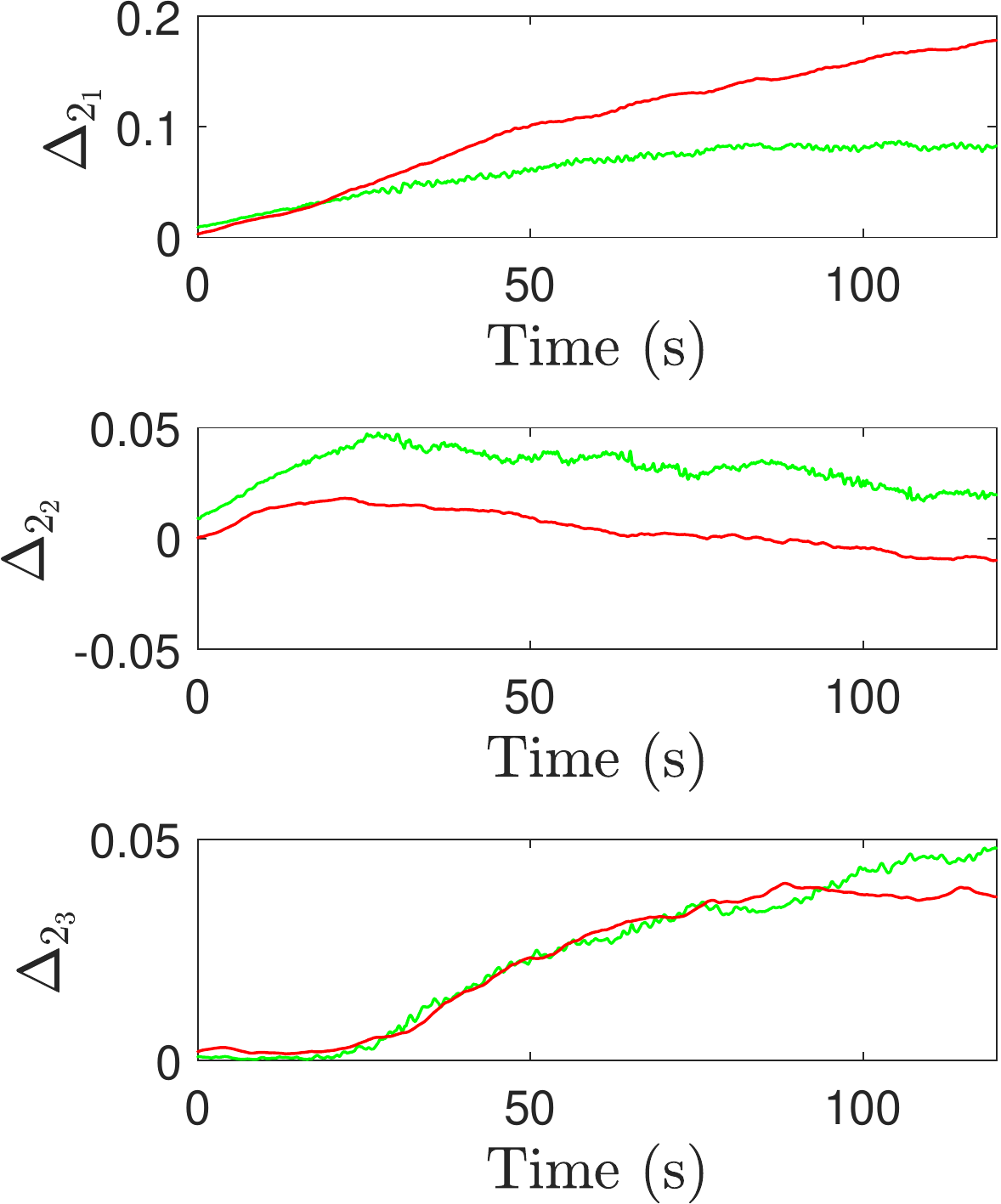}}
	}
	\centerline{
		\subfigure[Thrust (N), blue: without disturbance rejection~\cite{LeeLeoPICDC10} with wind, green:PID, red: adaptive controller with wind]{\includegraphics[width=0.5\columnwidth]{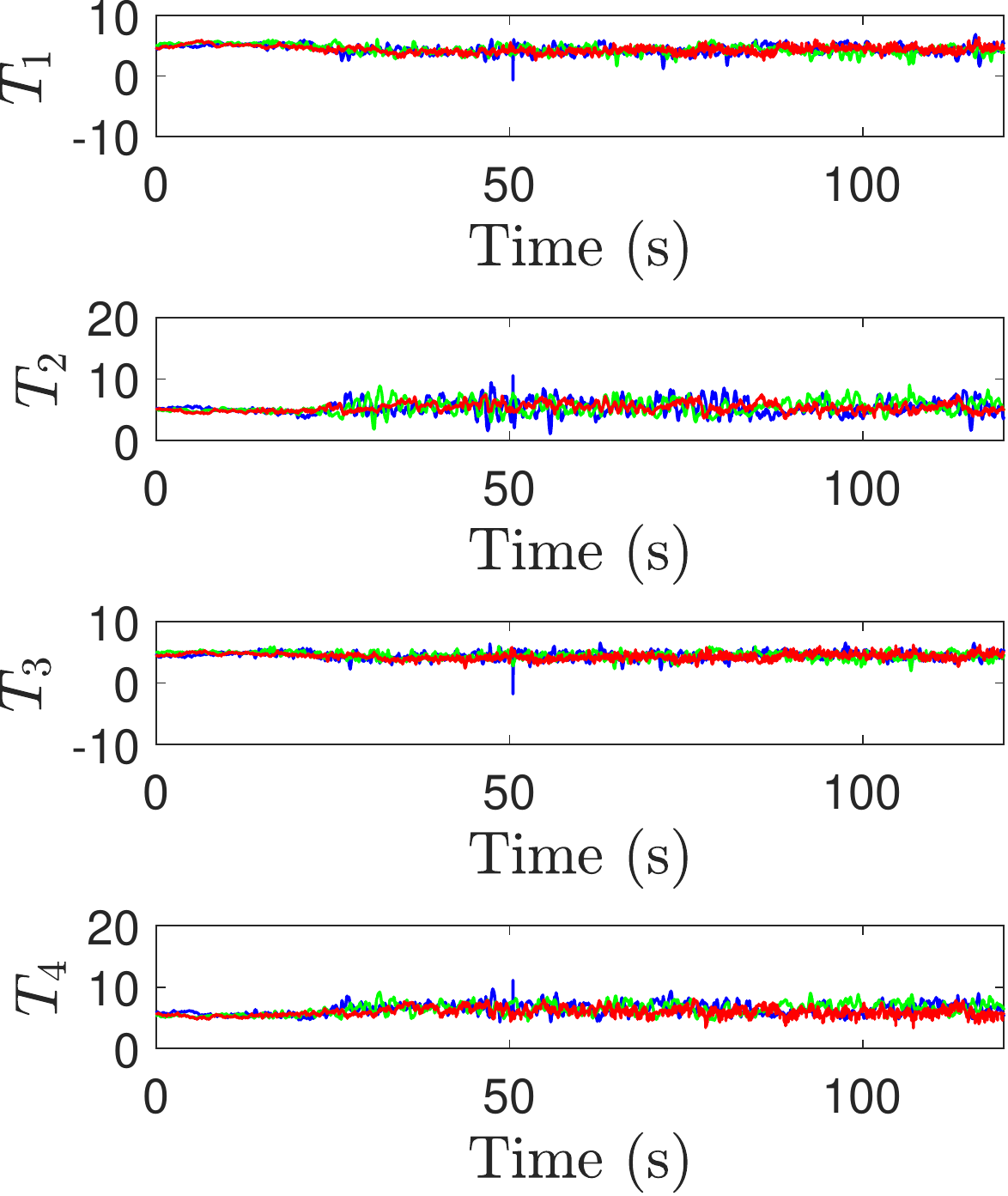}}
	}
    \caption{Hovering flight (adaptive term, and thrusts), blue:without disturbance rejection~\cite{LeeLeoPICDC10}, green:PID~\cite{Goodarzi_Lee_13}, red: adaptive controller}
		\label{fig:hoverflightDeltaNNvsPDvsPID}
\end{figure}

\begin{figure}
	\centering 
	\includegraphics[width=1\columnwidth]{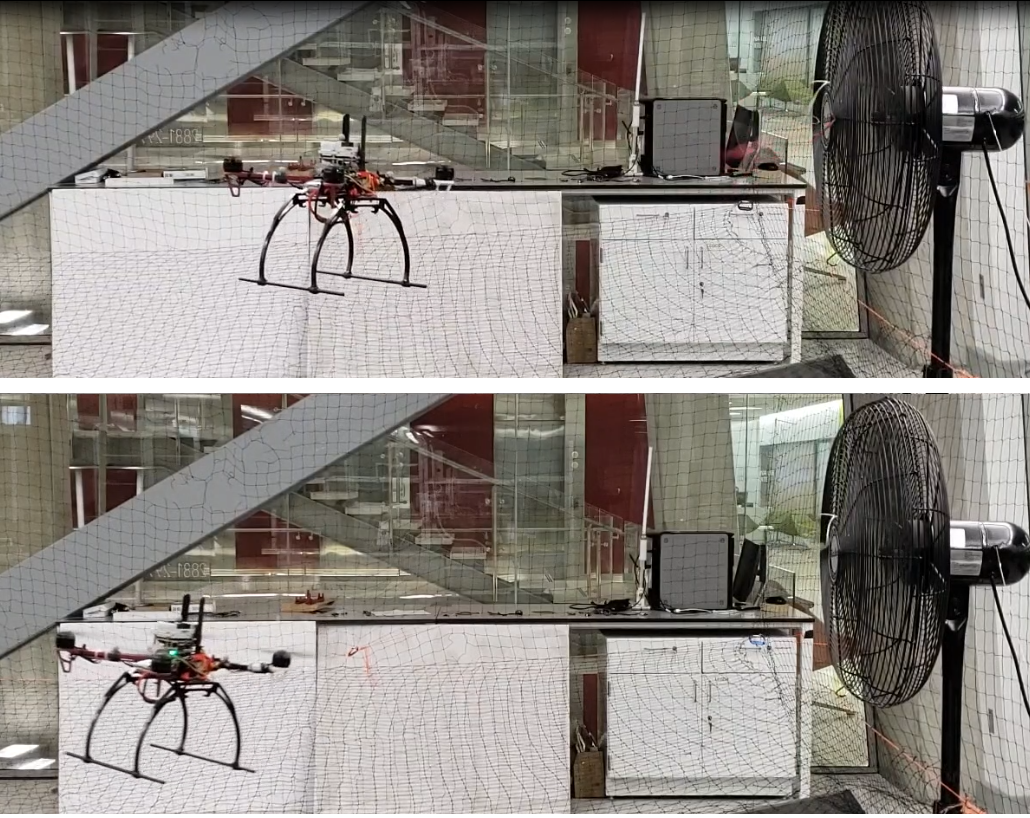} 
    \caption{Hovering flight (snapshot), top:adaptive controller, bottom:without disturbance rejection~\cite{LeeLeoPICDC10}.} \label{fig:photo_hover_flight}
\end{figure}

\subsection{Geometric Adaptive Control for Position Tracking}\label{sec:trac_flight}

In this section, the performance of the adaptive controller for the trajectory tracking is observed.
The desired trajectory is given by
\begin{gather}
x_d(t)=\begin{bmatrix} -0.67\\0.2 - 1.2 \cos(\frac{\pi t}{12} )\\-1.57\end{bmatrix},\quad b_{1}(t_k)=\begin{bmatrix} 1\\0\\0\end{bmatrix},\label{eqn:towardfan_xd}
\end{gather}
which is a sinusoidal oscillation along the second inertial axis. 

Figures \ref{fig:torwardfan_NNvsPDvsPID_xv}--\ref{fig:torwardfan_Delta} show the experimental data. 
The black lines show the desired trajectories.
The trajectories without disturbance rejection are plotted in blue, and those with the proposed adaptive controller in red.
It is illustrated that the proposed controller yields smaller tracking errors without excessive rotor thrust. 

\begin{figure}
	\centerline{
		\subfigure[Position ($m$)]{\includegraphics[width=0.5\columnwidth]{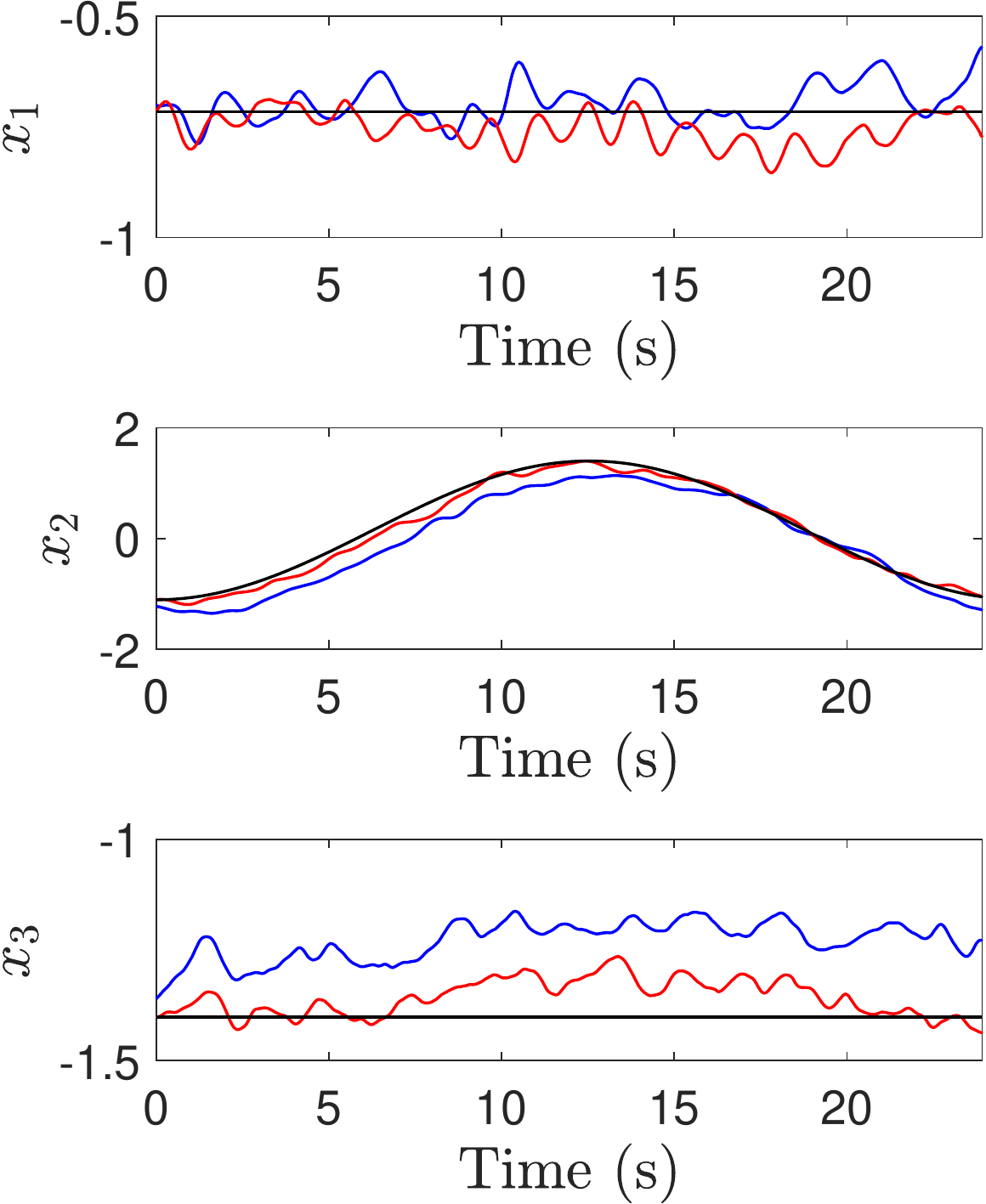}}
		\hfill
		\subfigure[Velocity ($\si{\meter \per \second}$)]{\includegraphics[width=0.5\columnwidth]{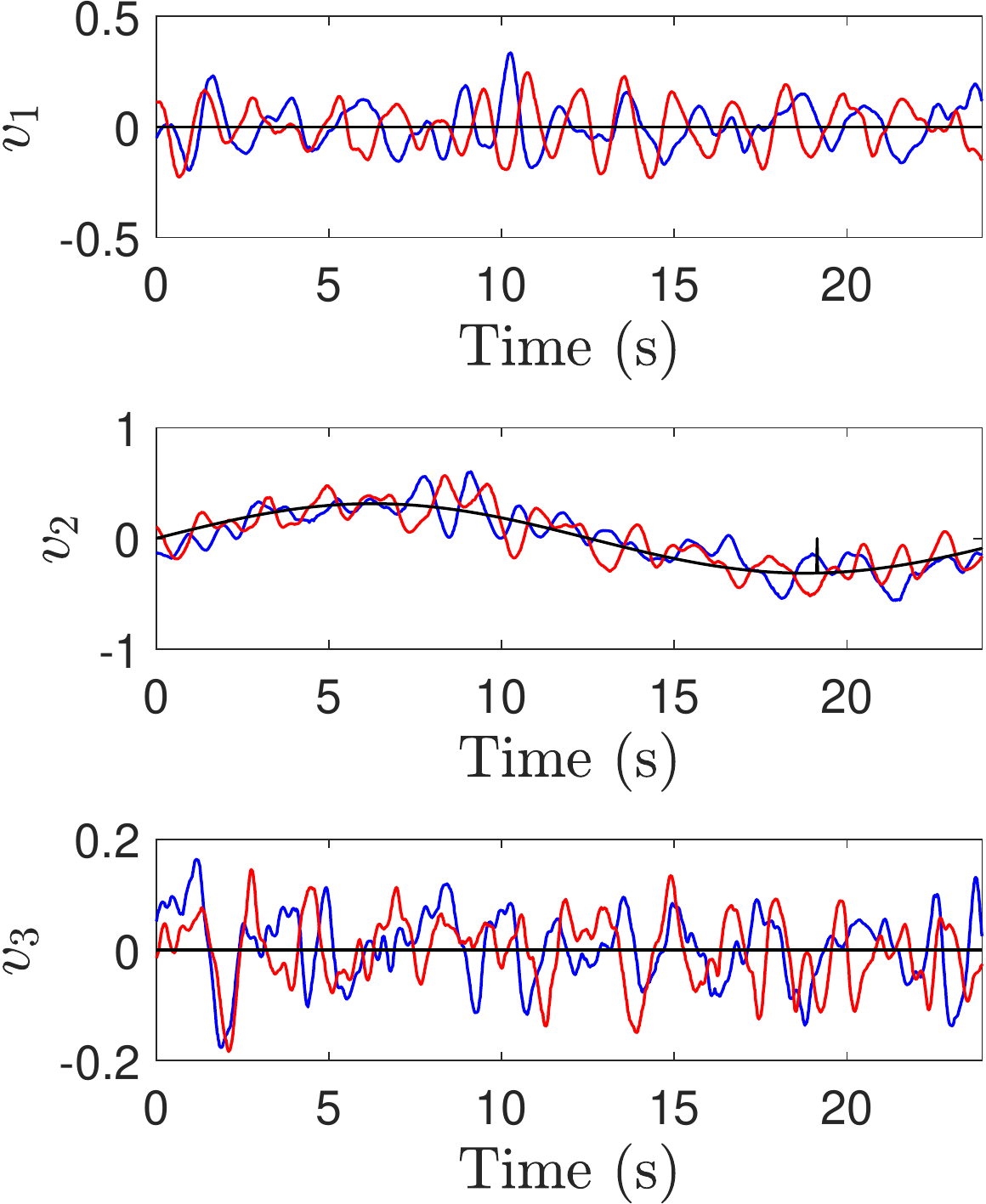}}
	}
	\centerline{
		\subfigure[Attitude]{\includegraphics[width=1\columnwidth]{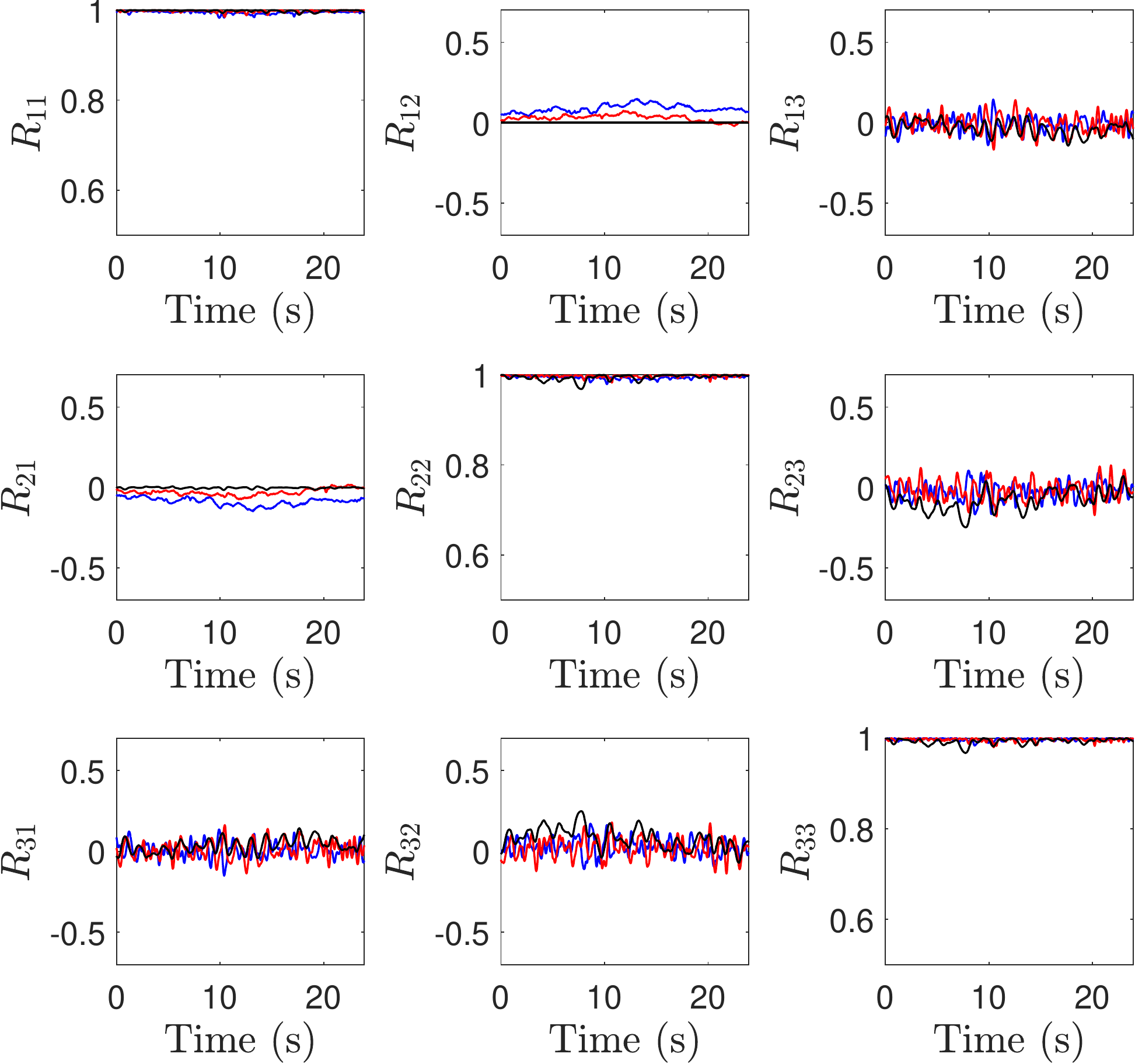}}
	}
    \caption{Position tracking (position, velocity, rotation matrix) black:desired, blue:without disturbance rejection~\cite{LeeLeoPICDC10}, red: adaptive controller}
	\label{fig:torwardfan_NNvsPDvsPID_xv} 
\end{figure}

\begin{figure}
	\centerline{
		\subfigure[Position error ($m$)]{\includegraphics[width=0.5\columnwidth]{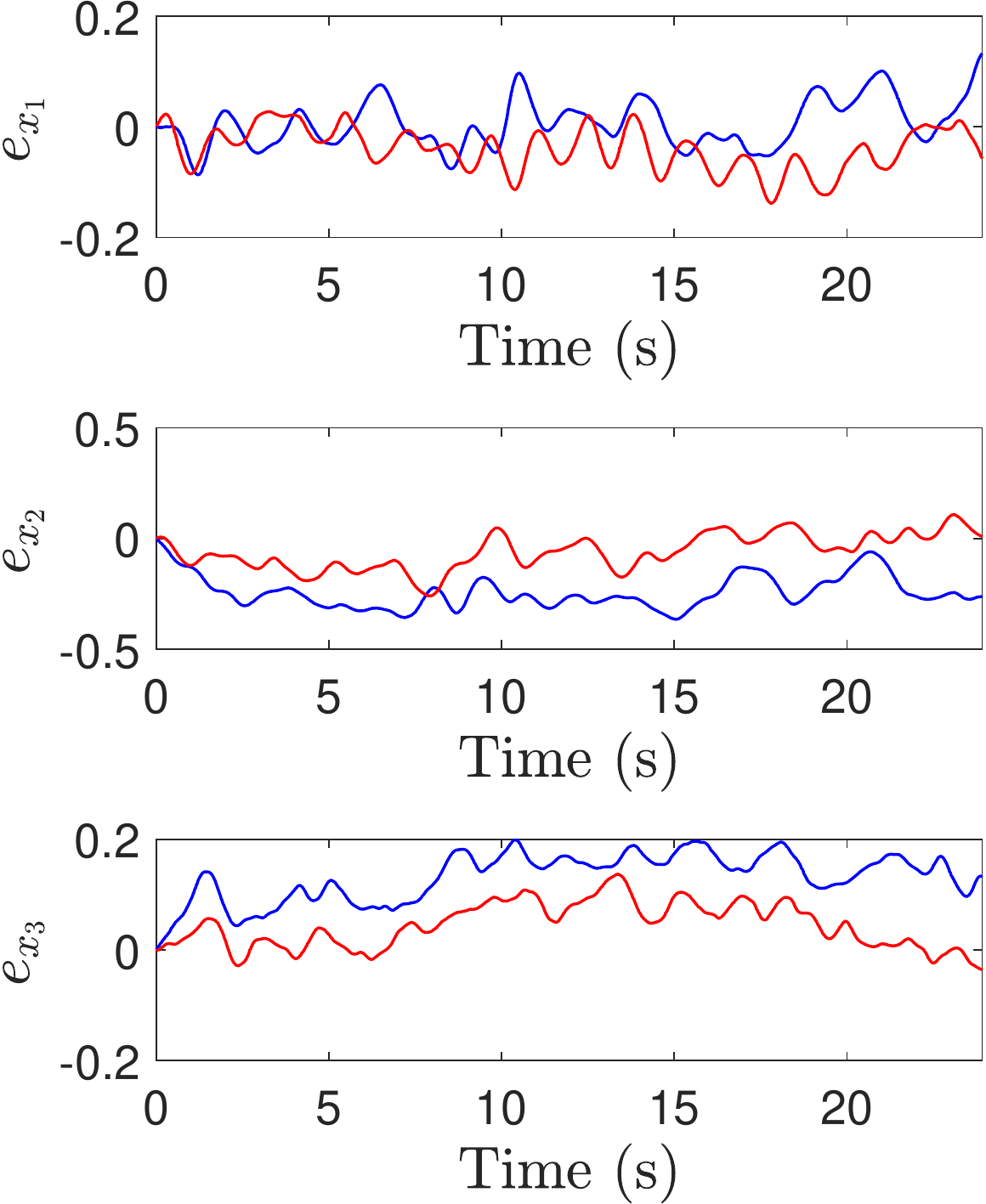}}
		\hfill
		\subfigure[Velocity error ($\si{\meter \per \second}$)]{\includegraphics[width=0.5\columnwidth]{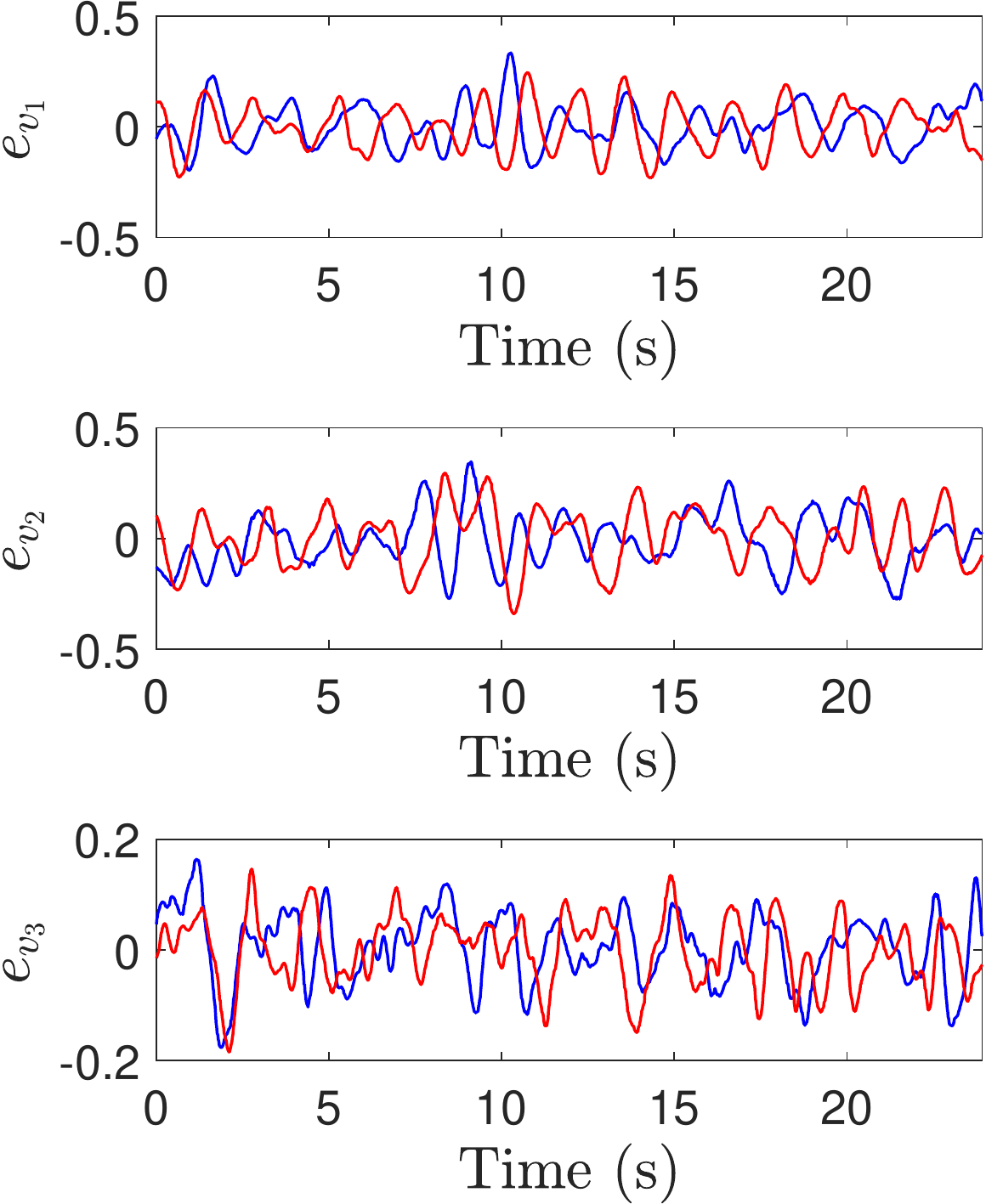}}
	}
	\centerline{
		\subfigure[Attitude error]{\includegraphics[width=0.5\columnwidth]{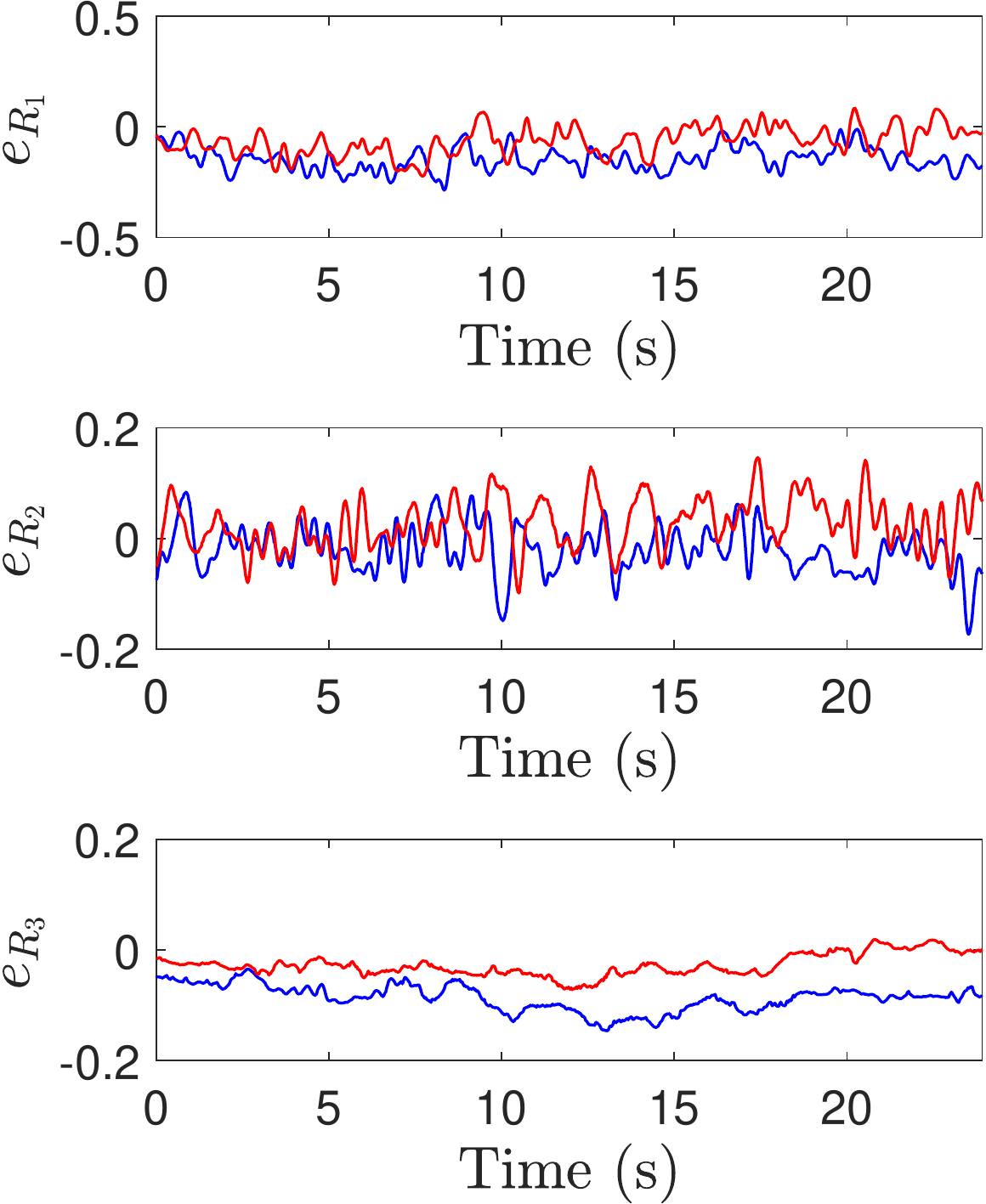}}
		\hfill
		\subfigure[Angular velocity error ($ \si{\radian \per \second}$]{\includegraphics[width=0.5\columnwidth]{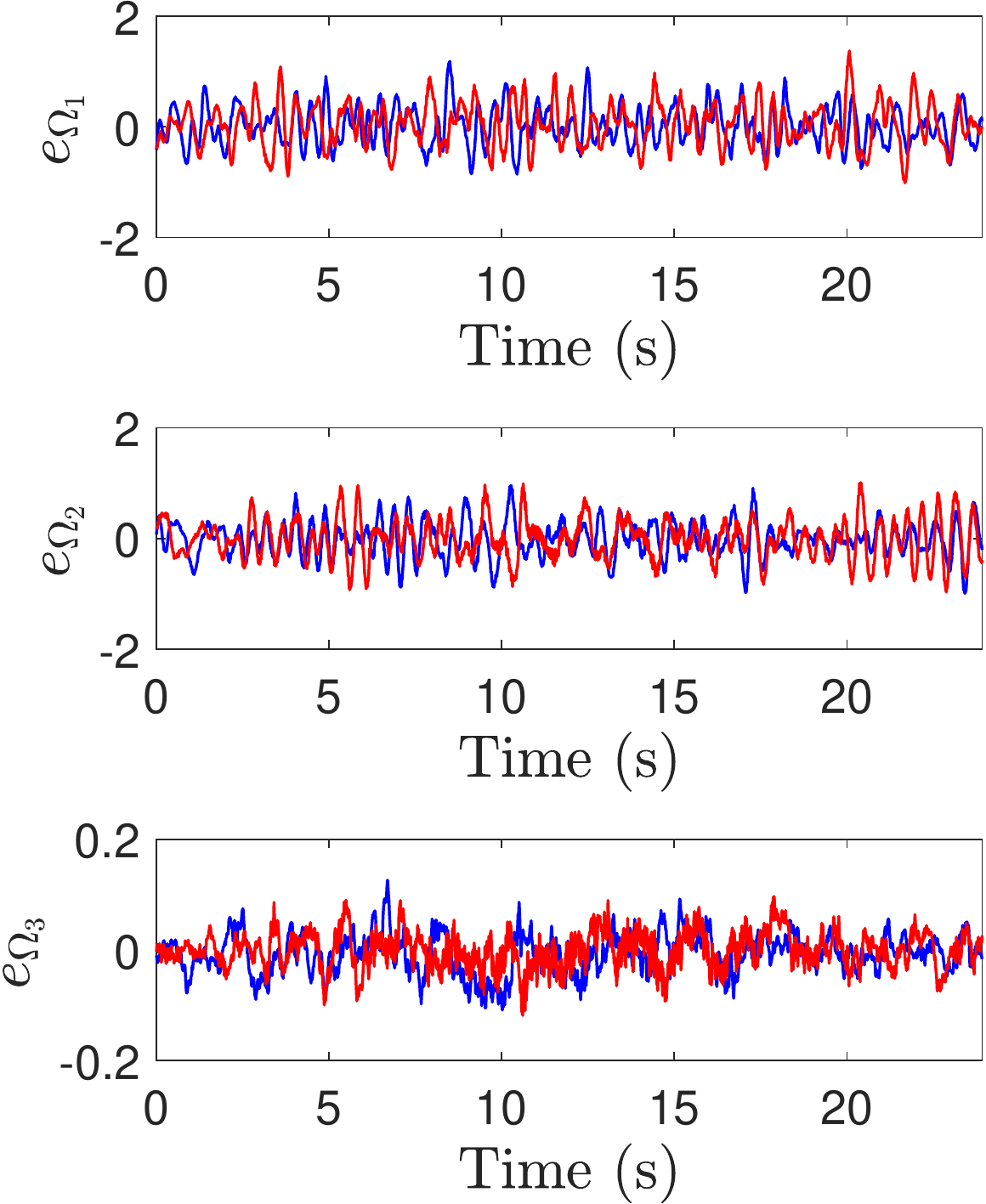}}
	}
    \caption{Position tracking (tracking errors), blue:without disturbance rejection~\cite{LeeLeoPICDC10}, red: adaptive controller}
	\label{fig:torwardfan_NNvsPDvsPID_eXeV} 
\end{figure}

\begin{figure}
	\centerline{
		\subfigure[$\Delta_1$ for position]{\includegraphics[width=0.5\columnwidth]{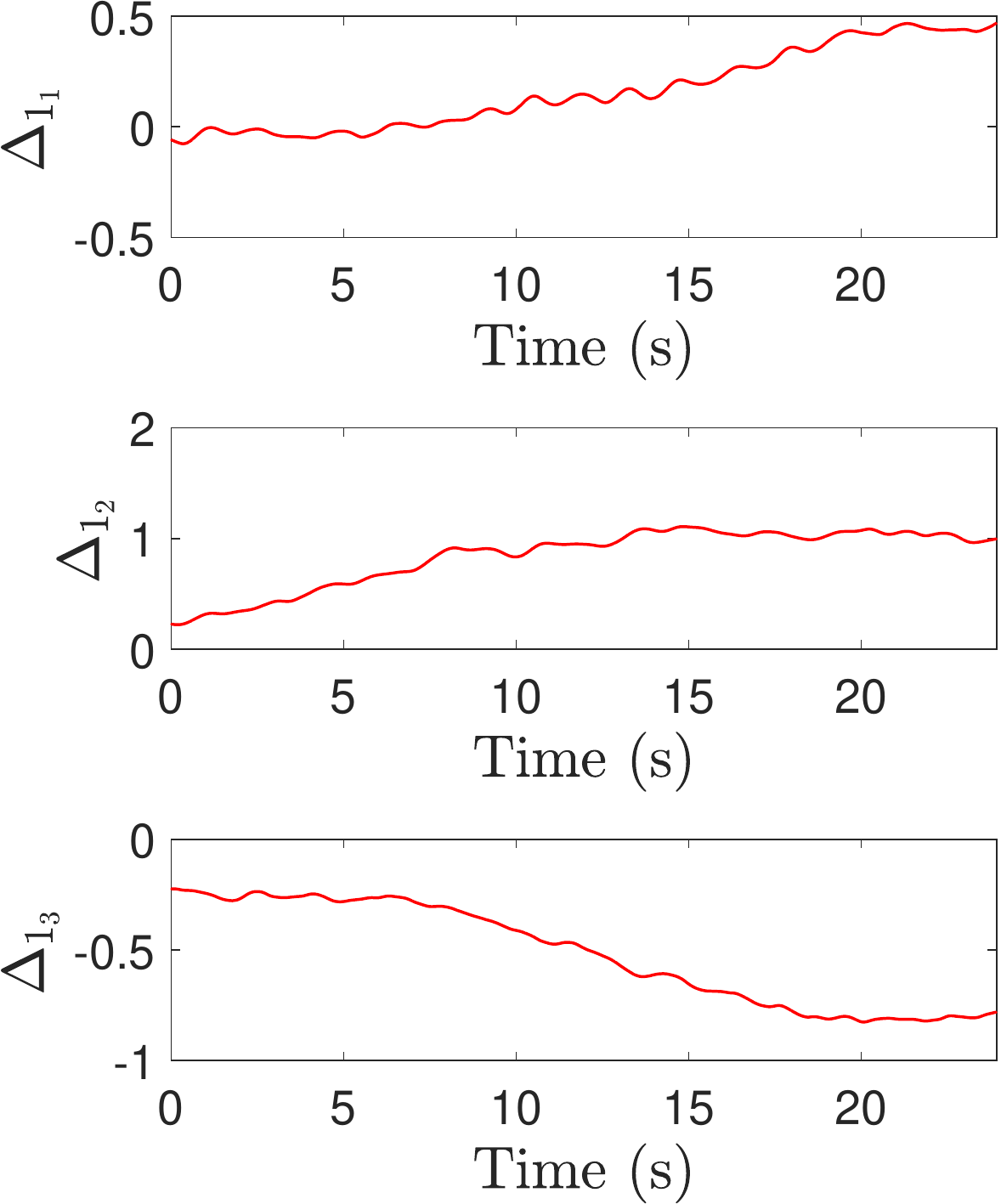}}
		\hfill
		\subfigure[$\Delta_2$ for attitude]{\includegraphics[width=0.5\columnwidth]{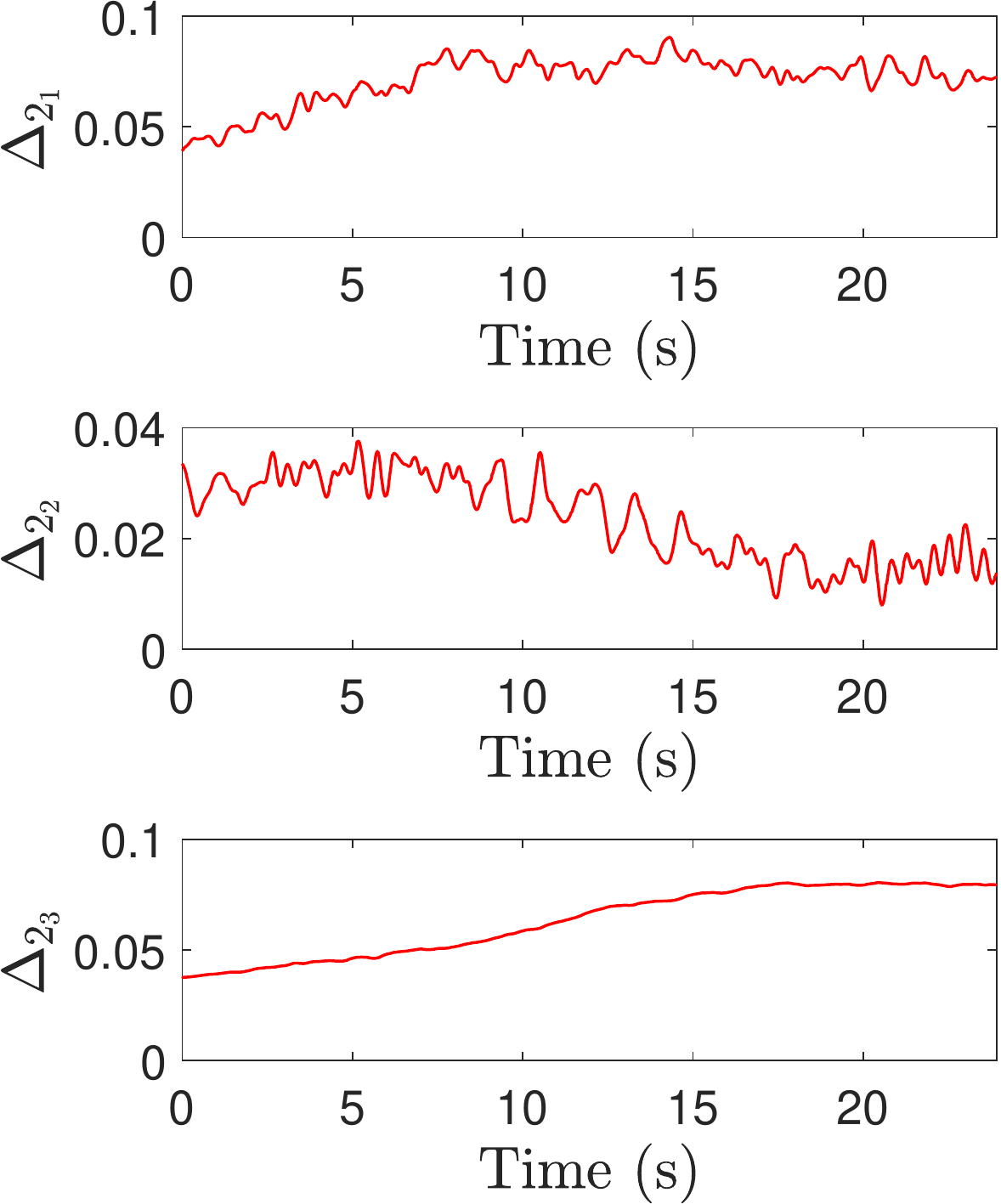}}
	}
	\centerline{
		\subfigure[Thrust ($\si{\newton}$), blue:without disturbance rejection~\cite{LeeLeoPICDC10} with wind, red: adaptive controller with wind]{\includegraphics[width=0.5\columnwidth]{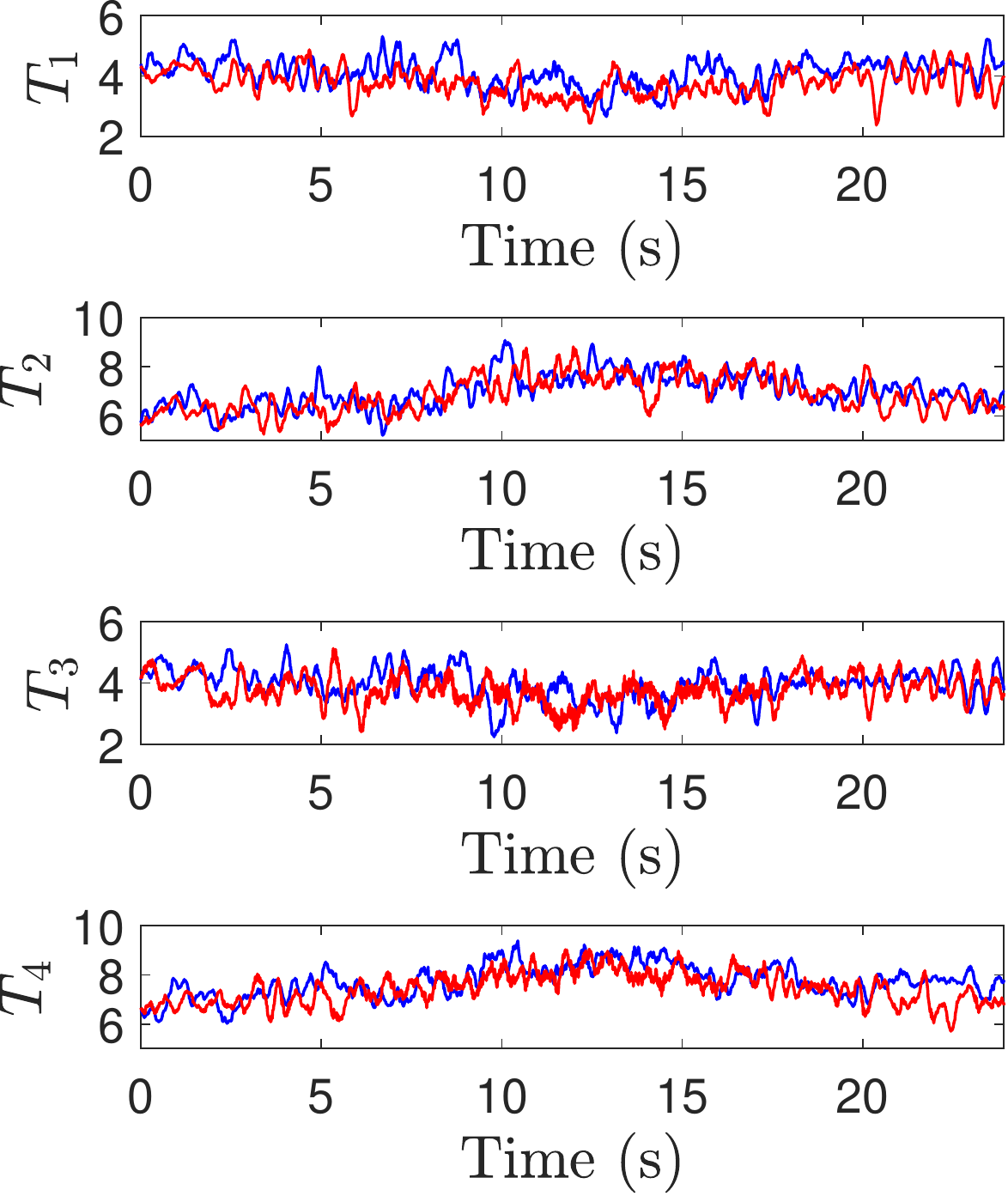}}
	}
    \caption{Position tracking (adaptive terms and thrust), blue:without disturbance rejection~\cite{LeeLeoPICDC10}, red: adaptive controller}
	\label{fig:torwardfan_Delta} 
\end{figure}

\subsection{Geometric Adaptive Control for Backflip}

The illustrate the performance of the proposed control system through an agile maneuver, here we present experimental results for a backflip maneuver. 

The desired trajectory is defined in the three sequences, including taking-off, backflip, and hovering. 
First, the quadrotor takes off to reach the desired upward velocity from $t_0=0 \si{\second}$ to $t_1=2.20\si{\second}$ as follows.
\begin{gather}
    x_d(t)=x_0+\frac{a t^2}{2}\begin{bmatrix}0\\0\\ 1 \end{bmatrix},\quad b_{1_d}=\begin{bmatrix}1\\0\\0\end{bmatrix},
\end{gather}
where $x_0=[-0.22,\,0.47,\,-0.50]^T$, $a = -0.50$, and the quadrotor is controlled using \refeqn{f}--\refeqn{A}. 

In the next step, the attitude is controlled with \refeqn{Mc} to rotate the quadrotor by $360^\circ$ along the $b_{1_d}=e_1$ axis.
The desired attitude trajectory is chosen as
\begin{gather}
    R_d(t)=\exp(\theta_d(t) \hat{b}_{1_{d}}),
\end{gather}
where the rotation angle is chosen as a second order polynomial of time,
\begin{equation}
    \theta_d(t)=\begin{cases}
\frac{1}{2} \alpha_{m} (t-t_1)^2 & \text{if $t_1< t< t_1+\frac{\delta t}{2} $} \\
&\\
\frac{1}{2}\Delta_t \alpha_{m} (t-t_1) & \\
-\frac{1}{2} \alpha_{m} (t-t_1-\frac{\delta t}{2})^2& \text{if $t_1+\frac{\Delta_t}{2}< t< t_1+\frac{6}{8}\delta t$},
\end{cases}\label{eqn:flip_theta}
\end{equation}
with
\begin{gather}
\alpha_{m} = 60.0,\quad
\Delta_t=\sqrt{\frac{8 \pi}{\alpha_{m}}}.
\end{gather}
The resulting desired angular velocity is 
\begin{equation}
    \Omega_d(t)=\begin{cases}
\alpha_{m} (t-t_1) b_{1_{d}} & \text{if $t_1< t< t_1+\frac{\delta t}{2} $} \\
\alpha_{m} (\Delta_t  + t_1 - t) b_{1_{d}}& \text{if $t_1+\frac{\delta t}{2}< t< t_1+\frac{6}{8}\delta t$},
\end{cases}\label{eqn:flip_wd}
\end{equation}

After backflip, again the quadrotor is controlled using \refeqn{f}--\refeqn{A} to make it hover at a fixed location specified as
\begin{gather}
    x_d(t)=x_0+\frac{a t_1^2}{2}\begin{bmatrix}0\\0\\ 1 \end{bmatrix},\quad b_{1_d}=\begin{bmatrix}1\\0\\0\end{bmatrix}.
\end{gather}

Figures \ref{fig:flip_NNvsPD_xv}--\ref{fig:flip_Delta} show the experimental results. 
The black lines show the desired trajectories.
The trajectories without disturbance rejection are plotted in blue, and with those of the proposed adaptive controller in red. 
The gray lines are to separate the three stages described above. 
The first gray line divides the take-off from the backflip and the second one separates the backflip from the last hovering stage.
For the control system presented in~\cite{LeeLeoPICDC10}, the angular velocity diverges during the backflip stage, resulting in a large attitude tracking error afterwards. 
More specifically, due to wind in $-e_2$ direction, the quadrotor could not  complete a swift rotation during the second step. 
Actually, it rotated only about $180^\circ$ along $e_1$ axis in the second step, and continued the rotation through the third stage, during which the quadrotor fail to regain control and crashes into the floor. 
See Figure \ref{fig:photo_flip_fail} for snapshots. 

In contrast, the proposed geometric adaptive controller with neural network result in a successful backflip maneuver followed by a stable hovering flight, as illustrated in Figure~\ref{fig:photo_flip}. 
It is remarkable that the neural network parameters are adjusted promptly over the short time period of the second backflip stage, to achieve the successful backflip maneuver. 
Such agile maneuver under the effects of wind has not been demonstrated yet\footnote{For the video file of this experiment, visit the FDCL YouTube channel at \url{https://youtu.be/a-DG2PcUu7k} or the experiment section of the FDCL website at \url{http://fdcl.seas.gwu.edu/}.}.

\begin{figure}
	\centerline{
		\subfigure[Position ($m$)]{\includegraphics[width=0.5\columnwidth]{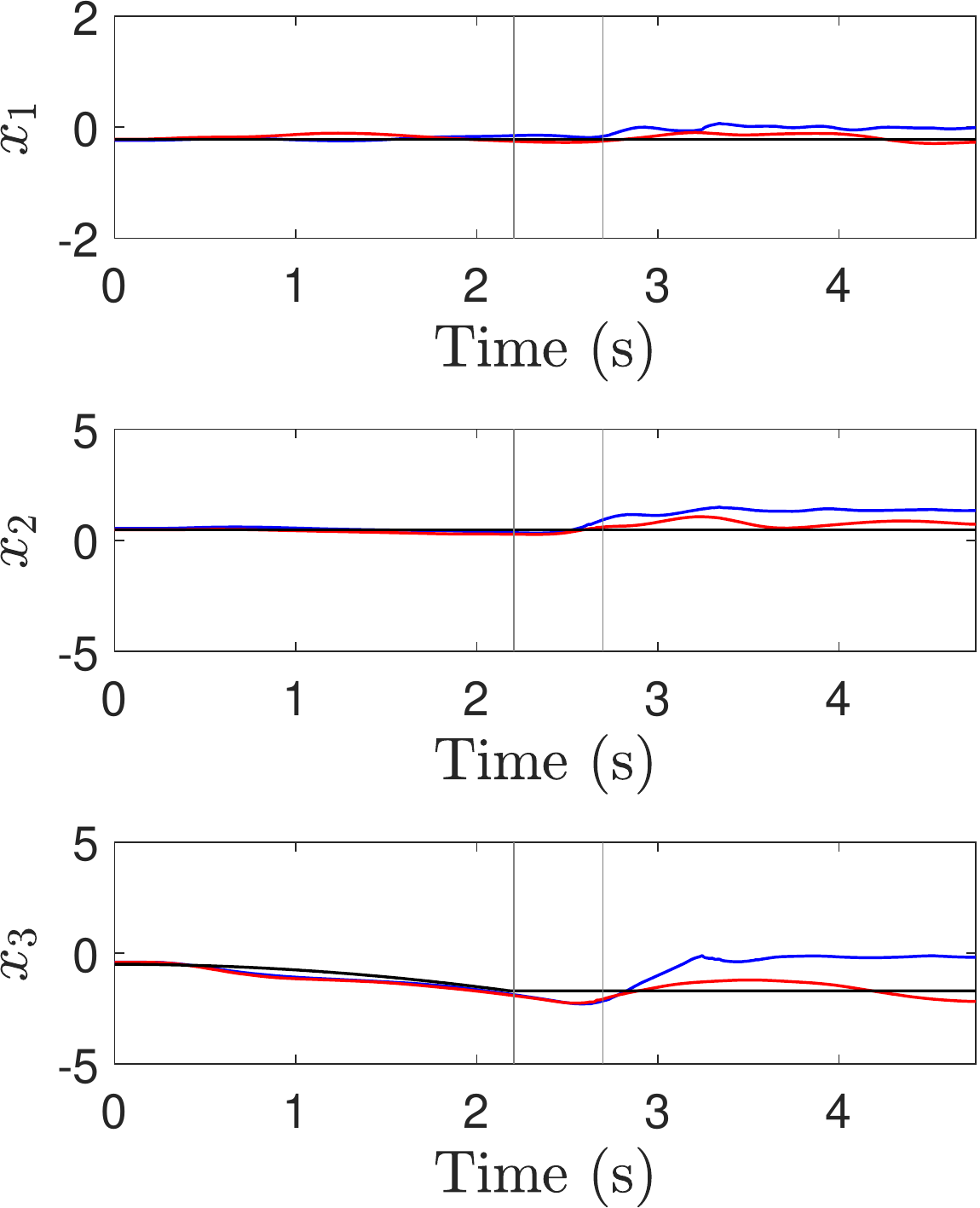}}
		\hfill
		\subfigure[Velocity ($\si{\meter \per \second}$]{\includegraphics[width=0.5\columnwidth]{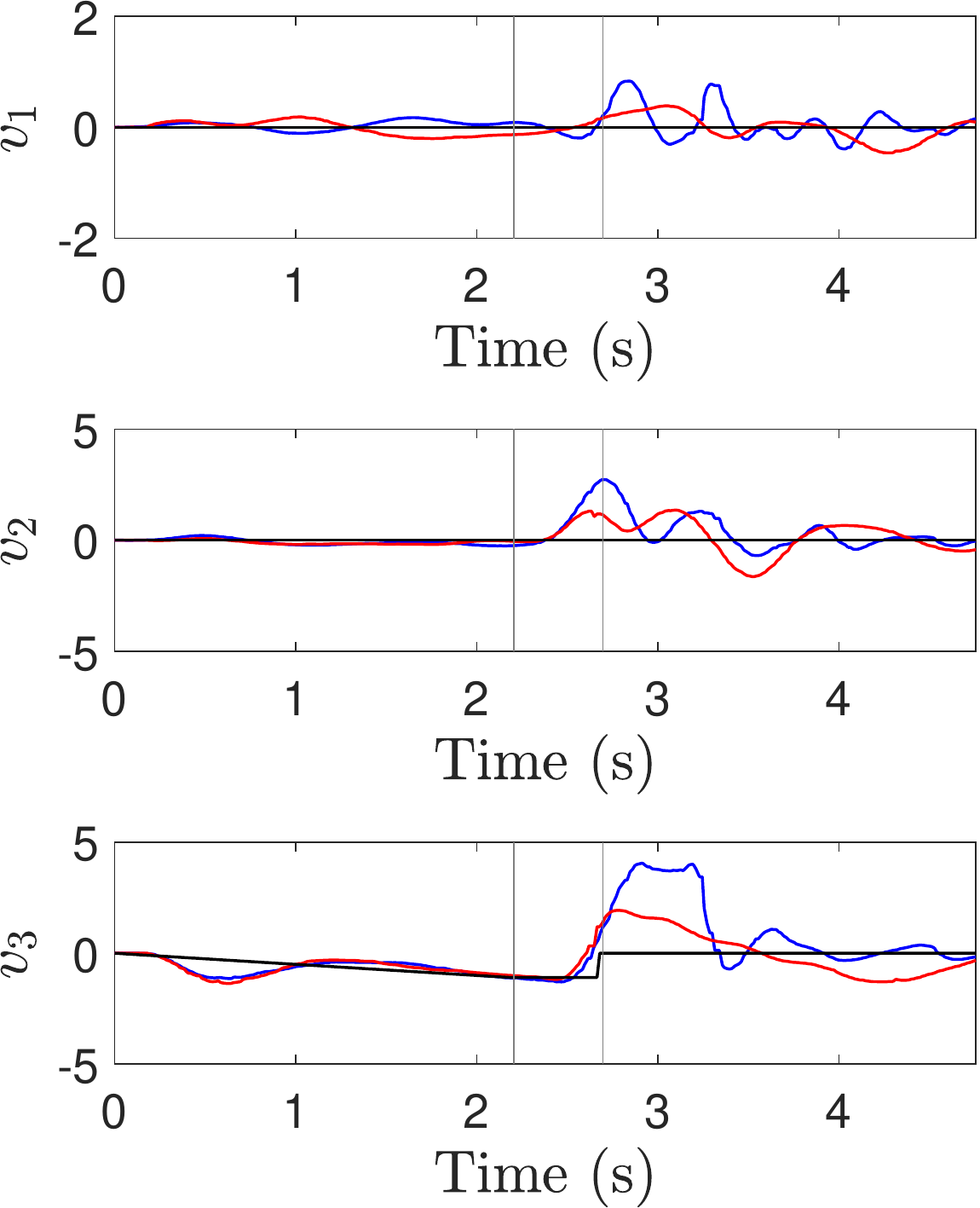}}
	}
	\centerline{
		\subfigure[Attitude]{\includegraphics[width=1\columnwidth]{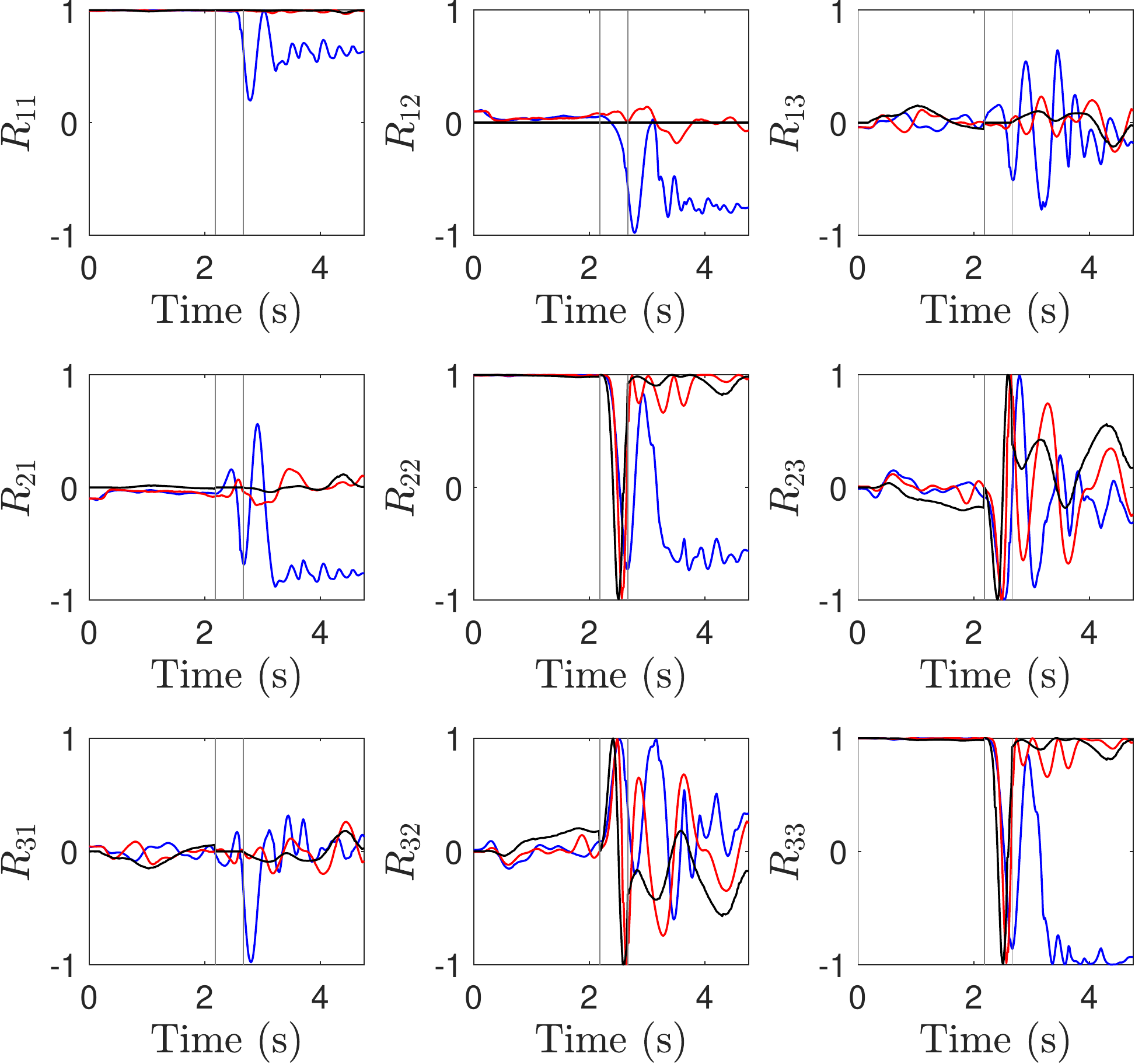}}
	}
    \caption{Backflip (position, velocity, and rotation matrix), black:desired, blue:without disturbance rejection~\cite{LeeLeoPICDC10}, red: adaptive controller}
	\label{fig:flip_NNvsPD_xv} 
\end{figure}

\begin{figure}
	\centerline{
		\subfigure[Position error ($m$)]{\includegraphics[width=0.5\columnwidth]{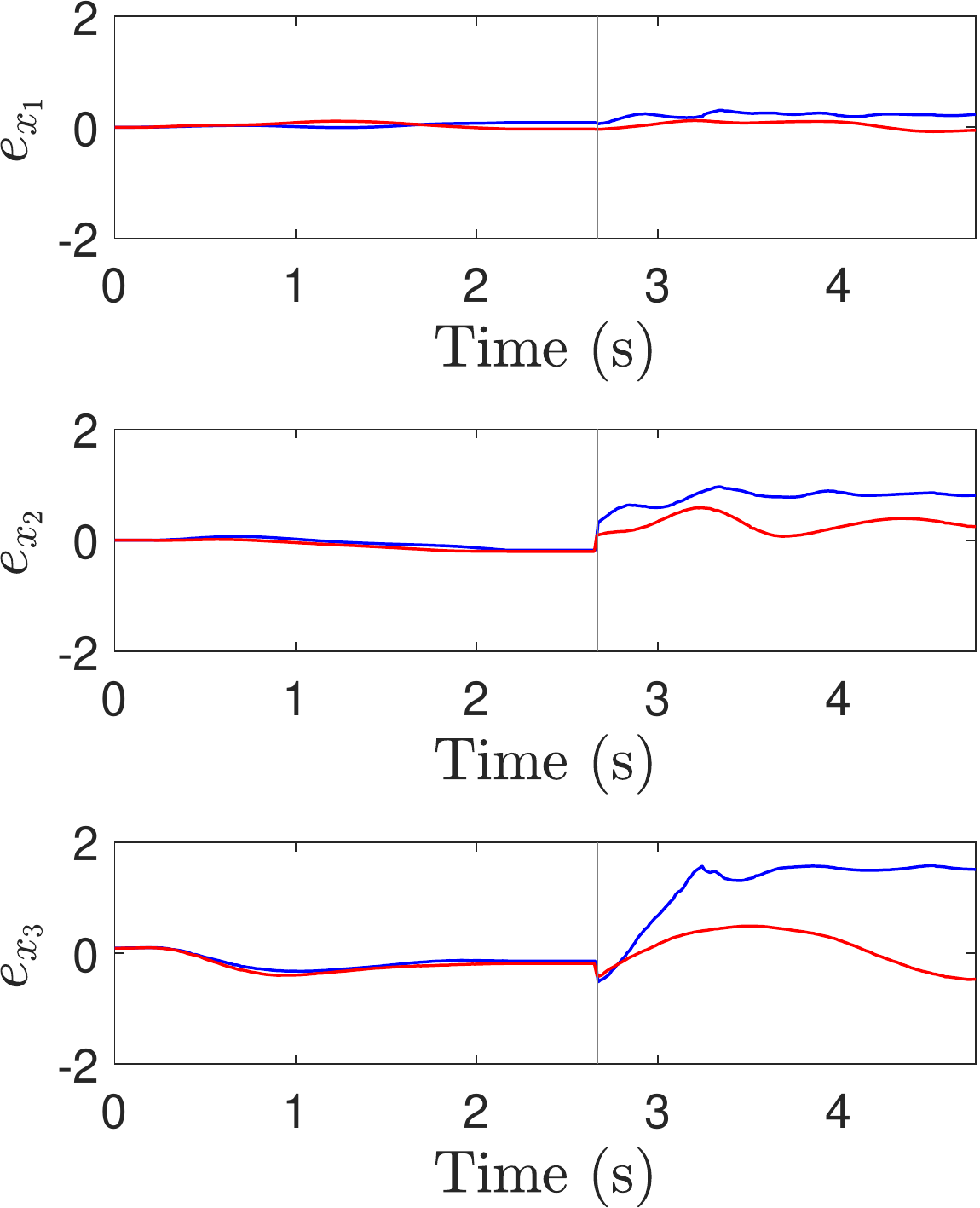}}
		\hfill
		\subfigure[Velocity error ($\si{\meter \per \second})$]{\includegraphics[width=0.5\columnwidth]{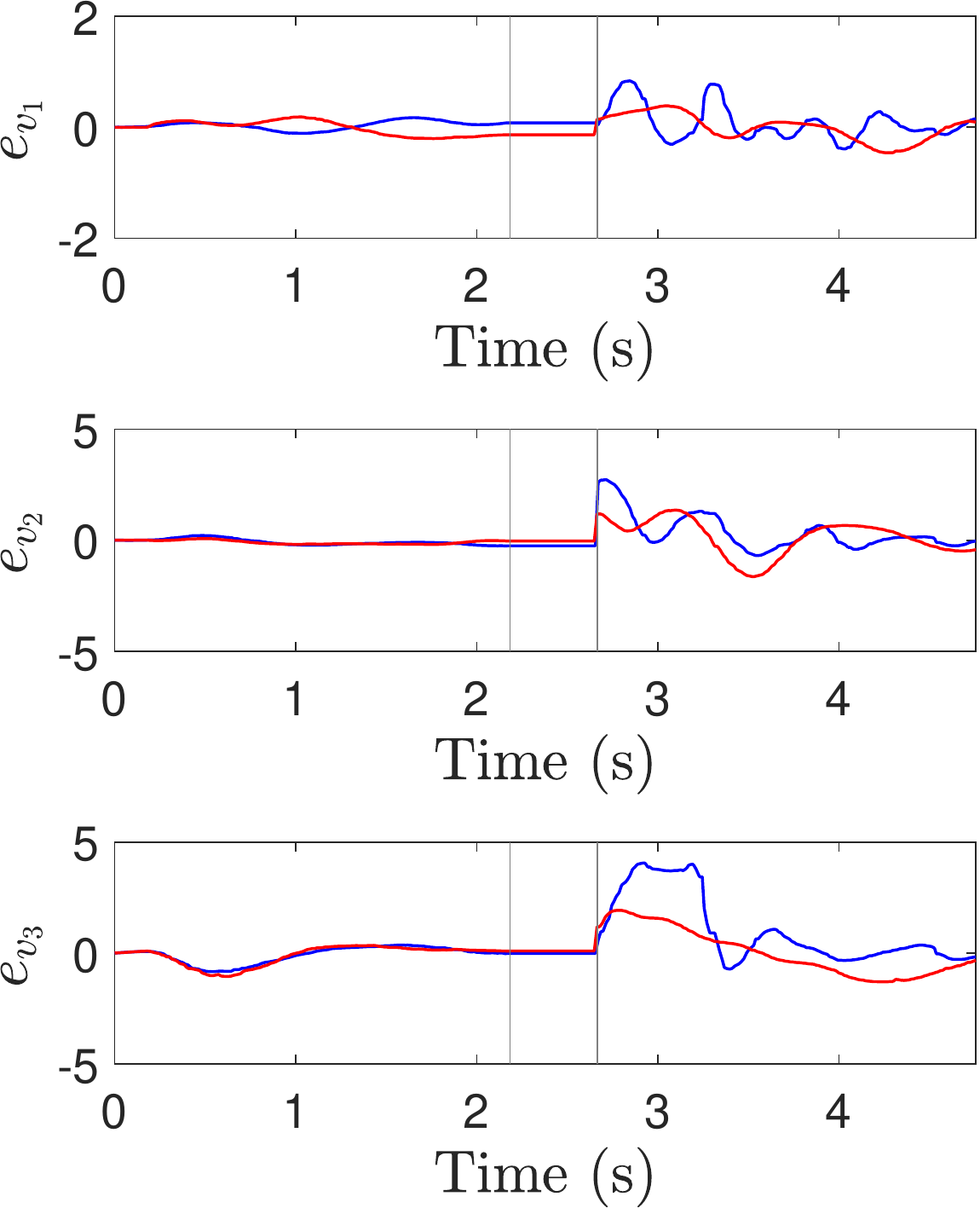}}
	}
	\centerline{
		\subfigure[Attitude error]{\includegraphics[width=0.5\columnwidth]{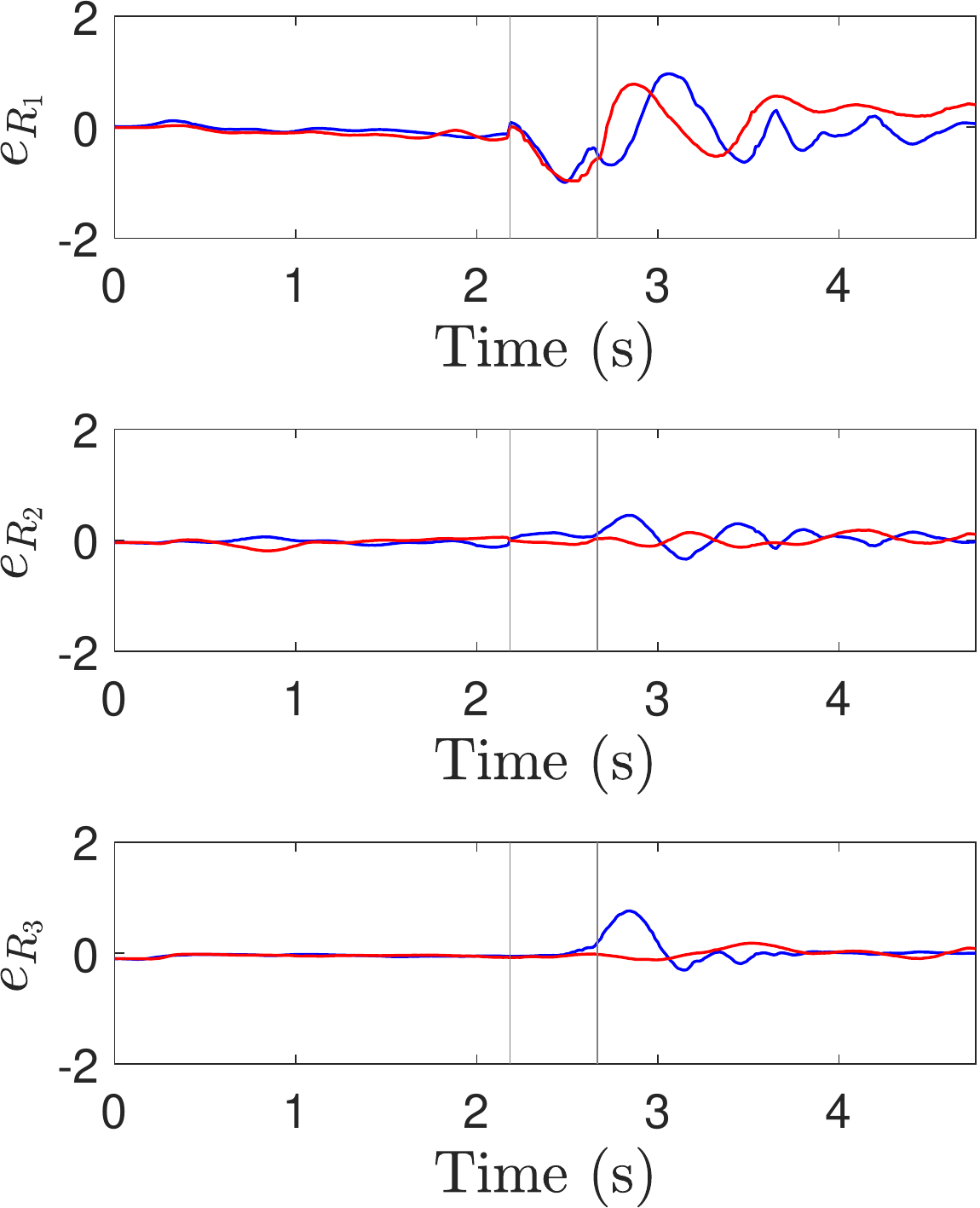}}
		\hfill
		\subfigure[Angular velocity error ($ \si{\radian \per \second}$]{\includegraphics[width=0.5\columnwidth]{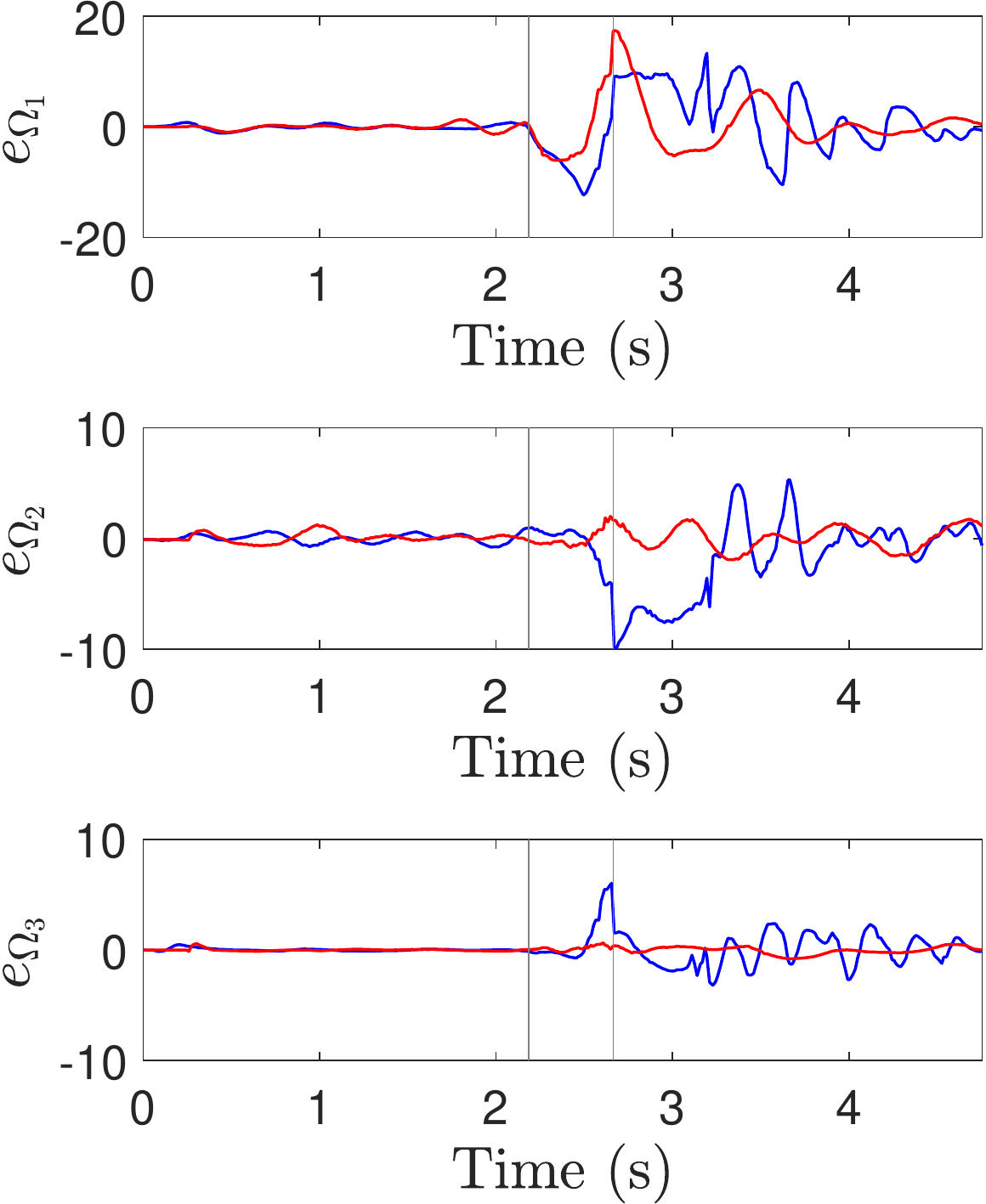}}
	}
    \caption{Backflip (tracking errors), blue:without disturbance rejection~\cite{LeeLeoPICDC10}, red: adaptive controller}
	\label{fig:flip_NNvsPD_eXeV}
\end{figure}

\begin{figure}
	\centerline{
		\subfigure[$\Delta_1$ for position]{\includegraphics[width=0.5\columnwidth]{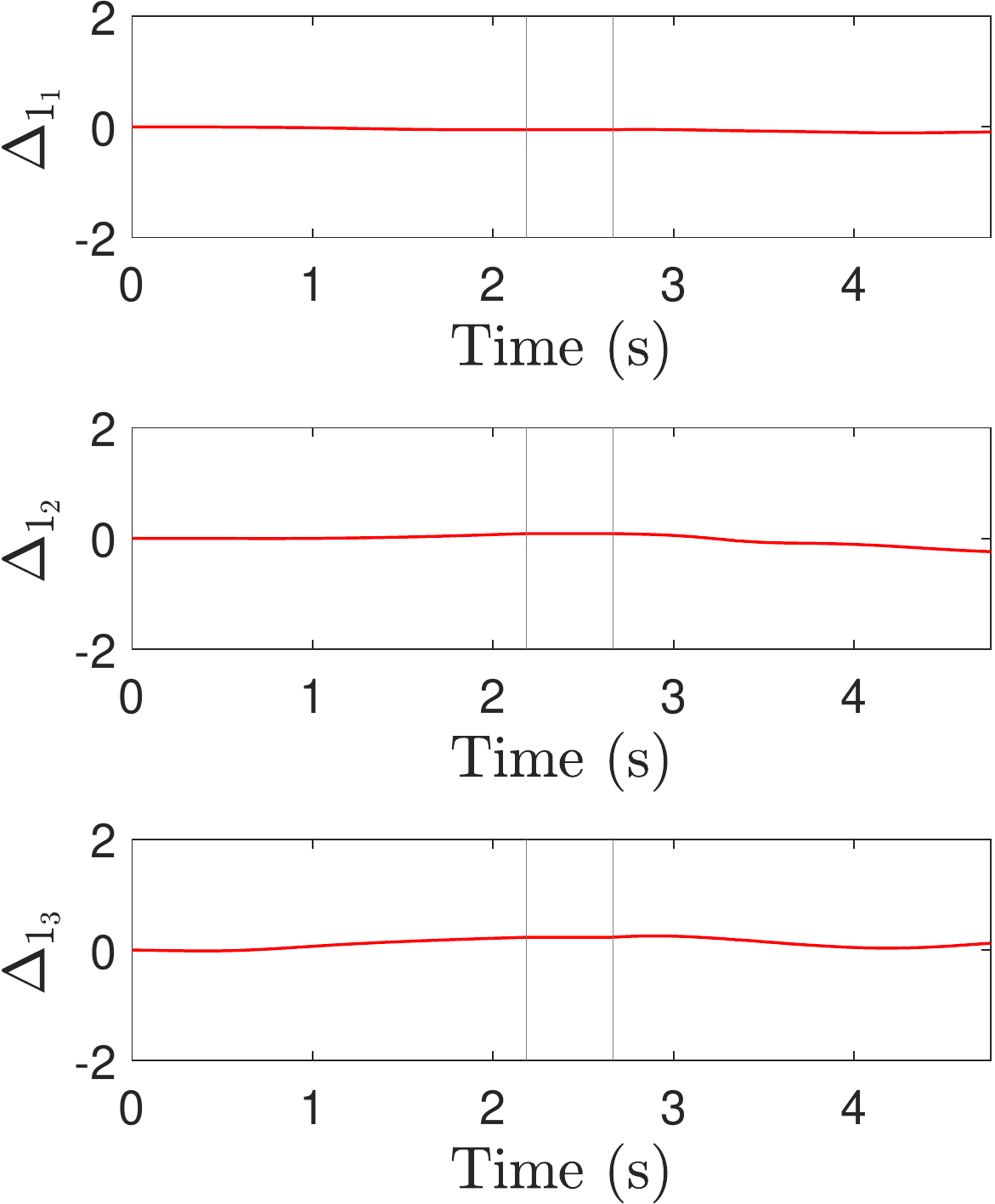}}
	\hfill
	\subfigure[$\Delta_2$ for attitude]{\includegraphics[width=0.5\columnwidth]{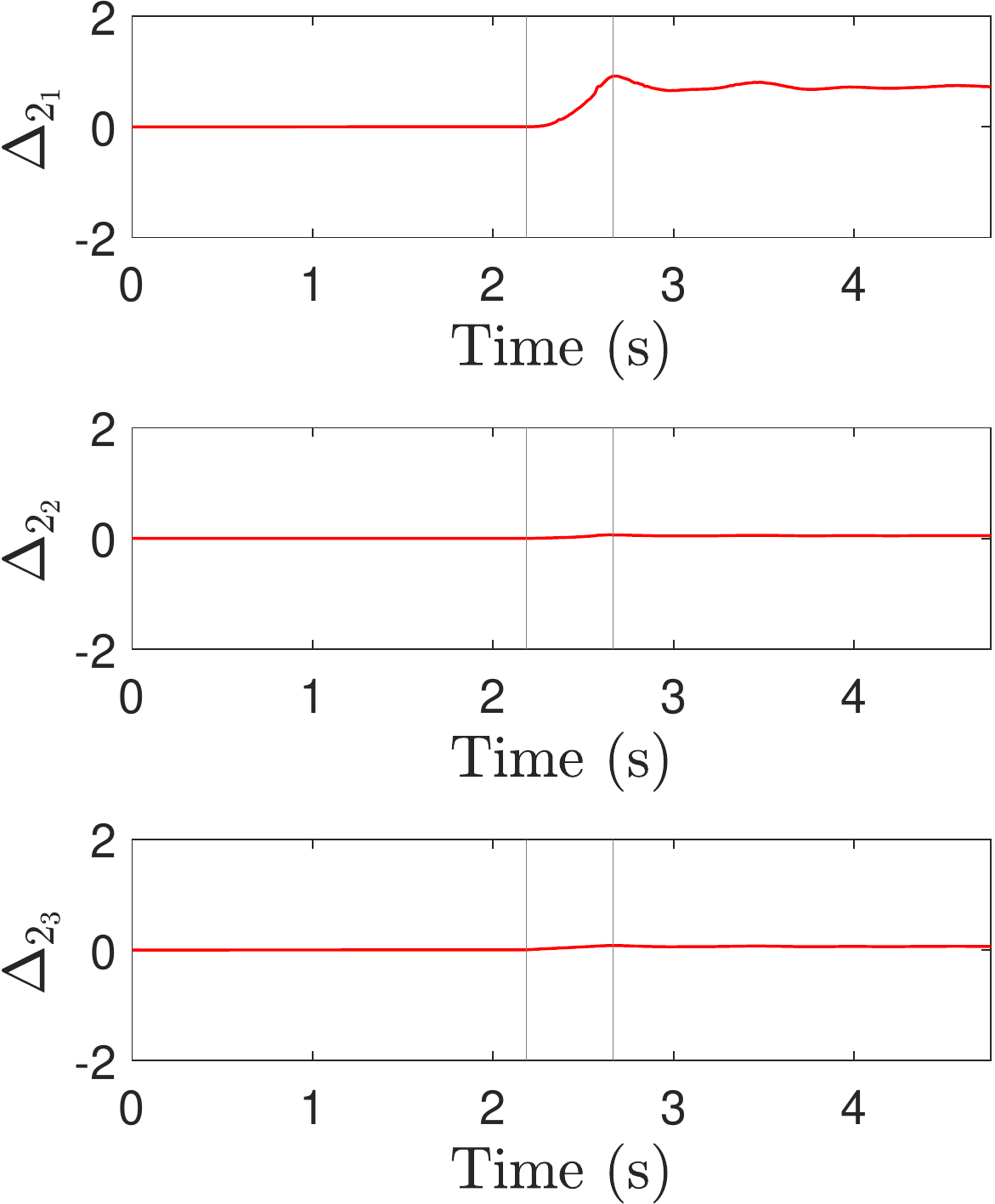}}
	}
	\centerline{
		\subfigure[Angular velocity ($\si{\radian \per \second}$)]{\includegraphics[width=0.5\columnwidth]{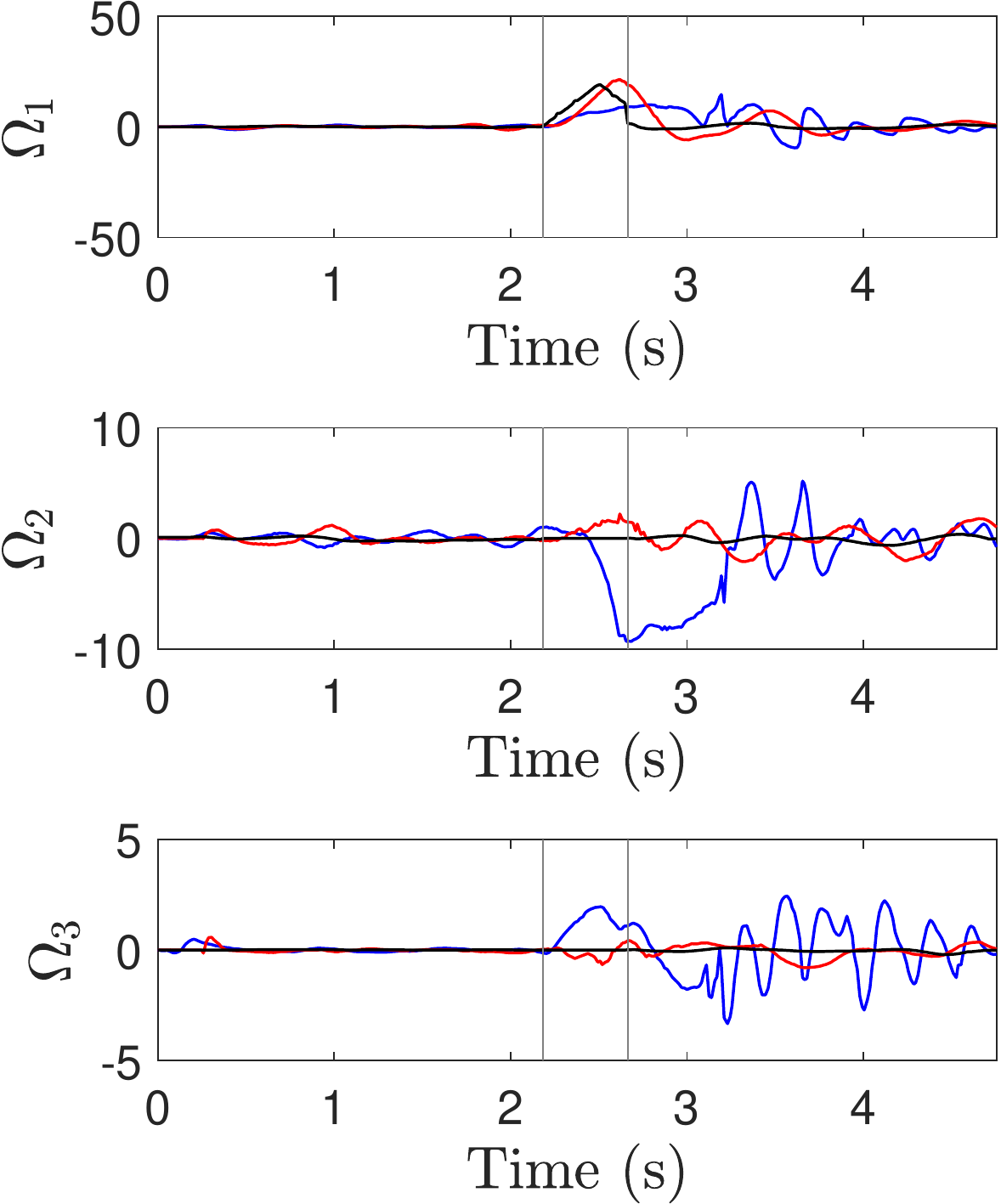}}
	\hfill
	\subfigure[Thrust ($\si{\newton}$)]{\includegraphics[width=0.5\columnwidth]{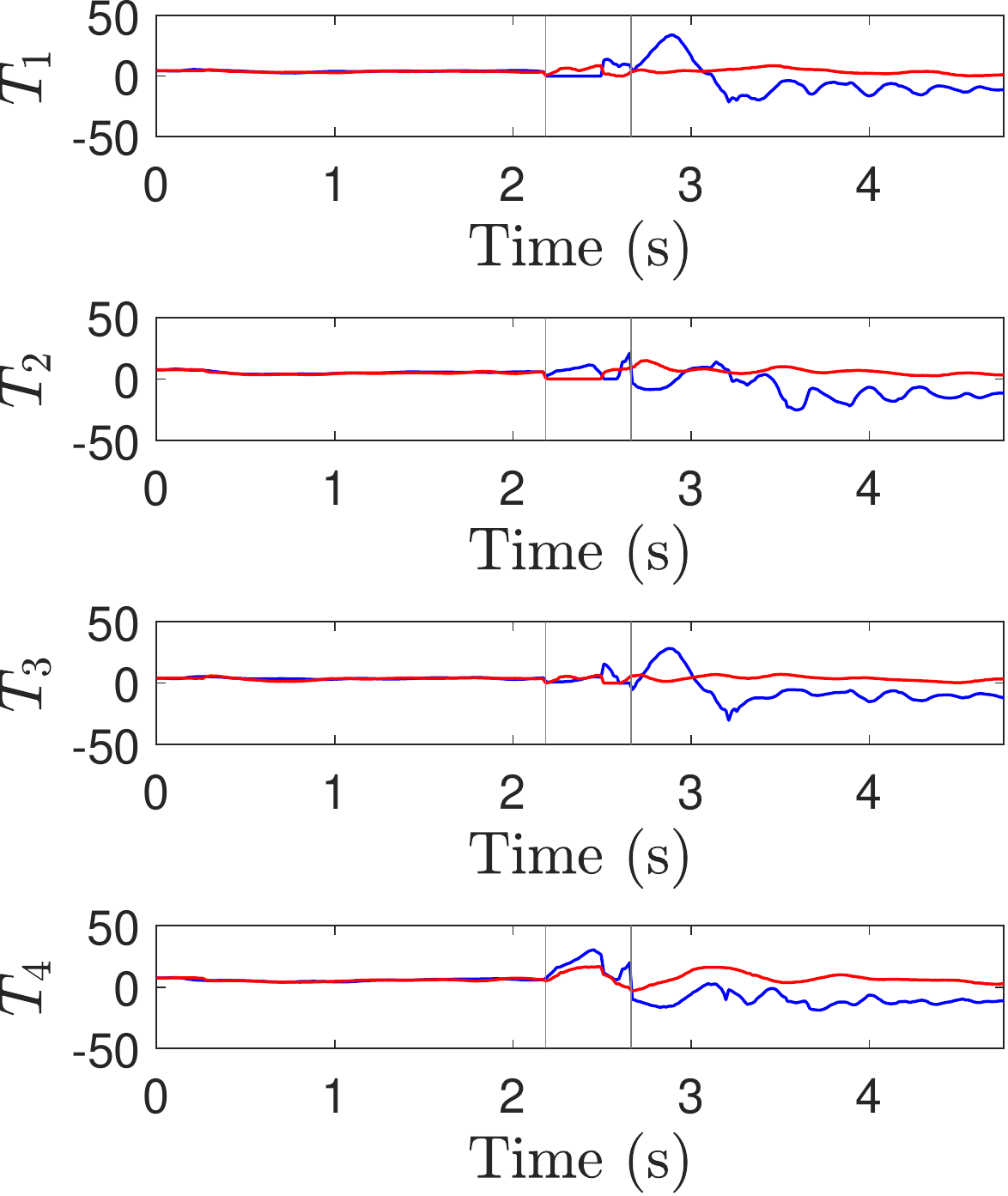}}	
	}
    \caption{Back flip (adaptive terms, angular velocity, and thrust), black:desired, blue:without disturbance rejection~\cite{LeeLeoPICDC10}, red: adaptive controller}
	\label{fig:flip_Delta} 
\end{figure}

\begin{figure}
	\centerline{
		\subfigure[At $t=0\si{\second}$]{\includegraphics[width=0.7\columnwidth]{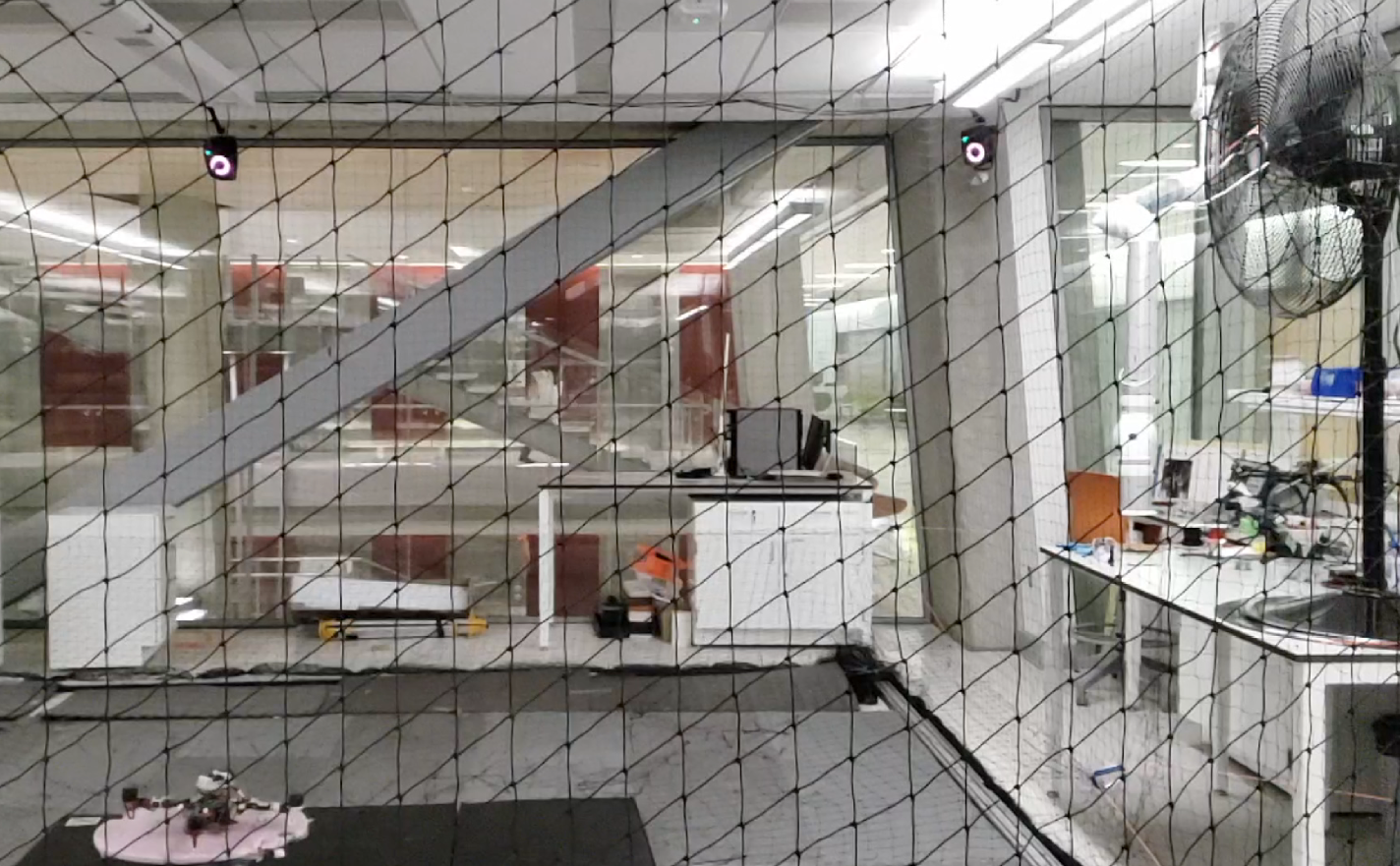}}
	}
	\centerline{
		\subfigure[At $t=2.20\si{\second}$]{\includegraphics[width=0.7\columnwidth]{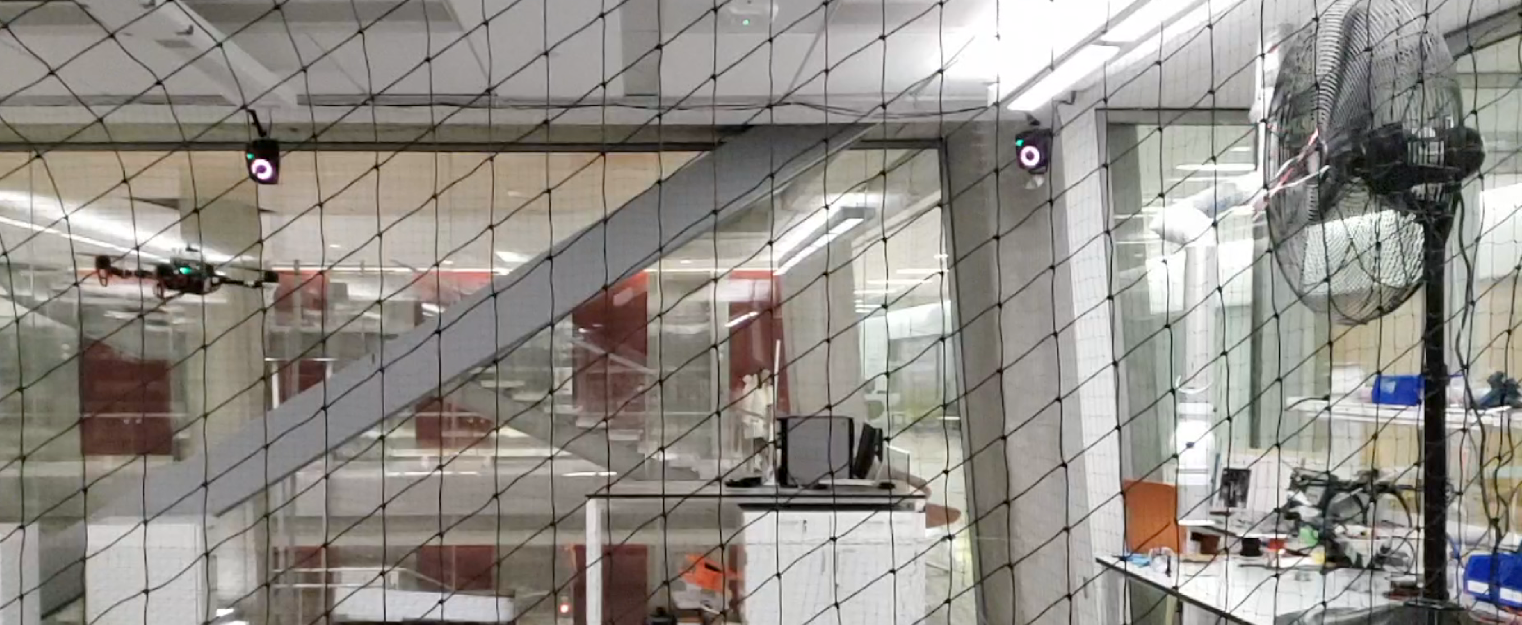}}
	}
	\centerline{
		\subfigure[At $t=2.52\si{\second}$]{\includegraphics[width=0.7\columnwidth]{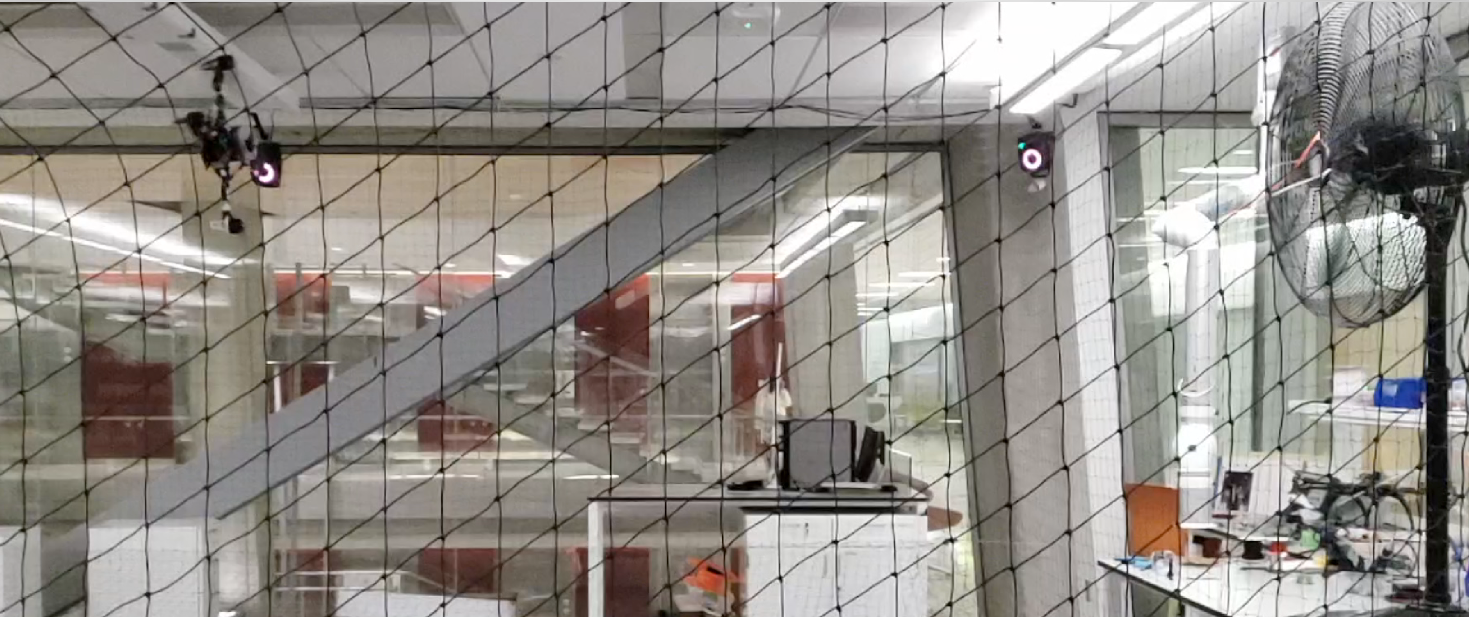}}
	}
	\centerline{
		\subfigure[At $t=2.71\si{\second}$]{\includegraphics[width=0.7\columnwidth]{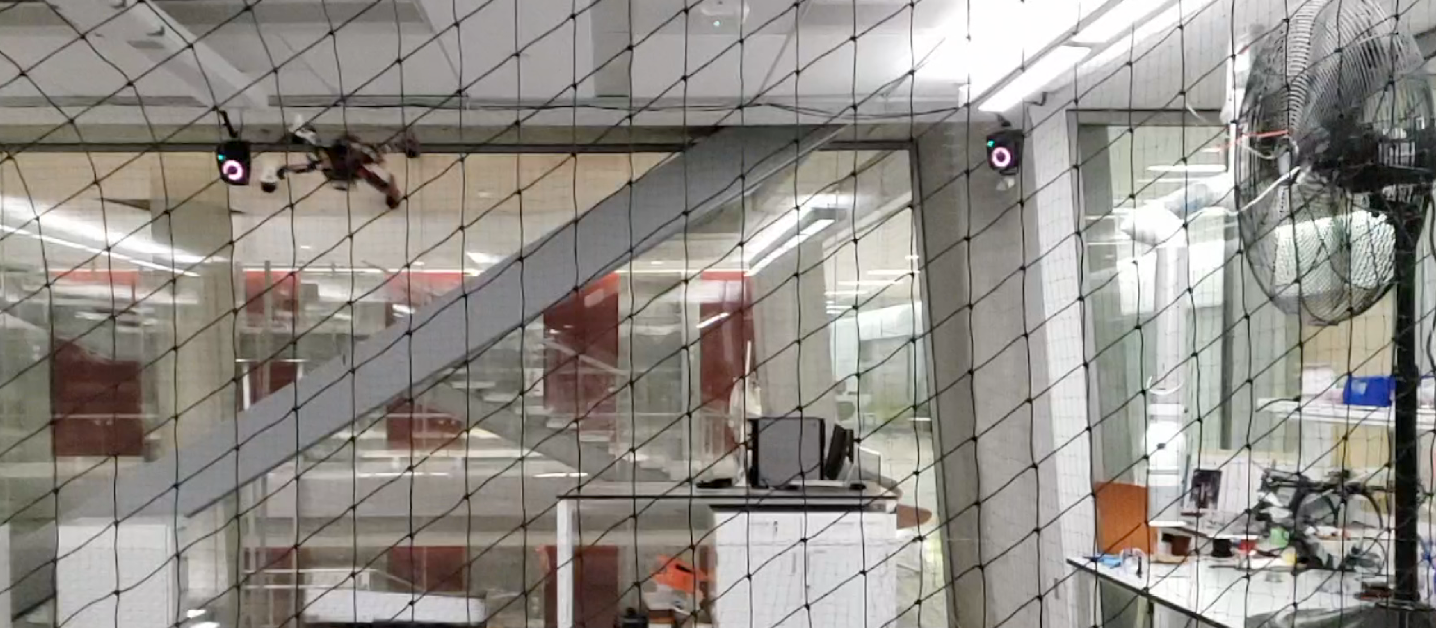}}
	}
	\centerline{
		\subfigure[At $t=2.80\si{\second}$]{\includegraphics[width=0.7\columnwidth]{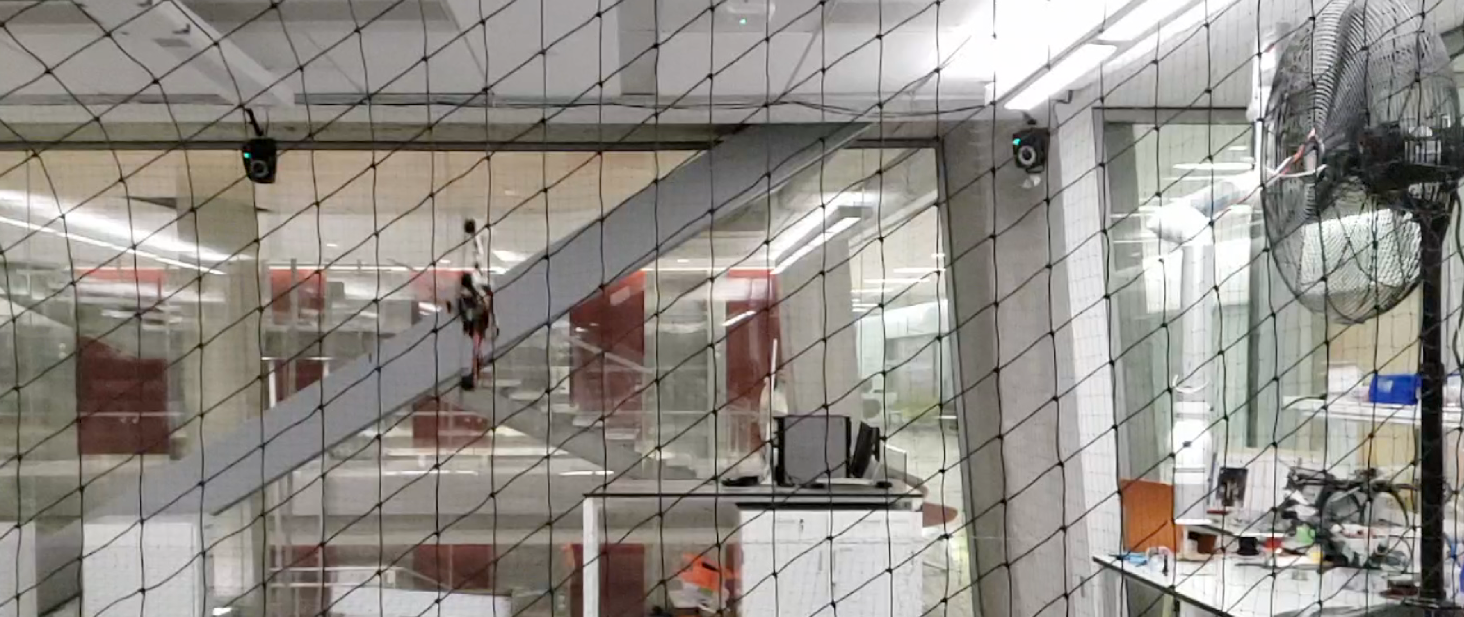}}
	}
	\centerline{
		\subfigure[At $t=3.24\si{\second}$]{\includegraphics[width=0.7\columnwidth]{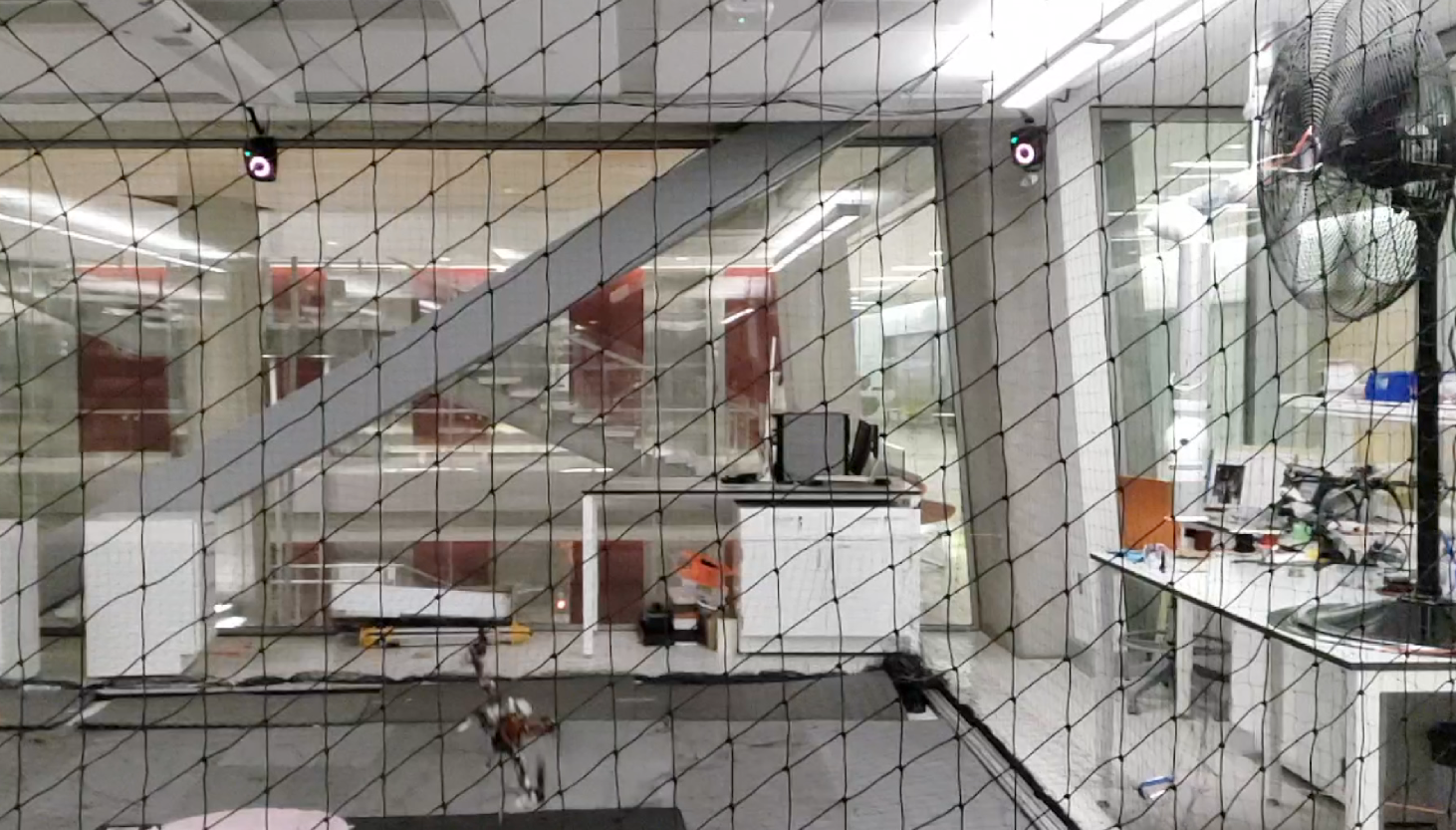}}
	}
    \caption{Backflip without disturbance rejection~\cite{LeeLeoPICDC10} (snapshots}
	\label{fig:photo_flip_fail} 
\end{figure}

\begin{figure}
	\centerline{
		\subfigure[At $t=0\si{\second}$]{\includegraphics[width=0.7\columnwidth]{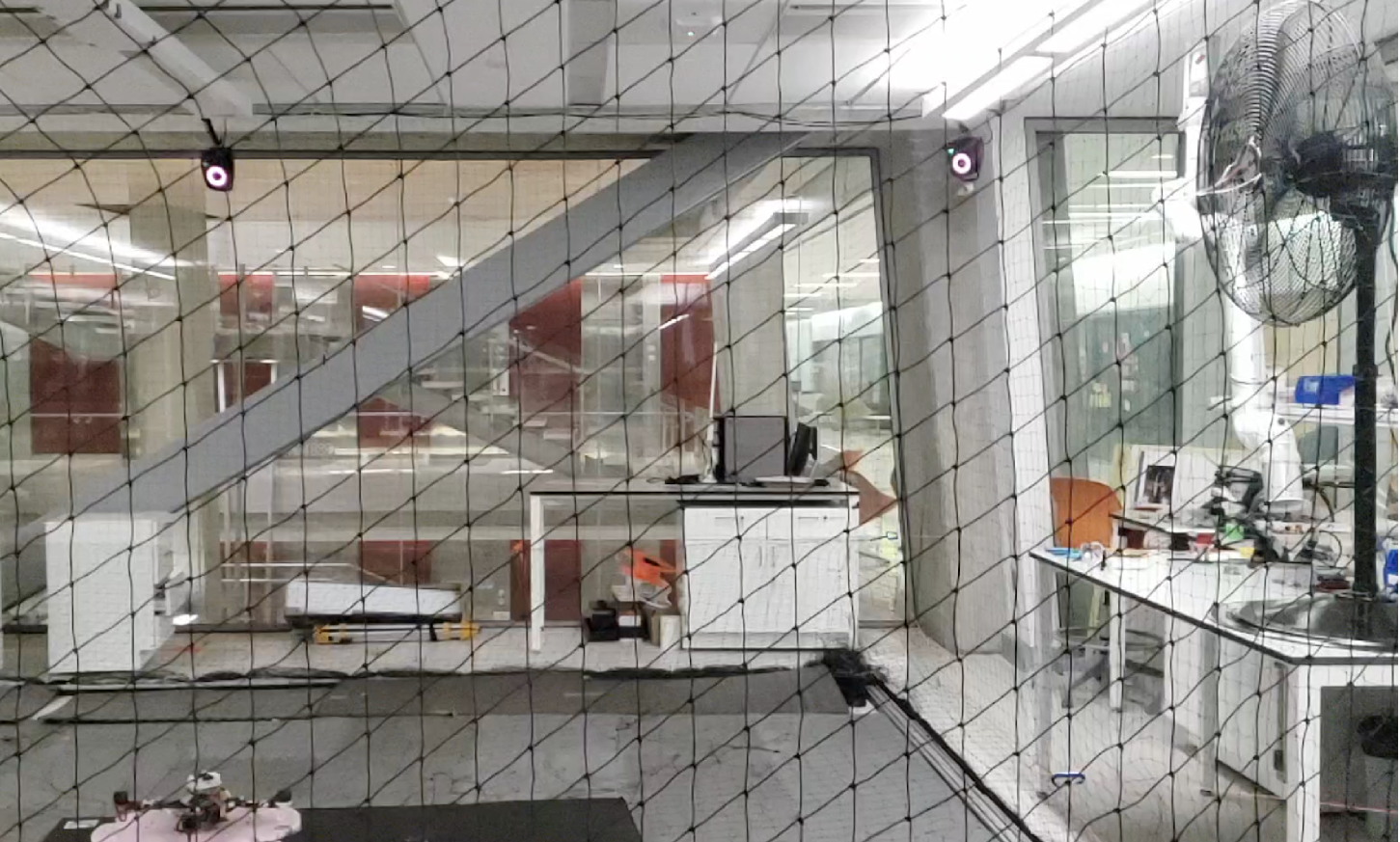}}
	}
	\centerline{
		\subfigure[At $t=2.02\si{\second}$]{\includegraphics[width=0.7\columnwidth]{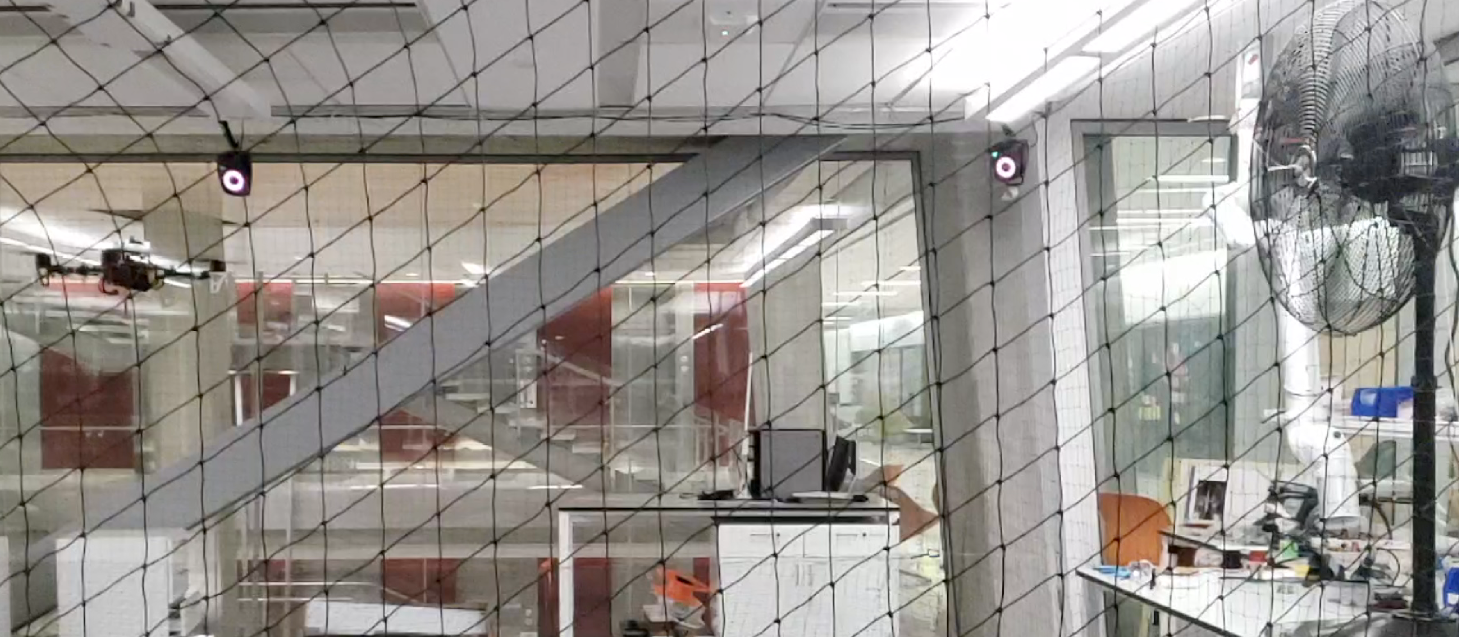}}
	}
	\centerline{
		\subfigure[At $t=2.48\si{\second}$]{\includegraphics[width=0.7\columnwidth]{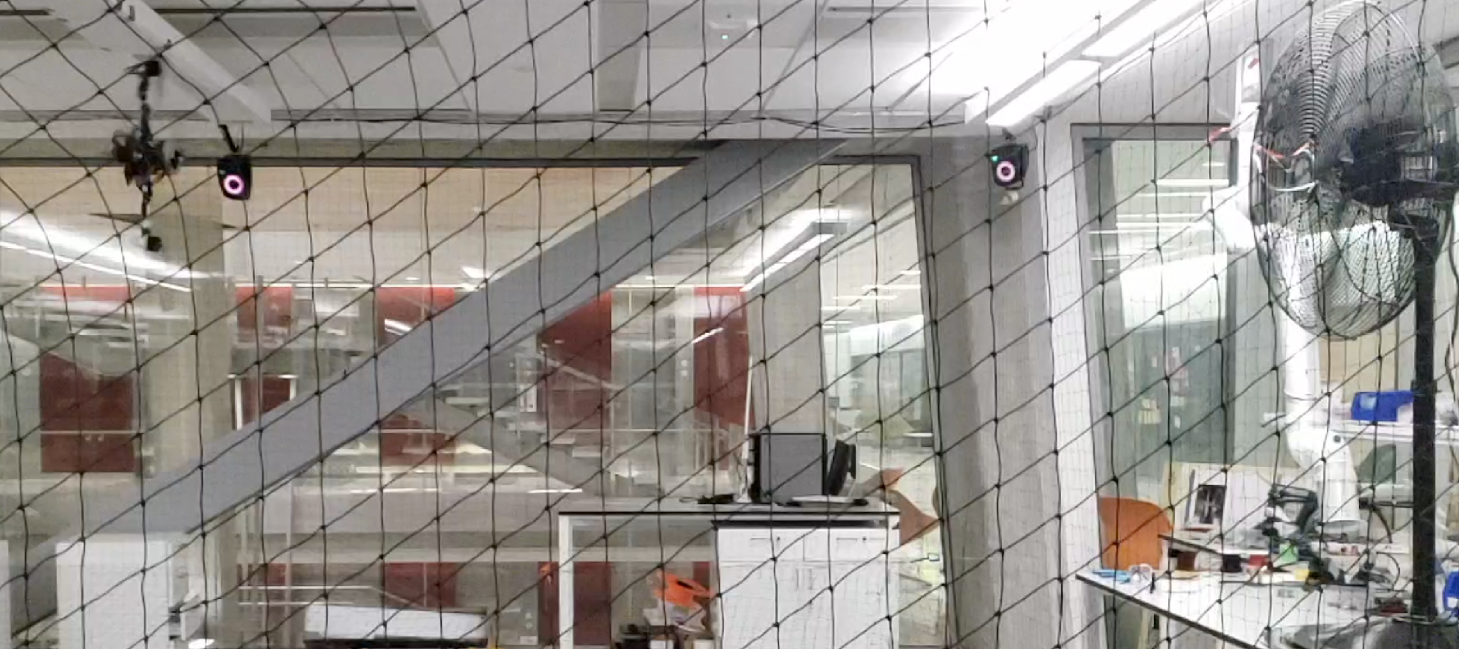}}
	}
	\centerline{
		\subfigure[At $t=2.55\si{\second}$]{\includegraphics[width=0.7\columnwidth]{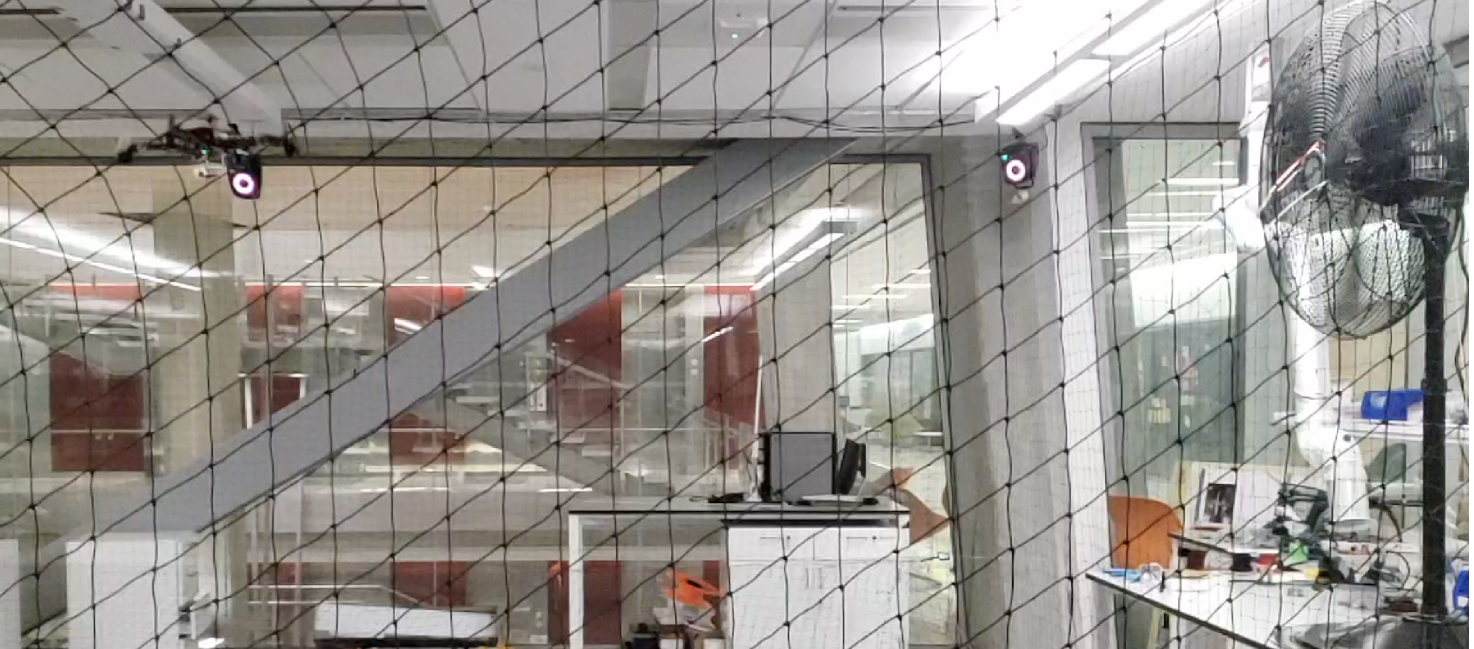}}
	}
	\centerline{
		\subfigure[At $t=2.60\si{\second}$]{\includegraphics[width=0.7\columnwidth]{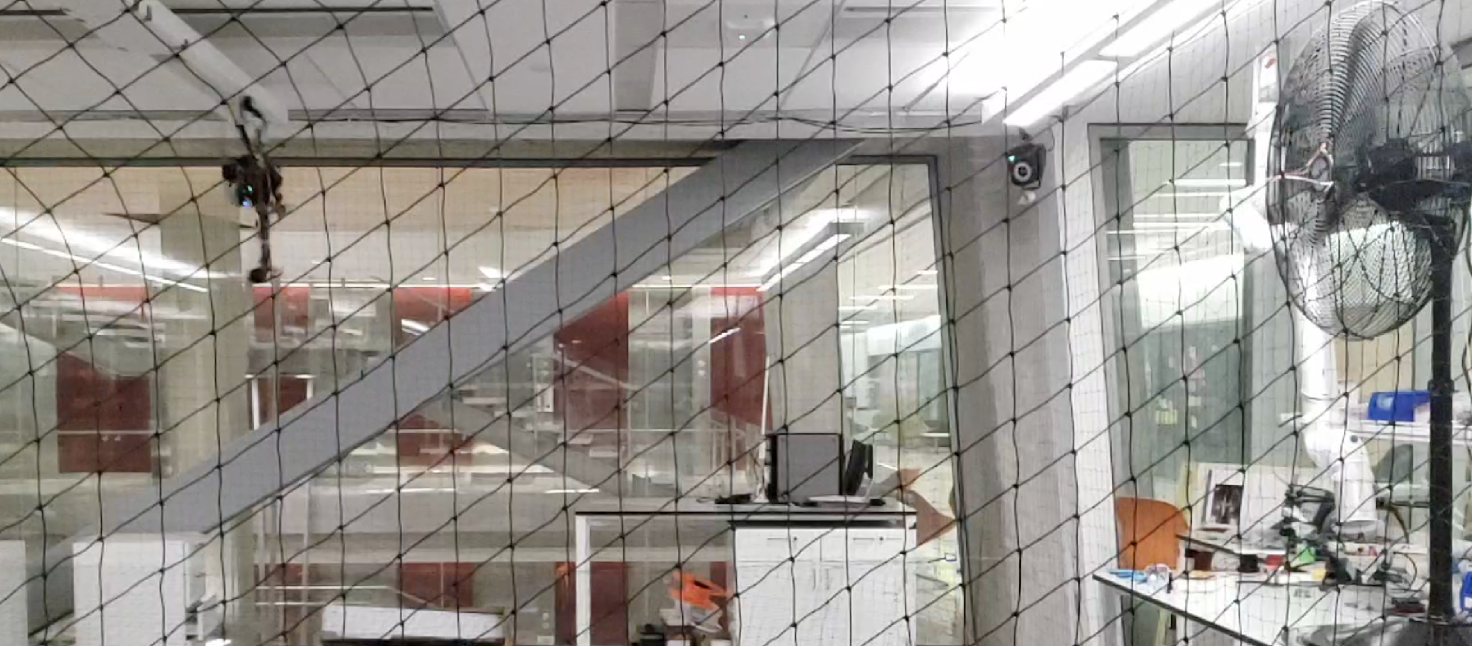}}
	}
	\centerline{
		\subfigure[At $t=2.74\si{\second}$]{\includegraphics[width=0.7\columnwidth]{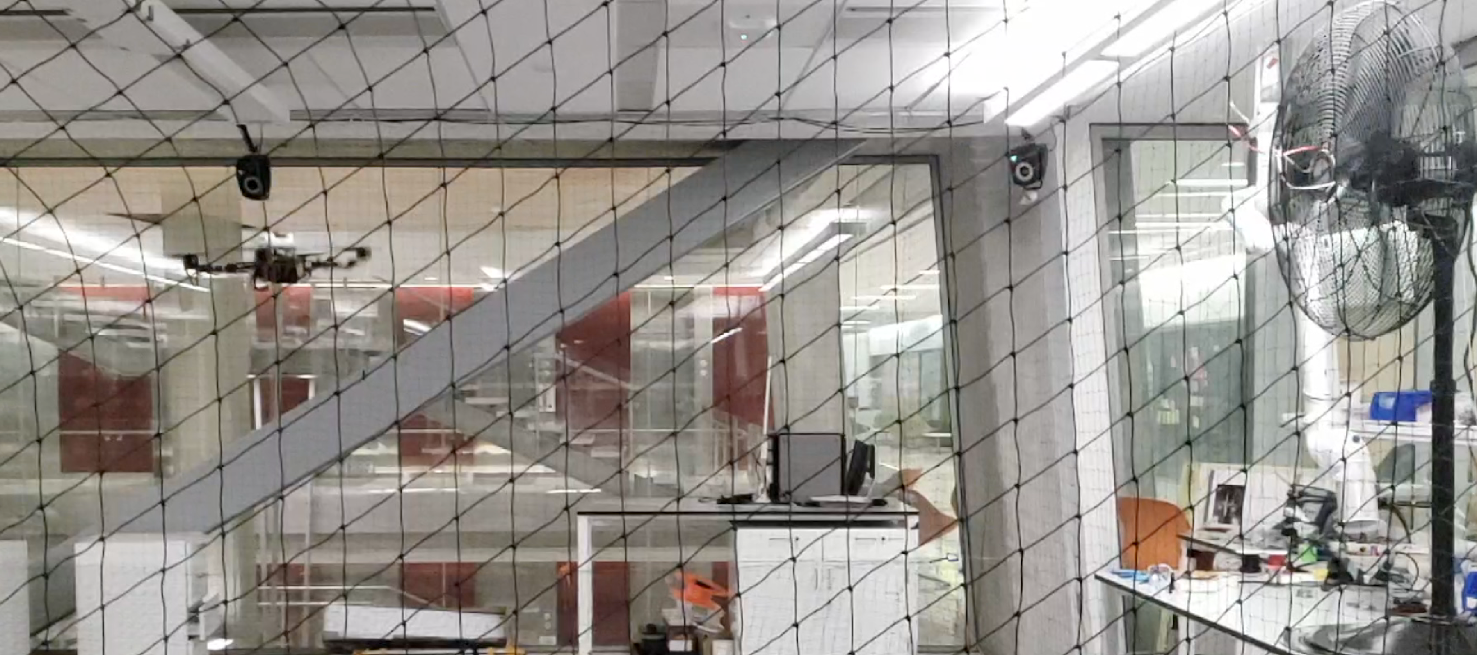}}
	}
    \caption{Backflip with the adaptive controller (snapshots)}
	\label{fig:photo_flip}
\end{figure}

\appendix

Here we present the proof of Proposition~\ref{prop:NN_prop}.
First, in Section \ref{Sec:identities}, selected identities that are used throughout the proof are presented.
Then in Section \ref{Sec:position}, we analyze the error dynamics for the position tracking command, which will be integrated with the attitude error dynamics presented in \ref{Sec:attitude}.
Finally, in Section~\ref{Sec:quadrotor}, we consider the stability of the complete dynamics.

\subsection{Identities} \label{Sec:identities}

For any $\mathcal{A}\in \Ree^{3\times 3},x, y\in\Ree^{3}$, $c_1, c_2,c_3\in \Ree$, 
\begin{gather}
\tr{ y x^T}=x^Ty,\label{eqn:tryx}\\
||x+y||\leq||x||+||y||,\label{eqn:x_y}\\
-c_1x^2+c_2x=-\frac{c_1}{2}x^2-\frac{c_1}{2}[x-\frac{c_2}{c_1}]^2+\frac{c_2^2}{2c_1}\nonumber\\
\leq-\frac{c_1}{2}x^2+\frac{c_2^2}{2c_1},\label{eqn:ax2+bx}\\
-c_1x^2-c_2xy-c_3y^2\leq-c_1x^2+c_2xy-c_3y^2.\label{eqn:quad}\\
\mathcal{A}^T\hat{x}+\hat{x} \mathcal{A}=([\tr{\mathcal{A}} I_{3 \times 3}-\mathcal{A}] x)^{\wedge}.\label{eqn:Axtrace}
\end{gather}

Let $\mathcal{V}_{0_i}$ be the part of the Lyapunov function dependent of $\tilde W_i, \tilde V_i$ defined as 
\begin{gather}
\mathcal{V}_{0_i}=\frac{1}{2\gamma_{w_i}}\tr{\tilde W_i^T \tilde W_i}+\frac{1}{2\gamma_{v_i}}\tr{\tilde V_i^T \tilde V_i},\label{eqn:Lyap0}.
\end{gather}
We find the upper bound of the following expression, defined as $\mathcal{B}_i\in\Re$,
\begin{align}
\mathcal{B}_i = -a_i^T(\tilde\Delta_i)+\dot{\mathcal{V}}_{0_i}.\label{eqn:B}
\end{align}
The error dynamics of the neural network weights from~\refeqn{errorNN} are give by
\begin{gather}
\dot{\tilde{W}}_i=-\dot{\bar{W}}_i,\quad
\dot{\tilde{V}}_i=-\dot{\bar{V}}_i.\label{eqn:tildeWV}
\end{gather}
We substitute~\refeqn{Wdot}--\refeqn{Vdot} into~\refeqn{tildeWV}. Using~\refeqn{OENN}, $\mathcal{B}_i$ is rewritten as
\begin{align}
\mathcal{B}_i&=a_i^T\{-\tilde W_i^T [\sigma(z_i)-\sigma^\prime(z_i)z_i]-\bar W_i^T \sigma^\prime(z_i)\tilde{z}_i+w_i\}\nonumber \\
&\quad+\tr{\tilde W_i^T[\sigma(z_i)a_i^T-\sigma^\prime(z_i)z_ia_i^T+\kappa_i \bar W_i]} \nonumber \\
&\quad+\tr{\tilde V_i^T\{x_{nn_i}[\sigma^\prime(z_i)^T \bar W_i a_i]^T+\kappa_i \bar V_i \}}.
\end{align}
Applying~\refeqn{tryx}, it reduces to
\begin{gather}
\mathcal{B}_i=\kappa_i \tr{\tilde Z_i^T \bar Z_i}+a_i^T(w_i).\label{eqn:Bi_init}
\end{gather}
We have
\begin{gather}
\tr{\tilde Z_i^T \bar Z_i}=\tr{\tilde Z_i^T Z_i}-\tr{\tilde Z_i^T \tilde Z_i}\leq||\tilde Z_i|| Z_{M_i}-||\tilde Z_i||^2.
\end{gather}
The inequality \refeqn{ax2+bx} implies
\begin{gather}
- ||\tilde Z_i||^2+Z_{M_i} ||\tilde Z_i||\leq-\frac{1}{2}||\tilde Z_i||^2+\frac{ Z_{M_i}^2}{2}.  \label{eqn:traceZtilde}
\end{gather}
Since $\norm{\sigma}\leq 1,\, \norm{\sigma^\prime}\leq 0.25$, it can be shown that the upper bound for~\refeqn{O} is
\begin{gather}
\norm{\mathcal{O}_i} \leq 2+0.25\norm{\tilde V_i} \norm{x_{nn_i}}.
\end{gather}
From~\refeqn{NNbound}, the upper bound of~\refeqn{wi} is
\begin{align}
\norm{w_i}\leq&
0.25 V_{M_i} \norm{\tilde W_i} \norm{x_{nn_i}}+W_{M_i} \norm{\mathcal{O}_i}+\epsilon_i.\label{eqn:normw}
\end{align}
Since $\norm{x_{nn_i}}\leq 1+\norm{x_{1_i}}+\norm{x_{2_i}}$, $\norm{\tilde Z_i}\geq\norm{\tilde W_i},\,\norm{\tilde Z_i}\geq\norm{\tilde V_i}$,~\refeqn{NNbound}, we obtain
\begin{gather}
\norm{w_i}\leq C_{1_i}+C_{2_i}||\tilde Z_i|| (1+\norm{x_{1_i}}+\norm{x_{2_i}}),
\end{gather}
where $C_{2_i}\geq 0.25 (V_{M_i}+W_{M_i}),\,C_{1_i}\geq 2W_{M_i}+\epsilon_i$.

Substituting~\refeqn{traceZtilde} and~\refeqn{normw} into~\refeqn{Bi_init}, 
\begin{align}
\mathcal{B}_i\leq&-\frac{\kappa_i}{2}||\tilde Z_i||^2+\frac{\kappa_i Z_{M_i}^2}{2}\nonumber\\&+\norm{a_i}\{C_{1_i}+C_{2_i}||\tilde Z_i|| (1+\norm{x_{1_i}}+\norm{x_{2_i}})\}.\label{eqn:last_term}
\end{align}
\subsection{Position Error Dynamics} \label{Sec:position}

Taking the derivative of~\refeqn{ev} and substituting~\refeqn{E3} and~\refeqn{EC2}, the error dynamics are defined as
\begin{gather}
\dot{e}_x=e_v,\\
m\dot{e_v}=mge_3-\Delta_1-fRe_3-m\ddot{x}_d.\label{eqn:mevdot1}
\end{gather}
Define $\mathcal{X}\in\Ree^3$ as
\begin{gather}
\mathcal{X}\equiv\frac{f}{e_3^TR_c^TRe_3}[(e_3^TR_c^TRe_3)Re_3-R_ce_3],
\end{gather}
where $e_3^TR_c^TRe_3>0$~\cite{LeeLeoPICDC10}. Equation \refeqn{mevdot1} is rewritten as
\begin{gather}
m\dot e_v=mge_3-\Delta_1-m\ddot x_d-\frac{f}{e_3^TR_c^TRe_3}R_ce_3-\mathcal{X}.\label{eqn:mevdot2}
\end{gather}
Since $b_{3c}=R_ce_3=\frac{-A}{\norm{A}}$, $f=-A^TRe_3$, we can conclude that $f=(\norm{A}R_ce_3)^TRe_3$, therefore
\begin{gather}
-\frac{f}{e_3^TR_c^TRe_3}R_ce_3=A.\label{eqn:newA}
\end{gather}
Substituting~\refeqn{newA},~\refeqn{A} into~\refeqn{mevdot2}, the velocity error dynamics is written as
\begin{gather}
m\dot e_v=-k_x e_x-k_ve_v-\tilde \Delta_1-\mathcal{X}.\label{eqn:mevdot}
\end{gather}

Next, we find the upper bound of $\mathcal{X}$. 
From~\refeqn{newA}, $\norm{A}=\norm{-\frac{f}{e_3^TR_c^TRe_3}R_ce_3}$. 
Since $R_ce_3$ is a unit vector, $\norm{A}=\norm{-\frac{f}{e_3^TR_c^TRe_3}}$. 
Consequently, the norm of $\mathcal{X}$ can be written as
\begin{gather}
\norm{\mathcal{X}}=\norm{A}\norm{[(e_3^TR_c^TRe_3)Re_3-R_ce_3}.
\end{gather}
Also, it is shown that $\norm{[(e_3^TR_c^TRe_3)Re_3-R_ce_3}\leq\norm{e_R}\leq\beta<1$, where $\beta=\sqrt{\psi_1(2-\psi_1)}$~\cite{LeeLeoPICDC10}. Substituting~\refeqn{A} and~\refeqn{acceleration_bound}, the upper bound of $\norm{\mathcal{X}}$ is given by
\begin{gather}
\norm{\mathcal{X}}\leq(k_x\norm{e_x}+k_v\norm{e_v}+B_1)\norm{e_R}.\label{eqn:normX}
\end{gather}

For a non-negative constant $c_1$, the Lyapunov function for the position dynamics is chosen as
\begin{gather}
\mathcal{V}_1=\frac{1}{2}k_x e_x^T e_x+\frac{1}{2}m e_v^T e_v+m c_1 e_x^T e_v+\mathcal{V}_{0_1},\label{eqn:V1}
\end{gather}
where $\mathcal{V}_{0_1}$ is given by~\refeqn{Lyap0}. It is straightforward to show
\begin{gather}
\lambda_m(\mathcal{M}_{11}) ||\mathcal{Z}_{11}||^2+\mathcal{V}_{0_1}\leq\mathcal{V}_1 \leq \lambda_M(\mathcal{M}_{12}) ||\mathcal{Z}_{11}||^2+\mathcal{V}_{0_1}, \label{eqn:V1bound}
\end{gather}
where
\begin{gather}
\mathcal{M}_{11}=\frac{1}{2}\begin{bmatrix} k_x & -mc_1
\\ -mc_1 & m \end{bmatrix},\quad\mathcal{M}_{12}=\frac{1}{2}\begin{bmatrix} k_x & mc_1
\\ mc_1 & m \end{bmatrix},\label{eqn:M1}\\
\mathcal{Z}_{11}=[||e_x||,||e_v||]^T.\label{eqn:Z11}
\end{gather}
If $c_1$ is sufficiently small such that
\begin{gather}
c_1<\sqrt{\frac{k_x}{m}},\label{eqn:c10}
\end{gather}
then $\mathcal{M}_{11},\mathcal{M}_{12}$ are positive-definite.

Taking the derivative of the $\Lya$ function, 
\begin{align}
\dot{\mathcal{V}}_1=&k_x e_v^Te_x+(e_v+c_1e_x)^Tm\dot e_v+mc_1e_v^Te_v+\dot{\mathcal{V}}_{0_1}.\label{eqn:vdot_mid1}
\end{align}
Substituting~\refeqn{mevdot} into~\refeqn{vdot_mid1} and rearranging,
\begin{align}
\dot{\mathcal{V}}_1=&(mc_1-k_v)e_v^Te_v-c_1 k_x e_x^T e_x-c_1k_v e_x^T e_v\nonumber \\
&-(e_v+c_1e_x)^T\mathcal{X}-(e_v+c_1e_x)^T\tilde \Delta_1+\dot{\mathcal{V}}_{0_1}.
\end{align}

From~\refeqn{a1a2}, the last two terms of the above expression are the same as~\refeqn{B}. 
Substituting its equivalent expression given by~\refeqn{last_term}, and substituting~\refeqn{normX},
\begin{align}
\dot{\mathcal{V}}_1\leq&-(k_v(1-\beta)-mc_1)e_v^Te_v-c_1k_x(1-\beta)e_x^Te_x\nonumber \\&+c_1k_v(1+\beta)\norm{e_x}\norm{e_v}-\frac{\kappa_1}{2}||\tilde Z_1||^2+\frac{\kappa_1 Z_{M_1}^2}{2}\nonumber\\
&+\norm{a_1}\{C_{1_1}+C_{2_1}||\tilde Z_1|| (1+\norm{x_{1_1}}+\norm{x_{2_1}})\}\nonumber \\&+\norm{e_R}\{B_1(c_1\norm{e_x}+\norm{e_v})+k_xe_{x_{max}}\norm{e_v}\},\label{eqn:Vdot1_mid0}
\end{align}
where $\norm{e_x}\leq e_{x_{\max}}$ is used for simplifying multiplication of the three vectors, $\norm{e_R}\norm{e_x}\norm{e_v}$, for a fixed positive constant $e_{x_{\max}}$.

It is assumed that the desired trajectory is bounded such that $\norm{x_d}\leq x_{d_{max}},\,\norm{\dot x_d}\leq v_{d_{max}}$, where $x_{d_{max}},\, v_{d_{max}}>0$. 
From~\refeqn{ev} and $x_{1_1}=x$, $x_{2_1}=v$, $\norm{x_1}\leq\norm{e_x}+x_{d_{max}}$, $\norm{x_2}\leq\norm{e_v}+v_{d_{max}}$. 
Substituting these into~\refeqn{Vdot1_mid0}, expanding $a_1$, and using~\refeqn{x_y}, we obtain
\begin{align}
\dot{\mathcal{V}}_1\leq&-(k_v(1-\beta)-mc_1)e_v^Te_v-c_1k_x(1-\beta)e_x^Te_x\nonumber \\&+c_1k_v(1+\beta)\norm{e_x}\norm{e_v}-\frac{\kappa_1}{2}||\tilde Z_1||^2+\frac{\kappa_1 Z_{M_1}^2}{2}\nonumber\\
&+C_{1_1}\norm{e_v} +C_{3_1} \norm{e_v}^2+C_{4_1}\norm{e_v}\norm{\tilde Z_1}\nonumber\\
&+c_1(C_{1_1}\norm{e_x}+C_{3_1}\norm{e_x}^2+C_{4_1}\norm{e_x}\norm{\tilde Z_1})\nonumber\\&+(1+c_1)C_{3_1}\norm{e_x}\norm{e_v}\nonumber \\&+\norm{e_R}\{B_1(c_1\norm{e_x}+\norm{e_v})+k_xe_{x_{max}}\norm{e_v}\},\label{eqn:Vdot1_mid1}
\end{align}
where $C_{1_1}\geq 2W_{M_1}+\epsilon,\,C_{2_1}\geq 0.25 (V_{M_1}+W_{M_1}),\,C_{3_1}\geq C_{2_1} Z_{M_1},\,C_{4_1}\geq C_{2_1}(1+x_{d_{max}}+v_{d_{max}})$. 

Using~\refeqn{ax2+bx}, and defining $k_{v_{\beta}}\equiv k_v(1-\beta)-mc_1-C_{3_1}$ and $k_{x_\beta}\equiv k_x(1-\beta)-C_{3_1}$, the following expressions are rearranged as
\begin{align}
-&k_{x_\beta}e_x^Te_x+C_{1_1}e_x\leq-\frac{k_{x_\beta}}{2}e_x^Te_x+\frac{C_{1_1}^2}{2k_{x_\beta}},\label{eqn:one1}\\
-&k_{v_{\beta}}e_v^Te_v+C_{1_1}e_v\leq-\frac{k_{v_{\beta}}}{2}e_v^Te_v+\frac{C_{1_1}^2}{2k_{v_{\beta}}}.\label{eqn:one2}
\end{align}
Substituting~\refeqn{one1}--\refeqn{one2} into~\refeqn{Vdot1_mid1} results in 
\begin{align}
\dot{\mathcal{V}}_1\leq&-\frac{c_1k_{x_\beta}}{2}e_x^Te_x-\frac{k_{v_\beta}}{2}e_v^Te_v-\frac{\kappa_1}{2}||\tilde Z_1||^2+k_{xv}\norm{e_x}\norm{e_v}\nonumber \\
&+C_{4_1}||e_v||||\tilde Z_1||+c_1C_{4_1}||e_x||||\tilde Z_1||+C_{5_1}\nonumber\\
&+\norm{e_R}\{c_1B_1\norm{e_x}+(B_1+k_xe_{x_{max}})\norm{e_v}\},\label{eqn:V1dot}
\end{align}
where $k_{xv}=c_1[(1+\beta)k_v+C_{3_1}]+C_{3_1}$, 
$C_{5_1}=\frac{c_1C_{1_1}^2}{2k_{x_\beta}}+\frac{C_{1_1}^2}{2k_{v_\beta}}+\frac{\kappa_1 Z_{M_1}^2}{2}$.

\subsection{Attitude Error Dynamics} \label{Sec:attitude}
Here, we analyze the error dynamics for the attitude tracking command. 
Let the attitude error function be
\begin{align}
    \Psi(R,R_c)=\frac{1}{2}\tr{I_{3\times 3}-R_c^TR}.\label{eqn:Psi}
\end{align}
Taking the derivative of \refeqn{eomega} and \refeqn{Psi}, and using \refeqn{Axtrace}, \refeqn{E4}, and \refeqn{Mc}, the attitude error dynamics are given by
\begin{gather}
\dot{e}_R=\frac{1}{2}(\tr{R^T R_c}I_{3\times 3}-R^TR_c)e_\Omega \equiv C(R_c^TR)e_\Omega,\label{eqn:eRdot}\\
J\dot{e_\Omega}=-k_Re_R-k_\Omega e_\Omega-\tilde\Delta_2,\label{eqn:eomegadot}\\
\dot \Psi(R,R_c)= e_R^Te_\Omega,\label{eqn:Psidot}\\
||C(R_c^TR)||\leq1.\label{eqn:CRCTR}
\end{gather}
For more details about proof of \refeqn{eRdot} and~\refeqn{Psidot}--\refeqn{CRCTR}, see~\cite{LeeLeoPICDC10}.

For a non-negative constant $c_2$, the Lyapunov function for the attitude dynamics is defined as
\begin{gather}
\mathcal{V}_2=\frac{1}{2} e_\Omega^T J e_\Omega+k_R \Psi(R,R_c)+c_2 e_R^T Je_\Omega+\mathcal{V}_{0_2},\label{eqn:V2}
\end{gather}
where $\mathcal{V}_{0_2}$ is given by~\refeqn{Lyap0}, and
\begin{gather}
\frac{1}{2}\norm{e_R}^2\leq\Psi(R,R_c)\leq\frac{1}{2-\psi_1}\norm{e_R}^2,
\end{gather}
with $\psi_1=\frac{1}{k_R}[\frac{1}{2}e_\Omega(0)^TJe_\Omega(0)+k_R\Psi(R(0),R_c(0))]$. The bounds of $\mathcal{V}_2$ are
\begin{gather}
\lambda_m(\mathcal{M}_{21}) ||\mathcal{Z}_{21}||^2+\mathcal{V}_{0_2}\leq\mathcal{V}_2 \leq \lambda_M(\mathcal{M}_{22}) ||\mathcal{Z}_{21}||^2+\mathcal{V}_{0_2}, \label{eqn:V2bound}
\end{gather}
where \begin{gather}
\mathcal{M}_{21}=\frac{1}{2}\begin{bmatrix} k_R & -c_2 \lambda_{M_J}
\\ -c_2 \lambda_{M_J} & \lambda_{m_J} \end{bmatrix},
\mathcal{M}_{22}=\frac{1}{2}\begin{bmatrix} \frac{2k_R}{2-\psi_1} & c_2\lambda_{M_J}
\\ c_2\lambda_{M_J} & \lambda_{M_J}\end{bmatrix},\label{eqn:M2}\\
\mathcal{Z}_{21}=[||e_R||,||e_\Omega||]^T,\label{eqn:Z21}
\end{gather}
with $\lambda_{m_J}=\lambda_m(J),\lambda_{M_J}=\lambda_M(J)$. Provided that $c_2$ is sufficiently small to satisfy the following inequality, the matrices $\mathcal{M}_{21},\mathcal{M}_{22}$ are positive-definite,
\begin{gather}
c_2<\min\{\frac{\sqrt{k_R\lambda_{m_J}}}{\lambda_{M_J}},  \sqrt{\frac{2k_R}{\lambda_M(2 - \psi_1)}}\},\label{eqn:c20}
\end{gather}
where $\psi_1 < 2$.

The time-derivative of the $\Lya$ function is given by
\begin{align}
\dot{\mathcal{V}}_2=& (e_\Omega+c_2e_R)^T J \dot e_\Omega+k_R \dot \Psi(R,R_c)+c_2 \dot e_R^T Je_\Omega\nonumber\\&+\dot{\mathcal{V}}_{0_2}.
\end{align}
Substituting error dynamics~\refeqn{eRdot}--\refeqn{CRCTR},~\refeqn{EC4}, and~\refeqn{Mc},
\begin{align}
\dot{\mathcal{V}}_2&= (e_\Omega+c_2e_R)^T (-k_Re_R-k_\Omega e_\Omega-\tilde \Delta_2)\nonumber \\
&\quad +k_R e_R^T e_\Omega+c_2 C(R_c^TR)e_\Omega^T Je_\Omega+\dot{\mathcal{V}}_{0_2}.
\end{align}
From~\refeqn{quad},~\refeqn{CRCTR}, and~$\norm{J}\leq\lambda_{M_J}$,
\begin{align}
\dot{\mathcal{V}}_2\leq&-c_2k_R e_R^Te_R+c_2 k_\Omega ||e_R||||e_\Omega||-(k_\Omega-c_2\lambda_{M_J})e_\Omega^Te_\Omega\nonumber \\
&-(e_\Omega+c_2e_R)^T(\tilde \Delta_2)+\dot{\mathcal{V}}_{0_2}.\label{eqn:vdot_mid1_1}
\end{align}

From~\refeqn{a1a2}, the last two terms of this expression are identical to~\refeqn{B}. 
Substituting its equivalent expression given by~\refeqn{last_term},
\begin{align}
\dot{\mathcal{V}}_2\leq&-c_2k_R e_R^Te_R+c_2 k_\Omega ||e_R||||e_\Omega||-(k_\Omega-c_2\lambda_{M_J})e_\Omega^Te_\Omega\nonumber \\
&-\frac{\kappa_2}{2}||\tilde Z_2||^2+\frac{\kappa_2 Z_{M_2}^2}{2}\nonumber\\&+\norm{a_2}\{C_{1_2}+C_{2_2}||\tilde Z_2|| (1+\norm{E(R)}+\norm{\Omega})\}.\label{eqn:vdot2_mid0}
\end{align}

It is assumed that $\norm{\dot{\bar{\Delta}}_1}\leq\delta_2$ and the desired trajectory is designed such that $\norm{\dddot{x}_d}\leq\delta_3$, where $\delta_2,\,\delta_3>0$. Thus $\norm{m\dddot{x}_d+\dot{\bar{\Delta}}_1}\leq B_2$. 
Taking the derivative of~\refeqn{A}, it can be shown that
\begin{gather}
\norm{\dot{A}}\leq k_x\norm{e_v}+k_v \norm{\dot{e}_v}+B_2.\label{eqn:Adot_bound}
\end{gather}
From~\refeqn{Rc}, $\dot R_c=[\dot b_{1c},\dot b_{2c},\dot b_{3c}]$.
Let $C=-b_{3_c}\times b_{1_d}$. 
We have 
\begin{gather}\dot b_{1c}=\dot b_{2c}\times b_{3c}+b_{2c}\times \dot b_{3c},\label{eqn:b1cdot}\\
\dot b_{2c}=-\frac{\dot C}{||C||}+\frac{C(C\cdot C)}{||C||^3},\label{eqn:b2cdot}\\
\dot b_{3c}=-\frac{\dot A}{||A||}+\frac{A(A\cdot A)}{||A||^3}.\label{eqn:b3cdot}
\end{gather}
From~\refeqn{acceleration_bound},~\refeqn{Adot_bound}, and~\refeqn{b3cdot},
\begin{gather}
\norm{\dot b_{3c}}\leq 2 \frac{k_x\norm{e_v}+k_v \norm{\dot{e}_v}+B_2}{k_x\norm{e_x}+k_v \norm{e_v}+B_1}\equiv B_3.\label{eqn:ndotb_3c}
\end{gather}
It is assumed that the desired trajectory is designed such that $\norm{\dot{b}_{1_d}}\leq \delta_4$, where $\delta_4>0$. 
Using~\refeqn{ndotb_3c}, it can be shown that $\norm{\dot C}\leq B_3+\delta_4$. 
From~\refeqn{b2cdot}, as $\norm{C}\leq 1$,
\begin{gather}
\norm{\dot b_{2c}}\leq 2(B_3+\delta_4),\label{eqn:ndotb_2c}
\end{gather}
From~\refeqn{b1cdot},~\refeqn{ndotb_3c}--\refeqn{ndotb_2c}
\begin{gather}
\norm{\dot b_{1c}}\leq 3B_3+2\delta_4.\label{eqn:ndotb_1c}
\end{gather}
Thus, from~\refeqn{ndotb_3c}--\refeqn{ndotb_1c}, 
it can be shown that $\norm{\dot{R}_c}\leq B_4$, for a positive $B_4$.
From~\refeqn{EC3}, $\norm{\Omega_c}\leq B_4$. Since~\refeqn{eomega}, $\norm{\Omega}\leq\norm{e_\Omega}+B_4$.

We have $\norm{E(R)}\leq E_{max}$ for a positive $ E_{max}$.
Substituting these into~\refeqn{vdot2_mid0}, and expanding $a_2$ with\refeqn{x_y},
\begin{align}
\dot{\mathcal{V}}_2\leq&-c_2k_R e_R^Te_R+c_2 k_\Omega ||e_R||||e_\Omega||-(k_\Omega-c_2\lambda_{M_J})e_\Omega^Te_\Omega\nonumber \\
&-\frac{\kappa_2}{2}||\tilde Z_2||^2+\frac{\kappa_2 Z_{M_2}^2}{2}\nonumber\\&+C_{1_2}\norm{e_\Omega} +C_{3_2} \norm{e_\Omega}^2+C_{4_2}\norm{e_\Omega}\norm{\tilde Z_2}\nonumber\\
&+c_2(C_{1_2}\norm{e_R}+C_{4_2}\norm{e_R}\norm{\tilde Z_2})\nonumber\\&+c_2C_{3_2}\norm{e_R}\norm{e_\Omega},\label{eqn:vdot2_mid1}
\end{align}
where $C_{1_2}\geq 2W_{M_2}+\epsilon_2,\,C_{2_2}\geq 0.25 (V_{M_2}+W_{M_2}),\,C_{3_2}\geq C_{2_2} Z_{M_2},\,C_{4_2}\geq C_{2_2}(1+E_{max}+B_4)$.
Using~\refeqn{ax2+bx}, the following expressions are rearranged into
\begin{align}
-&k_{R} e_R^Te_R+C_{1_2}||e_R||\leq-\frac{k_{R}}{2}e_R^Te_R+\frac{C_{1_2}^2}{2k_{R}},\label{eqn:two1}\\
-&k_{\Omega_\beta}e_\Omega^Te_\Omega+C_{1_2}||e_\Omega||\leq-\frac{k_{\Omega_\beta}}{2}e_\Omega^Te_\Omega+\frac{C_{1_2}^2}{2k_{\Omega_\beta}},\label{eqn:two2}
\end{align}
where $k_{\Omega_\beta}=k_\Omega-c_2\lambda_{M_J}-C_{3_2}$.
Then substituting~\refeqn{two1}--\refeqn{two2} in~\refeqn{vdot2_mid1}
\begin{align}
\dot{\mathcal{V}}_2\leq&-\frac{c_2k_{R}}{2}e_R^Te_R-\frac{k_{\Omega_\beta}}{2}e_\Omega^Te_\Omega-\frac{\kappa_2}{2}||\tilde Z_2||^2\nonumber \\
&+k_{R\Omega} ||e_R||||e_\Omega||\nonumber\\
&+C_{4_2}||e_\Omega||||\tilde Z_2||+c_2C_{4_2}||e_R||||\tilde Z_2||+C_{5_2}.\label{eqn:V2dot0}
\end{align}
where $k_{R\Omega}=c_2(\kappa_\Omega+C_{3_2}),\, C_{5_2}=\frac{c_2C_{2_1}^2}{2k_{R}}+\frac{C_{2_1}^2}{2k_{\Omega_\beta}}+\frac{\kappa_2 Z_{M_2}^2}{2}$. 


\subsection{Stability Proof for Quadrotor Dynamics} \label{Sec:quadrotor}

Here, we combine the position error dynamics and the attitude error dynamics to show the stability properties of the complete controlled quadrotor. 
The Lyapunov function is chosen as $\mathcal{V}=\mathcal{V}_1+\mathcal{V}_2$, where $\mathcal{V}_1,\mathcal{V}_2$ are given by~\refeqn{V1},~\refeqn{V2}. From~\refeqn{V1bound} and~\refeqn{V2bound}, the bound on $\mathcal{V}$ is given by
\begin{align}
\lambda_m&(\mathcal{M}_{11}) ||\mathcal{Z}_{11}||^2+\lambda_m(\mathcal{M}_{21}) ||\mathcal{Z}_{21}||^2+\mathcal{V}_{0_1}+\mathcal{V}_{0_2}\leq \mathcal{V}\nonumber\\&\leq \lambda_M(\mathcal{M}_{12}) ||\mathcal{Z}_{11}||^2+\lambda_M(\mathcal{M}_{22}) ||\mathcal{Z}_{21}||^2+\mathcal{V}_{0_1}+\mathcal{V}_{0_2}.
\end{align}
The upper bound can be rewritten as
\begin{align}
\mathcal{V}\leq
&\frac{1}{2} \mathcal{Z}_{1}^T \mathcal{N}_{1}^\prime \mathcal{Z}_{1}+\frac{1}{2} \mathcal{Z}_{2}^T \mathcal{N}_{2}^\prime \mathcal{Z}_{2}+\frac{1}{2}\mathcal{Z}_{3}^T\mathcal{N}_{3}^\prime\mathcal{Z}_{3},
\end{align}
where
\begin{gather*}
\mathcal{N}_{1}^\prime=\begin{bmatrix}\frac{k_x}{2}&mc_1&0\\mc_1&\frac{m}{2}&0\\0&0&\frac{1}{\min\{\gamma_{w_1},\gamma_{v_1}\}}\end{bmatrix},\nonumber\\
\mathcal{N}_{2}^\prime=\begin{bmatrix}\frac{k_R}{2-\psi_2}&c_2\lambda_{M_J}&0\\c_2\lambda_{M_J}&\lambda_{M_J}&0\\0&0&\frac{1}{\min\{\gamma_{w_2},\gamma_{v_2}\}}\end{bmatrix},\nonumber\\
\mathcal{N}_{3}^\prime=\begin{bmatrix}\frac{k_x}{2}&0&0\\0&\frac{m}{2}&0\\0&0&\frac{1}{2-\psi_2}\end{bmatrix},\label{eqn:N1prime}\\
\mathcal{Z}_{1}=[||e_x||,\norm{e_v},||\tilde Z_1||]^T,\quad  \mathcal{Z}_{2}=[\norm{e_R},||e_\Omega||,||\tilde Z_2||]^T,\\
\mathcal{Z}_{3}=[\norm{e_x},||e_v||,||e_R||]^T.
\end{gather*}
As discussed above, the matrices $\mathcal{N}_1^\prime,\mathcal{N}_2^\prime,\mathcal{N}_3^\prime$ are positive-definite if $c_1,c_2$ are sufficiently small.

The derivative of the $\Lya$ function is $\dot{\mathcal{V}}=\dot{\mathcal{V}}_1+\dot{\mathcal{V}}_2$.
From~\refeqn{V1dot} and~\refeqn{V2dot0}, it can be written as
\begin{align}
\dot{\mathcal{V}}\leq
&-\frac{1}{2} \mathcal{Z}_{1}^T \mathcal{N}_{1} \mathcal{Z}_{1}-\frac{1}{2} \mathcal{Z}_{2}^T \mathcal{N}_{2} \mathcal{Z}_{2}-\frac{1}{2}\mathcal{Z}_{3}^T\mathcal{N}_{3}\mathcal{Z}_{3} +C_{5},
\end{align}
where $C_5=C_{5_1}+C_{5_2}$, and
\begin{gather}
\mathcal{N}_{1}=\begin{bmatrix}\frac{c_1 k_{x_\beta}}{2}&-\frac{k_{xv}}{2}&-c_1C_{4_1}\\-\frac{k_{xv}}{2}&\frac{k_{v_\beta}}{2}&-C_{4_1}\\-c_1C_{4_1}&-C_{4_1}&\kappa_1\end{bmatrix},\label{eqn:N1}\\
\mathcal{N}_{2}=\begin{bmatrix}\frac{c_2 k_{R}}{2}&-k_{R\Omega}&-c_2C_{4_2}\\-k_{R\Omega}&k_{\Omega_\beta}&-C_{4_2}\\-c_2C_{4_2}&-C_{4_2}&\kappa_2\end{bmatrix},\label{eqn:N2}\\
\mathcal{N}_{3}=\begin{bmatrix}\frac{c_1 k_{x_\beta}}{2}&-\frac{k_{xv}}{2}&-c_1B_1\\-\frac{k_{xv}}{2}&\frac{c_1 k_{v_\beta}}{2}&-B_1-k_xe_{x_{max}}\\-c_1B_1&-B_1-k_xe_{x_{max}}&\frac{c_2 k_{R}}{2}\end{bmatrix}.\label{eqn:N3}
\end{gather}

We can show that choosing sufficiently large $k_x$, $k_v$, $k_R$, $k_\Omega$, $\gamma_{w_i},\gamma_{v_i}$, $\kappa_i$, and sufficiently small $c_i$, for $i \in\{1,2\}$, the matrices $\mathcal{N}_{1} ,\,\mathcal{N}_{2} ,\,\mathcal{N}_{3}$ become positive definite.
Consequently, there exists $\mu>0$ such that
\begin{align}
\dot{\mathcal{V}}\leq-\nu \mathcal{V}+C_{5},\label{eqn:Vdot_final}
\end{align}
If $\mathcal{V}>\frac{C_{5}}{\nu}$, then $\dot{\mathcal{V}}<0$. Therefore, according to~\cite{khalil1996noninear}, $e_x,e_v,e_R,e_\Omega,\tilde Z_1$ and $\tilde Z_2$ are bounded and converge exponentially to the set $\mathcal{D}$
\begin{align}
\mathcal{D}&=\{e_x,e_v,e_R,e_\Omega \in \Ree^3,\tilde Z_1 \in \Ree^{N_{1_1}+N_{2_1}+2\times N_{2_1}+N_{3_1}},\nonumber\\& \tilde Z_2 \in \Ree^{N_{1_2}+N_{2_2}+2\times N_{2_2}+N_{3_2}}\arrowvert \norm{e_x}^2+\norm{e_v}^2+\norm{e_R}^2\nonumber\\
           & +\norm{e_\Omega}^2
+\frac{1}{\gamma_1}\|\tilde Z_1\|^2+\frac{1}{\gamma_2}\|\tilde Z_2\|^2\leq \frac{C_{5}}{\nu}\},\label{eqn:D}
\end{align}
where $\gamma_1=\max\{\gamma_{v_1},\gamma_{w_1}\},\,\gamma_2=\max\{\gamma_{v_2},\gamma_{w_2}\}$. 

\section*{Acknowledgment}
The authors would like to especially thank Mr. Kanishke Gamagedara for his contribution to the drone hardware platform which is used for experimental validation, and its maintenance.

\bibliography{total}
\bibliographystyle{IEEEtran}

\end{document}